\documentclass[11pt]{article}

\usepackage[margin=1in]{geometry}
\usepackage{amsmath,amssymb,amsthm}
\usepackage{mathtools}
\usepackage{bm}
\usepackage[Symbolsmallscale]{upgreek}
\usepackage{enumitem}
\usepackage{microtype}
\usepackage{aliascnt}
\usepackage[backend=bibtex,style=alphabetic,sorting=nyt,giveninits=true,
  maxalphanames=3,maxbibnames=99,isbn=false]{biblatex}
\AtBeginBibliography{\small\setlength{\emergencystretch}{1em}}
\usepackage[hidelinks]{hyperref}
\usepackage[nameinlink,noabbrev]{cleveref}

\numberwithin{equation}{section}

\newtheorem{theorem}{Theorem}[section]
\newaliascnt{proposition}{theorem}
\newtheorem{proposition}[proposition]{Proposition}
\aliascntresetthe{proposition}
\newaliascnt{lemma}{theorem}
\newtheorem{lemma}[lemma]{Lemma}
\aliascntresetthe{lemma}
\newaliascnt{corollary}{theorem}
\newtheorem{corollary}[corollary]{Corollary}
\aliascntresetthe{corollary}
\newaliascnt{remark}{theorem}
\newtheorem{remark}[remark]{Remark}
\aliascntresetthe{remark}

\crefname{proposition}{proposition}{propositions}
\Crefname{proposition}{Proposition}{Propositions}
\crefname{lemma}{lemma}{lemmas}
\Crefname{lemma}{Lemma}{Lemmas}
\crefname{corollary}{corollary}{corollaries}
\Crefname{corollary}{Corollary}{Corollaries}
\crefname{remark}{remark}{remarks}
\Crefname{remark}{Remark}{Remarks}

\newcommand{\R}{\mathbb{R}}
\newcommand{\E}{\mathbb{E}}
\newcommand{\cN}{\mathsf{N}}
\newcommand{\Law}{\operatorname{Law}}
\newcommand{\W}{W}
\newcommand{\Id}{I}
\newcommand{\ip}[2]{\langle #1,#2\rangle}
\newcommand{\norm}[1]{\lVert #1\rVert}
\newcommand{\mmid}{\mathrel{\Vert}}
\newcommand{\e}{\mathrm{e}}
\newcommand{\deq}{\coloneqq}
\DeclareMathOperator*{\argmin}{arg\,min}

\title{Sharp dimensional analysis of midpoint methods for Langevin sampling}
\author{
 Fan Chen\thanks{Department of Electrical Engineering and Computer Science,
 Massachusetts Institute of Technology.
 Email: \href{mailto:fanchen@mit.edu}{\texttt{fanchen@mit.edu}}.}
 \and
 Sinho Chewi\thanks{Department of Statistics and Data Science, Yale University.
 Email: \href{mailto:sinho.chewi@yale.edu}{\texttt{sinho.chewi@yale.edu}}.}
 \and
 Jianfeng Lu\thanks{Department of Mathematics, Duke University.
 Email: \href{mailto:jianfeng@math.duke.edu}{\texttt{jianfeng@math.duke.edu}}.}
 \and
 Matthew S. Zhang\thanks{Department of Mathematics, Massachusetts Institute of Technology.
 Email: \href{mailto:shuns436@mit.edu}{\texttt{shuns436@mit.edu}}.}
}
\date{September 2026}

\begin{document}
\maketitle

\begin{abstract}
We study deterministic and randomized midpoint discretizations of Langevin
 dynamics for a target \(\pi\propto\e^{-V}\), where
\(0 \prec \alpha I\preceq\nabla^2V\preceq\beta I\) and \(\kappa\deq\beta/\alpha\).
To achieve \(\sqrt\alpha\,W_2\leqslant\varepsilon\), we show that deterministic Heun uses at most
\(\widetilde O(\kappa^{4/3}d^{1/3}\varepsilon^{-2/3})\) gradient
queries, and underdamped exponential midpoint uses
\(\widetilde O(\kappa^{5/4}d^{1/4}\varepsilon^{-1/2})\).
The proofs exploit cancellation at stationarity and smoothing using techniques from Malliavin calculus, outperforming previous upper bounds based on standard couplings.
At bounded condition number, a lower bound matches the $d$ and $\varepsilon$ powers of both
deterministic methods.
To contrast, for the randomized midpoint methods and Poisson midpoint with
at least two grid points (both overdamped and underdamped variants), a simple Gaussian calculation yields a lower bound \(d^{1/3}\varepsilon^{-1/3}\) to get an $\varepsilon$-close sample despite starting at a benign initialization. This shows surprisingly that in high dimensions,
deterministic discretizations can outperform their random counterparts.
\end{abstract}

\tableofcontents
\clearpage

\section{Introduction}
\label{sec:introduction}

We study sampling from \(\pi(\mathrm d x)\propto\e^{-V(x)}\,\mathrm d x\)
on \(\R^d\), where \(V\in C^2(\R^d)\) satisfies
\begin{equation}
    0 \prec \alpha\Id_d\preceq\nabla^2V(x)\preceq\beta\Id_d\,,
    \qquad x\in\R^d\,.
    \label{eq:sc-smooth}
\end{equation}
Write \(\kappa\deq\beta/\alpha\). No bounds on third or
higher derivatives of \(V\) are assumed.  We count gradient evaluations needed to
produce a distribution \(\mu\) with
\begin{equation}
    \sqrt\alpha\,\W_2(\mu,\pi)\leqslant\varepsilon\,,
    \label{eq:global-absolute-criterion}
\end{equation}
where \(\W_2\) is the \(2\)-Wasserstein distance.

A basic approach is to discretize overdamped or underdamped Langevin
dynamics. Euler methods evaluate the gradient at the initial point of each step;
midpoint methods correct this Euler prediction with additional gradient
evaluations to improve accuracy. In this work we compare deterministic quadrature\footnote{Deterministic refers only to
quadrature times; the Brownian increments remain random.}---Heun and
underdamped exponential midpoint---with both randomized
midpoint (RLMC and RULMC)
\cite{ShenLee2019, HeBalasubramanianErdogdu2020, YuKaragulyanDalalyan2024,AltschulerChewiZhang2026}
and Poisson midpoint (PLMC and PULMC) \cite{KandasamyNagaraj2024, SrinivasanNagaraj2026}.

At bounded condition number, prior analyses suggested the following hierarchy of performance.
The best known complexity of overdamped Euler, known as Langevin Monte Carlo (LMC), in our setting was for a long time
\(\widetilde O(d\varepsilon^{-2})\) \cite{Dalalyan2017}, and deterministic
midpoint quadrature seemed to offer no improvement under
\eqref{eq:sc-smooth}. Randomization improved the guarantee to
\(\widetilde O(d^{1/2}\varepsilon^{-1})\)
\cite{HeBalasubramanianErdogdu2020}, while introducing momentum (i.e., discretizing underdamped Langevin dynamics) improved both:
to \(\widetilde O(d^{1/2}\varepsilon^{-1})\) for deterministic exponential
Euler \cite{ChengEtAl2018}, and
\(\widetilde O(d^{1/3}\varepsilon^{-2/3})\) for randomized midpoint
\cite{ShenLee2019}. This gave rise to the intuition that randomization and
momentum each bring an advantage, reducing the exponent of \(d\varepsilon^{-2}\) from \(1\) to \(1/2\),
and then to \(1/3\) when combined.

Two recent results challenged this picture. Pedrotti and
Whalley~\cite{PedrottiWhalley2026} show that LMC already matches the classical randomized overdamped guarantee. The recently proposed Poisson midpoint method~\cite{SrinivasanNagaraj2026} obtained accuracy rates $\varepsilon^{-1/3}$ exceeding the conjectured exponent above.
These results reveal cancellations missed by conventional bounds and reopen the question of comparison between deterministic and
randomized quadrature.
In response to this, we investigate the convergence properties of several classes of midpoint methods in the present work.

\subsection{Our results and their implications}
\label{subsec:intro-results}

\noindent\textbf{Key takeaways.} At bounded condition number, our results
give three conclusions.
\begin{enumerate}[label=(\roman*),leftmargin=*]
\item \textbf{Improved complexity estimates for midpoint schemes.}
Heun achieves \(\widetilde O(d^{1/3}\varepsilon^{-2/3})\), and
underdamped exponential midpoint achieves
\(\widetilde O(d^{1/4}\varepsilon^{-1/2})\), using two gradient
evaluations per step. This is incredibly surprising --- until the recent work of~\cite{ChenEtAl2026Exact}, the $d^{1/4}$ rate would have been state-of-the-art. The underdamped rate also completely bypasses the complexity
required for path approximation in~\cite{CaoLuWang2021}, and shows that the former does not
constrain distributional approximation.

\item \textbf{Deterministic can beat randomized.} This is perhaps the most shocking result in the context of current literature. The \(d^{1/3}\) and \(d^{1/4}\) deterministic rates both improve on the classical randomized midpoint
guarantees; at fixed accuracy, the underdamped deterministic rate even beats
the \(d^{1/3}\) lower bound for every randomized method studied here.

\item \textbf{Lower bounds make the separations sharp.}
We give lower bounds for every method in \Cref{tab:intro-comparison}.
They match the deterministic dimension and accuracy powers, establish a
common $d^{1/3}$ obstruction for all randomized discretizations, and match PULMC's leading rate in its
regime of sufficiently small \(\varepsilon\). These are concrete limits of the
methods, rather than just differences between available analyses; they are stated more precisely below.
\end{enumerate}

\begin{table}[t]
\centering
\small
\renewcommand{\arraystretch}{1.25}
\begin{tabular}{@{}lll@{}}
\hline
Method & Upper bound & Lower bound\\
\hline
Overdamped Euler (LMC) & \(d^{1/2}\varepsilon^{-1}\) \cite{PedrottiWhalley2026}
 & \(d^{1/2}\varepsilon^{-1}\) \cite[Example~2]{Wibisono2018}\\
Heun & \(\bm{d^{1/3}\varepsilon^{-2/3}}\)
 & \(\bm{d^{1/3}\varepsilon^{-2/3}}\)\\
RLMC & \(d^{1/2}\varepsilon^{-1}\) \cite{HeBalasubramanianErdogdu2020}
 & \(\bm{d^{1/3}\varepsilon^{-1/3}}\)\\
PLMC & \(d^{2/3}\varepsilon^{-2/3}\) \cite{SrinivasanNagaraj2026}
 & \(\bm{d^{1/3}\varepsilon^{-1/3}}\)\\
\hline
Underdamped exponential midpoint
 & \(\bm{d^{1/4}\varepsilon^{-1/2}}\)
 & \(\bm{d^{1/4}\varepsilon^{-1/2}}\)\\
RULMC & \(d^{1/3}\varepsilon^{-2/3}\) \cite{ShenLee2019}
 & \(\bm{d^{1/3}\varepsilon^{-1/3}}\)\\
PULMC (small accuracy regime) & \(d^{1/3}\varepsilon^{-1/3}\) \cite{SrinivasanNagaraj2026}
 & \(\bm{d^{1/3}\varepsilon^{-1/3}}\)\\
\hline
\end{tabular}
\caption{Gradient query bounds for \(\sqrt\alpha\,\W_2\leqslant\varepsilon\)
at bounded condition number, omitting logarithms. Bold entries are new
results. ``Lower bound'' means a necessary number of iterations for some
admissible target and an initial distribution with expected energy gap
\(O(d)\), as specified in \Cref{thm:intro-lower}.
Poisson methods use at least two grid points and have constant expected
gradient cost per batch. The cited upper bounds have their own
initialization hypotheses; PULMC's general \(p\) bound is discussed below.}
\label{tab:intro-comparison}
\end{table}

The following summarizes our upper bounds; see \Cref{thm:et-upper,thm:uld-midpoint-main}.

\begin{theorem}[Deterministic midpoint guarantees, informal]
\label{thm:intro-midpoint}
Assume \eqref{eq:sc-smooth} and \(0<\varepsilon\leqslant\sqrt d/2\). Then, Heun initialized at the minimizer
\(x_\star\) attains \eqref{eq:global-absolute-criterion} with
\[
    {\widetilde O\bigl(\kappa^{4/3}d^{1/3}\varepsilon^{-2/3}\bigr)}
\]
gradient queries. Underdamped exponential midpoint with friction
\(2\sqrt\beta\), initialized with \(X_0=x_\star\) and independent
\(P_0\sim\cN(0,\Id_d)\), attains the same position accuracy with
\[
    {\widetilde O\bigl(\kappa^{5/4}d^{1/4}\varepsilon^{-1/2}\bigr)}
\]
gradient queries. Here \(\widetilde O\) suppresses logarithmic factors.
\end{theorem}

\begin{theorem}[Stationary bias and iteration lower bounds, informal]
\label{thm:intro-lower}
For each method in \Cref{tab:intro-comparison},
and every sufficiently small normalized step \(h\), there exist a
strongly convex, smooth target with bounded condition number and an
initial distribution satisfying
\[
 \E[V(X_0)-V(x_\star)]\leqslant Cd\,,
 \qquad x_\star\deq\argmin V\,,
\]
such that
for every sufficiently small \(\varepsilon>0\), achieving
\(\sqrt\alpha\,\W_2(\nu_N,\pi)\leqslant\varepsilon\) requires at least
the query complexity in~\Cref{tab:intro-comparison}'s lower bound column, with logarithmic
factors suppressed. Here \(\nu_N\) is the position distribution after
\(N\) steps. In the underdamped case, this initial distribution is on $X_0$ (with \(P_0\sim\cN(0,\Id_d)\) independent of
\(X_0\)).

The stationary bounds are stated in
\Cref{prop:heun-oscillatory-lower,prop:exponential-midpoint-Gaussian-lower,prop:overdamped-common-construction,prop:randomized-underdamped-Gaussian-lower}.
The finite iteration bounds and their initialization conditions are in
\Cref{cor:midpoint-iteration-lower};
\Cref{thm:et-linear-kappa-lower} additionally quantifies the condition
number dependence for Heun.
\end{theorem}

\paragraph{Deterministic upper bounds.}
The proofs uncover higher-order, distributional cancellations at stationarity which had been missed in previous coupling analyses. A discrete Poisson
corrector \cite{PedrottiWhalley2026} combined with Malliavin--Stein and
Gaussian smoothing estimates \cite{PaulinWhalley2026} are the main techniques for analyzing overdamped Heun.
For exponential midpoint, weighted divergence estimates in
\(\dot H^{-1}\) reveal cancellation in the conditional mean, and a
transport interpolation controls the
full local error, extending \cite{HigherOrderRHMC} to a non-Gaussian
base measure. See
\Cref{app:overdamped-proof,app:underdamped-proof}.

\paragraph{Randomized algorithm bounds.}
Substituting \(\epsilon=\varepsilon/\sqrt d\) in the tolerance
\(\epsilon\sqrt{d/\alpha}\) of
\cite[Corollaries~1--2]{SrinivasanNagaraj2026} gives PLMC's
\(\widetilde O(d^{2/3}\varepsilon^{-2/3})\) upper bound and, for each fixed
integer \(p\geqslant3\), the PULMC upper bound
\[
    \widetilde O_p\bigl(
       d^{1/3}\varepsilon^{-1/3}
       +d^{\frac{3p+2}{8p+6}}
          \varepsilon^{-\frac{p+2}{4p+3}}\bigr)\,.
\]
For any fixed \(p>3\), sufficiently small accuracy
\(\varepsilon\leqslant d^{-p/(2\,(p-3))}\) makes the first term dominate,
up to logarithms. This \(d^{1/3}\varepsilon^{-1/3}\) scale matches our
lower bound. The PLMC bound has no such parameter, and does not match the lower bound.
With respect to upper bounds, PULMC retains an
accuracy advantage over deterministic exponential midpoint, while the
latter wins in dimension: neither upper bound dominates for all
\((d,\varepsilon)\), whereas both RLMC and RULMC have uncompetitive upper bounds. We leave proofs of improved upper bounds, where they are currently loose, as a question for future work.

\paragraph{Lower bounds and sharpness.}
The deterministic lower bounds match the dimension and accuracy powers
of \Cref{thm:intro-midpoint} at bounded condition number. Heun requires
an oscillatory target, whereas a Gaussian target already gives the
matching exponential midpoint bound. An anisotropic Gaussian extends
the deterministic exponential midpoint lower bound to the case where
\(\kappa\gg1\); see \Cref{prop:exponential-midpoint-Gaussian-lower}.
For Heun, coupling the oscillatory
coordinates to a weakly confined coordinate also gives
\[
    N\gtrsim\kappa d^{1/3}\varepsilon^{-2/3}\log(1/\varepsilon)\,,
    \qquad d^{-1}\leqslant\varepsilon\leqslant\varepsilon_0\,,
\]
under the initialization condition in \Cref{thm:et-linear-kappa-lower}.
The condition number dependence is not fully matched by the corresponding upper bound.

For the randomized methods, isotropic Gaussian targets reveal a different
obstruction. Conditional on past quadrature choices, all coordinates have
the same random variance. Its fluctuations make their squared magnitudes
vary together, creating an error in the distribution of \(\norm X\) that
cannot be detected by inspecting one coordinate alone. This shared
fluctuation explains the common dimension obstruction for overdamped and
underdamped randomized methods; see \Cref{sec:randomized-lower}.

Finally, other approaches can provide better complexity guarantees in
this setting, including Metropolization
\cite{ChewiEtAl2021MALA,AltschulerChewi2024}, the proximal sampler
\cite{LeeShenTian2021,ChenChewiSalimWibisono2022}, first-order rejection
sampling (FORS) \cite{ChenEtAl2026FORS,ChenEtAl2026Exact}, and
discretizations of Hamiltonian dynamics \cite{HigherOrderRHMC}.
We limit the present study to the midpoint discretizations described above.

\subsection{Related work}
\label{subsec:related-work}

\paragraph{Euler discretizations.}
Early non-asymptotic analyses of overdamped Euler (LMC) appeared in
\cite{DalalyanTsybakov2012}. \(\W_2\) guarantees were developed in
\cite{Dalalyan2017,DurmusMoulines2019} and sharpened in
\cite{DalalyanKaragulyan2019}.
At bounded condition number, the classical guarantee under
\eqref{eq:sc-smooth} was \(\widetilde O(d\varepsilon^{-2})\).
The analysis extended to KL divergence in
\cite{ChengBartlett2018,DurmusMajewskiMiasojedow2019}, and to targets
satisfying a log-Sobolev inequality in \cite{VempalaWibisono2019}.
R\'enyi guarantees followed in
\cite{GaneshTalwar2020,ErdogduHosseinzadehZhang2022}, with extensions
under a Poincar\'e inequality and intermediate functional inequalities
in \cite{ChewiEtAl2025LMC}.
Underdamped exponential Euler followed a similar trajectory: \(\W_2\)
analysis in \cite{ChengEtAl2018}, refined in
\cite{DalalyanRiouDurand2020}; KL analysis in \cite{MaEtAl2021}; and
R\'enyi guarantees in \cite{GaneshTalwar2020,AltschulerChewi2024}, with
weaker smoothness and functional inequality assumptions treated in
\cite{ZhangEtAl2023}.
Its classical Wasserstein guarantee at bounded condition
number is \(\widetilde O(d^{1/2}\varepsilon^{-1})\), improving on the
older overdamped rate.
See \cite[Chapters~4--5]{ChewiBook} for a survey.
The recent work of Pedrotti and
Whalley~\cite{PedrottiWhalley2026} revises the overdamped baseline significantly: they show that
LMC attains \(\widetilde O(\kappa\sqrt d\,\varepsilon^{-1})\) under
\eqref{eq:sc-smooth}, without changing the algorithm. In their work, a discrete Poisson
corrector removes the leading conditional discretization error,
revealing cancellations missed by standard pathwise estimates.

\paragraph{Deterministic midpoint methods.}
Just as with ODE integrators, more accurate quadrature provides a natural improvement to Euler. Classical weak and invariant distribution error expansions
explain its advantage under additional smoothness
\cite{TalayTubaro1990,MattinglyStuartTretyakov2010}.
Runge--Kutta and midpoint splitting schemes were analyzed in
\cite{LiEtAl2019,SanzSernaZygalakis2021,PaulinWhalley2024}.
Anticipating Girsanov theory \cite{Buckdahn1994,Nualart2006} provides
another route to KL and R\'enyi analysis of midpoint interpolations, as
developed in \cite{Zhang2025Midpoint}. Under additional Hessian
regularity, that work obtains a \(d^{1/4}\varepsilon^{-1/2}\) guarantee
for a deterministic underdamped double midpoint variant.
Under only \eqref{eq:sc-smooth}, however, the standard estimates did not
establish an advantage for deterministic midpoint rules over Euler.

\paragraph{Randomized midpoint.}
Shen and Lee~\cite{ShenLee2019} evaluate the predicted path at a uniformly
random time. Conditional on that path, this gives an unbiased estimate of
its force integral. Separating bias from centered fluctuations yields the
classical guarantees \(\widetilde O(d^{1/2}\varepsilon^{-1})\) for RLMC
and \(\widetilde O(d^{1/3}\varepsilon^{-2/3})\) for RULMC
\cite{ShenLee2019,HeBalasubramanianErdogdu2020,YuKaragulyanDalalyan2024}.
These improvements suggested that randomization was essential without
higher derivative bounds. Cao, Lu, and Wang~\cite{CaoLuWang2021}
reinforced this intuition by proving optimality for \emph{strong
approximation}: no algorithm in their oracle model improves the worst
case error \(\sqrt d\,N^{-3/2}\) in approximating a prescribed diffusion
path. This leaves open whether the same rate limits sampling from the
target distribution.
Further work on these algorithms studied invariant bias and asymptotic behavior
\cite{HeBalasubramanianErdogdu2020}, sharpened Wasserstein bounds
\cite{YuKaragulyanDalalyan2024}, and extended the analysis to KL divergence
through shifted composition
\cite{AltschulerChewi2026,AltschulerChewiZhang2026}.

\paragraph{Poisson midpoint.}
Kandasamy and Nagaraj~\cite{KandasamyNagaraj2024} use a different
randomization: on a fine grid, each gradient correction is independently
retained with probability equal to the reciprocal of the grid size.
Only a constant number of gradients is queried in expectation, and their
analysis controls distributional error through Gaussian perturbations.
Srinivasan and Nagaraj~\cite{SrinivasanNagaraj2026} then exploit higher
order centering and Gaussian smoothing to obtain accuracy dependence
\(\varepsilon^{-2/3}\) for PLMC and \(\varepsilon^{-1/3}\) for PULMC.
The latter goes beyond the apparent strong approximation barrier, though
the dimension dependence is more intricate; the full bounds and their
moment parameter \(p\) are discussed in \Cref{subsec:intro-results}.
Together with the revised Euler analysis, these results motivate a
comparison based directly on sampling error rather than pathwise error.
\par

\paragraph{Related settings and discretizations.}
In diffusion models, \cite{ZhangEtAl2025DDPM} analyzes randomized midpoint
for the reverse-time SDE and obtains sublinear dimension dependence for
DDPM sampling.
Related midpoint and leapfrog splittings also appear in HMC
\cite{BouRabeeMarsden2025}, including the construction of
algorithmic warm starts in \cite{ZhangAltschulerChewi2026}.
Random ODE approximations provide a further family of methods, including
the shifted ODE method and its numerical implementations
\cite{FosterLyonsOberhauser2021,ScottEtAl2026}.

\paragraph{AI usage.}
This research project was initially motivated by the desire to understand the optimal complexity bounds for (Poisson) randomized midpoint discretizations.
GPT-5.6 Sol produced both upper and lower bounds for these methods, which did not fully match; the lower bounds are included in this work. In an effort to produce algorithms matching the lower bounds, we gradually realized that deterministic midpoint methods were better candidates, and that their complexity bounds surprisingly surpassed the randomized midpoint lower bounds.
At this point, we switched the focus of the project to pinning down the optimal rates for deterministic midpoint methods.
GPT-5.6 Sol proved the main theorems in this paper and organized the initial draft. We then edited and simplified the proofs and prepared the manuscript. We take full responsibility for its contents.

\section{Preliminaries}
\label{sec:preliminaries}
\label{subsec:intro-setup}

\subsection{Notation and normalization}
\label{sec:normalization}

We assume \eqref{eq:sc-smooth} throughout and write \(x_\star\) for the
unique minimizer of \(V\).  For probability measures with finite second
moments,
\[
    \W_2^2(\mu,\nu)
    \deq\inf_{\Gamma\in\Pi(\mu,\nu)}
       \int\norm{x-y}^2\,\Gamma(\mathrm d x,\mathrm d y)\,,
\]
where \(\Pi(\mu,\nu)\) is the set of their couplings.  We denote this space
of measures by \(\mathcal P_2\). We write
\(\operatorname{KL}(\mu\mmid\nu)\deq\int\log(\mathrm d\mu/\mathrm d\nu)\,\mathrm d\mu\)
when \(\mu\ll\nu\), and \(+\infty\) otherwise. For a Markov kernel \(K\), \(\mu K\)
is the distribution after applying one update to a point drawn from \(\mu\).
For an integrable function \(f\), write
\(\mu(f)\deq\int f\,\mathrm d \mu\).
Vector norms are Euclidean; \(\norm{\cdot}_p\) denotes a probabilistic
\(L^p\) norm when the underlying distribution is specified.  Matrix norms
carry an operator, Hilbert--Schmidt, or Schatten subscript.
We use \(A\lesssim B\) for \(A\leqslant CB\) with a universal constant,
and \(\gtrsim\) and \(\asymp\) analogously.  The constants \(c,C>0\)
may change between displays.  The notation \(\widetilde O\) hides fixed
powers of logarithms in \(d,\kappa,\varepsilon^{-1}\).

The algorithms and main theorems are stated in the original units.
For the upper bound proofs we normalize \(\beta=1\), so that
\(\alpha=\kappa^{-1}\).  Indeed, replacing \(x\) by \(\sqrt\beta\,x\)
and time by \(\beta t\) for overdamped dynamics, or by
\(\sqrt\beta\,t\) for underdamped dynamics, preserves the form of the
algorithms.  Momentum is unchanged in the latter rescaling.  The normalized
steps are therefore \(\beta h\) and \(\sqrt\beta\,h\), respectively.
The proofs use \(h\) for the normalized step; the scaled distance
\(\sqrt\alpha\,\W_2\) is unchanged.  This convention applies to the
proof outlines and upper bound appendices unless original units are
explicitly stated.  We fix a sufficiently small universal
\(h_0\in(0,1/3]\) for short-time estimates; stronger restrictions are
stated where needed.

\subsection{Overdamped dynamics and algorithms}
\label{sec:variants}

The overdamped Langevin diffusion is
\begin{equation}
    \mathrm d X_t=-\nabla V(X_t)\,\mathrm d t
       +\sqrt2\,\mathrm d B_t\,.
    \label{eq:langevin}
\end{equation}
Here, \(B\) is standard Brownian motion in \(\R^d\).  Its generator is
\(L\deq\Delta-\nabla V\cdot\nabla\), and \(P_h\deq\e^{hL}\) denotes its
transition kernel over time \(h\).  The target \(\pi\) is reversible for
\(P_h\).

Given the current position \(x\), first predict the path by keeping the
gradient equal to \(\nabla V(x)\):
\begin{equation*}
    X_t^{\mathrm{pred}}
       \deq x-t\,\nabla V(x)+\sqrt2\,B_t\,,
    \qquad 0\leqslant t\leqslant h\,.
\end{equation*}
All methods below reuse this predictor and the same Brownian path within a
step.  Across steps, Brownian paths and all additional random variables are
independent.

\paragraph{Deterministic Heun.}
The Heun method~\cite{HighamKloeden2021} predicts the endpoint,
\begin{equation}
    X_h^{\mathrm{pred}}\deq x-h\,\nabla V(x)+\sqrt2\,B_h\,,
    \label{eq:et-predictor}
\end{equation}
and averages the initial and predicted gradients:
\begin{equation}
    X^+\deq x-\frac h2\,\{\nabla V(x)+\nabla V(X_h^{\mathrm{pred}})\}
       +\sqrt2\,B_h\,.
    \label{eq:et-step}
\end{equation}
We denote its kernel by \(K_h^{\mathrm{Heun}}\).  It uses two gradient
queries per step. Here, deterministic refers to the quadrature times;
the Brownian increment remains random.

\paragraph{Randomized midpoint (RLMC).}
Draw \(\Theta\sim\mathsf{Unif}[0,1]\), independently of \(B\), and set
\begin{equation}
    X^+\deq x-h\,\nabla V(X_{\Theta h}^{\mathrm{pred}})+\sqrt2\,B_h\,.
    \label{eq:rm-step}
\end{equation}
This is the randomized midpoint update of~\cite{ShenLee2019, HeBalasubramanianErdogdu2020} (see also
\cite[Section~5.1]{ChewiBook}); its kernel is \(K_h^{\mathrm{RLMC}}\).
It uses one gradient at \(x\) to form the predictor and one at the predicted
point.  The scalar \(\Theta\) is common to all coordinates.

\paragraph{Poisson midpoint (PLMC).}
Fix \(k\geqslant1\), put \(s_i\deq i/k\), and draw independent indicators
\(\mathsf B_i\sim\mathsf{Bernoulli}(1/k)\) for
\(i\in\{1,\ldots,k-1\}\), independently of \(B\).  One step of length
\(h\), also called a batch, is
\begin{equation}
    X^+\deq x-h\,\nabla V(x)
       +h\sum_{i=1}^{k-1}\mathsf B_i\,
          \{\nabla V(x)-\nabla V(X_{s_i h}^{\mathrm{pred}})\}
       +\sqrt2\,B_h\,.
    \label{eq:poisson-step}
\end{equation}
The
correction at time zero would vanish and is omitted.  Reusing the initial
gradient, the expected query count is \(1+(k-1)/k=2-1/k\).
For \(k=1\), the sum is empty and the update is the Euler-Maruyama scheme, known as unadjusted Langevin
algorithm (ULA).  We write \(K_h^{\mathrm{PLMC},k}\) for this kernel
\cite{KandasamyNagaraj2024,SrinivasanNagaraj2026}.

\subsection{Underdamped dynamics and algorithms}
\label{sec:underdamped-algorithms}

The underdamped Langevin diffusion evolves a position and a momentum:
\begin{align*}
    \mathrm d X_t&=P_t\,\mathrm d t\,,\\
    \mathrm d P_t&=-\nabla V(X_t)\,\mathrm d t
       -\gamma P_t\,\mathrm d t+\sqrt{2\gamma}\,\mathrm d B_t\,.
\end{align*}
Its invariant distribution is
\begin{equation*}
    \bm\pi(\mathrm d x,\mathrm d p)
       \deq\pi(\mathrm d x)\,\cN(0,\Id_d)(\mathrm d p)\,.
\end{equation*}
For a distribution \(\mu\) on phase space, \(\mu_x\) denotes the position
marginal; the required accuracy is
\(\sqrt\alpha\,\W_2(\mu_x,\pi)\leqslant\varepsilon\).

\begin{remark}[Friction]
\label{rem:friction-choice}
We choose \(\gamma=2\sqrt\beta\), following the synchronous coupling
analysis of underdamped Langevin in
\cite[Theorem~5.3.6]{ChewiBook}.  This choice permits contraction in a
fixed norm combining position and momentum.  It becomes \(\gamma=2\)
under the normalization in \Cref{sec:normalization}.  Our results do not
address discretizations with smaller friction.
\end{remark}

Integrating the friction term exactly gives
\begin{align*}
    X_t&=x+a_t p-\int_0^t a_{t-s}\nabla V(X_s)\,\mathrm d s
       +\sqrt{2\gamma}\int_0^t a_{t-s}\,\mathrm d B_s\,,\\
    P_t&=\e^{-\gamma t}p
       -\int_0^t\e^{-\gamma(t-s)}\nabla V(X_s)\,\mathrm d s
       +\sqrt{2\gamma}\int_0^t\e^{-\gamma(t-s)}\,\mathrm d B_s\,,
\end{align*}
where
\[
    a_t\deq\frac{1-\e^{-\gamma t}}\gamma\,,
    \qquad
    c_t\deq\int_0^t a_s\,\mathrm d s=\frac{t-a_t}\gamma\,.
\]
Evaluating the gradient at the initial point \(x\) in the position formula
gives the predictor
\begin{equation}
    X_t^{\mathrm{pred}}
       \deq x+a_t p-c_t\nabla V(x)
          +\sqrt{2\gamma}\int_0^t a_{t-s}\,\mathrm d B_s\,,
    \qquad 0\leqslant t\leqslant h\,.
    \label{eq:uld-predictor}
\end{equation}
This is the position component of the usual ULMC update, evaluated at an
intermediate time; see \cite[Section~5.3.2]{ChewiBook}.
All stochastic integrals within a step use the same Brownian path.
Conditional on the evaluation times they are jointly Gaussian, so the
algorithms require only finitely many Gaussian draws with the appropriate
covariance, rather than an entire Brownian path.

\paragraph{Exponential midpoint (ExpMid).}
Predict the position at time \(h/2\) using \eqref{eq:uld-predictor}, and
approximate each gradient integral using the gradient evaluated at this
predicted midpoint:
\begin{align}
    \widehat X_h^{\mathrm{ExpMid}}
    &\deq x+a_hp-h a_{h/2}\nabla V(X_{h/2}^{\mathrm{pred}})
       +\sqrt{2\gamma}\int_0^h a_{h-s}\,\mathrm d B_s\,,
    \label{eq:EM-x}\\
    \widehat P_h^{\mathrm{ExpMid}}
    &\deq\e^{-\gamma h}p-h\e^{-\gamma h/2}\nabla V(X_{h/2}^{\mathrm{pred}})
       +\sqrt{2\gamma}\int_0^h\e^{-\gamma(h-s)}\,\mathrm d B_s\,.
    \label{eq:EM-v}
\end{align}
These formulas define the exponential midpoint method studied here, with
kernel \(K_h^{\mathrm{ExpMid}}\).  Each step uses two gradient queries.
We write \((\widehat X_n,\widehat P_n)\) for its iterates.

\paragraph{Randomized midpoint (RULMC).}
The usual underdamped randomized midpoint method draws a common
\(\Theta\sim\mathsf{Unif}[0,1]\), independently of \(B\), and sets
\begin{align*}
    \widehat X_h^{\mathrm{RULMC}}
    &\deq x+a_hp-h a_{(1-\Theta)h}\nabla V(X_{\Theta h}^{\mathrm{pred}})
    +\sqrt{2\gamma}\int_0^h a_{h-s}\,\mathrm d B_s\,,
    \\
    \widehat P_h^{\mathrm{RULMC}}
    &\deq\e^{-\gamma h}p-h\e^{-\gamma(1-\Theta)h}
       \nabla V(X_{\Theta h}^{\mathrm{pred}})
       +\sqrt{2\gamma}\int_0^h\e^{-\gamma(h-s)}\,\mathrm d B_s\,.
\end{align*}
This is Algorithm~1 of Shen and Lee~\cite{ShenLee2019}, expressed in the
underdamped convention above.  Its kernel is \(K_h^{\mathrm{RULMC}}\).

\paragraph{Poisson midpoint (PULMC).}
Fix \(k\geqslant1\), let \(\Delta\deq h/k\) and \(t_i\deq i\Delta\), and draw
independent \(\mathsf B_i\sim\mathsf{Bernoulli}(1/k)\) for
\(i\in\{1,\ldots,k-1\}\), independently of \(B\).
For \(i\in\{0,\ldots,k-1\}\), define
\begin{align*}
    \omega_i^x
    &\deq\int_{t_i}^{t_{i+1}}a_{h-s}\,\mathrm d s
       =c_{h-t_i}-c_{h-t_{i+1}}\,, \qquad
    \omega_i^p
    \deq\int_{t_i}^{t_{i+1}}\e^{-\gamma(h-s)}\,\mathrm d s
       =a_{h-t_i}-a_{h-t_{i+1}}\,.
\end{align*}
The update is
\begin{align}
    \widehat X_h^{\mathrm{PULMC},k}
    &\deq x+a_hp-c_h\nabla V(x)\notag\\
    &\quad-k\sum_{i=1}^{k-1}\mathsf B_i\omega_i^x\,
          \{\nabla V(X_{t_i}^{\mathrm{pred}})-\nabla V(x)\}
       +\sqrt{2\gamma}\int_0^h a_{h-s}\,\mathrm d B_s\,,
    \label{eq:uld-PULMC-x}\\
    \widehat P_h^{\mathrm{PULMC},k}
    &\deq\e^{-\gamma h}p-a_h\nabla V(x)\notag\\
    &\quad-k\sum_{i=1}^{k-1}\mathsf B_i\omega_i^p\,
          \{\nabla V(X_{t_i}^{\mathrm{pred}})-\nabla V(x)\}
       +\sqrt{2\gamma}\int_0^h\e^{-\gamma(h-s)}\,\mathrm d B_s\,.
    \label{eq:uld-PULMC-v}
\end{align}
We write \(K_h^{\mathrm{PULMC},k}\) for the kernel, using the finite grid
formulation of \cite{SrinivasanNagaraj2026}.  As in PLMC, the time zero
correction vanishes, and a batch has expected query count \(2-1/k\).
For \(k=1\) this is the ordinary ULMC update, with the gradient evaluated
at the initial point.

\section{Overdamped upper bound}
\label{sec:overdamped-upper}

The deterministic Heun method achieves a squared stationary error of order
\(h^3\) under strong convexity and a Lipschitz gradient, without a bound on
third derivatives.  The step threshold \(h_0\) and the normalization used in the proof
are fixed in \Cref{sec:normalization}.

\begin{theorem}[Heun upper bound]
\label{thm:et-upper}
There are universal constants \(c,C>0\) such that, if
\(0<\beta h\leqslant h_0\), then \(K_h^{\mathrm{Heun}}\) has a unique
invariant distribution \(\widehat\pi_h^{\mathrm{Heun}}\) and
\begin{equation}
    \alpha\W_2^2(\widehat\pi_h^{\mathrm{Heun}},\pi)
    \leqslant C\kappa d\,(\beta h)^3\,.
    \label{eq:et-stationary-upper}
\end{equation}
For every \(\mu\in\mathcal P_2(\R^d)\) and integer \(N\geqslant0\),
\begin{equation}
    \alpha\W_2^2(\mu (K_h^{\mathrm{Heun}})^N,\pi)
    \leqslant C\e^{-cN\beta h/\kappa}\alpha\W_2^2(\mu,\pi)
       +C\kappa d\,(\beta h)^3\,.
    \label{eq:et-finite-time-upper}
\end{equation}
In particular, for every
\(\varepsilon>0\), the number of gradient queries sufficient
for \eqref{eq:global-absolute-criterion} is at most
\[
    C\,\Bigl(\kappa+
       \frac{\kappa^{4/3}d^{1/3}}{\varepsilon^{2/3}}\Bigr)
       \log\Bigl(\e+\frac{\alpha\W_2^2(\mu,\pi)}{\varepsilon^2}\Bigr)\,.
\]
Initialization at the minimizer \(x_\star\) satisfies
\(\alpha W_2^2(\delta_{x_\star},\pi) \leqslant d\),
and hence, for \(0<\varepsilon\leqslant\sqrt d/2\), gives
\begin{equation}
    N_{\nabla V}\leqslant C\,{\frac{\kappa^{4/3}d^{1/3}}{\varepsilon^{2/3}}}
       \log\frac d{\varepsilon^2}\,.
    \label{eq:et-complexity}
\end{equation}
\end{theorem}

We explain the proof in the normalized units \(\beta=1\),
\(\alpha=\kappa^{-1}\) of \Cref{sec:normalization}.  The technical estimates are
proved in \Cref{app:overdamped-proof}.  The first ingredient is contraction
of the numerical kernel itself, which is relatively standard.

\begin{proposition}[Contractivity of the Heun kernel]
\label{prop:et-contractivity}
For \(0<h\leqslant1\) and all \(\mu,\nu\in\mathcal P_2(\R^d)\),
\begin{equation}
    \W_2(\mu K_h^{\mathrm{Heun}},\nu K_h^{\mathrm{Heun}})
    \leqslant\rho\W_2(\mu,\nu)\,,
    \qquad \rho\leqslant1-\frac h{2\kappa}\,.
    \label{eq:et-contraction-factor}
\end{equation}
In particular, \(K_h^{\mathrm{Heun}}\) has a unique invariant
 distribution in \(\mathcal P_2(\R^d)\).
\end{proposition}

Rescaling \Cref{prop:et-contractivity} to the original units gives
contraction factor \(\rho\leqslant1-\alpha h/2\) whenever
\(0<h\leqslant1/\beta\), with the same existence and uniqueness conclusion.

The usual Langevin displacement estimate also gives an elementary comparison,
which when combined with contraction would give only
\(\kappa^{-1}\W_2^2(\widehat\pi_h^{\mathrm{Heun}},\pi)
\lesssim\kappa dh\).  To recover two further powers of \(h\), we
instead seek a different route of analysis. First, separate the conditional mean of the error from the centered fluctuation.  This is the
organizing principle of local error analysis; see
\cite[Section~5.1.1 and Lemma~5.1.2]{ChewiBook} and
\cite[Lemma~4.4]{HigherOrderRHMC}.

To measure the error of one Heun step, we compare it with the
Langevin diffusion run for the same time \(h\) from the same starting
point and driven by the same Brownian motion.  Let
\((X_t)_{0\leqslant t\leqslant h}\) solve \eqref{eq:langevin} with
\(X_0\sim\pi\).  Since \(\pi\) is invariant, \(X_t\sim\pi\) for every
\(t\); we call this the \emph{stationary path}, and we call \(X_h\) its
\emph{endpoint}, i.e., the exact position after one step.  In integral
form,
\[
    X_h=X_0-\int_0^h\nabla V(X_t)\,\mathrm d t+\sqrt2\,B_h\,.
\]
The Heun step \eqref{eq:et-step} from \(X_0\) evaluates its second
gradient at the predictor \(X_h^{\mathrm{pred}}\), which is only an
approximation of \(X_h\).  To separate this approximation from the
quadrature error, we first analyze an idealized step that evaluates the
second gradient at the endpoint \(X_h\) itself:
\[
    \widetilde X_h
    \deq X_0-\frac h2\,\{\nabla V(X_0)+\nabla V(X_h)\}+\sqrt2\,B_h\,.
\]
This step cannot be implemented, because \(X_h\) is unknown, and we use
it only as an intermediate object in the proof; we call it the
\emph{auxiliary step}, and its transition kernel
\(\widetilde K_h\) is defined in \eqref{eq:et-exact-reference}.  It
differs from \(K_h^{\mathrm{Heun}}\) by only \(O(\sqrt{dh^5})\) in the
comparisons needed below.  Subtracting the two displays, the error of the
auxiliary step is
\[
    \widetilde X_h-X_h
    =\int_0^h\nabla V(X_t)\,\mathrm d t
       -\frac h2\,\{\nabla V(X_0)+\nabla V(X_h)\}\,,
\]
the error of the trapezoidal rule for the integral of \(\nabla V\) along
the path.  Its conditional mean given the endpoint is
\(e_h(X_h)\), where, by reversibility of \(\pi\),
\[
    e_h\deq\int_0^hP_t(\nabla V)\,\mathrm d t
       -\frac h2\,\{\nabla V+P_h(\nabla V)\}\,.
\]
Thus, the systematic error is the trapezoidal quadrature defect for the
semigroup.

This systematic error is the main obstacle.  Unlike the centered
fluctuation \(\widetilde X_h-X_h-e_h(X_h)\), which is discussed below, it
is a deterministic function of the current position.  The fresh noise of
the next step therefore does not damp it, and contraction alone would let
it accumulate over the relaxation time.  We instead cancel it in advance by
shifting the point at which the comparison is made.  Fix a vector field
\(\chi\colon\R^d\to\R^d\), to be chosen below, and consider two points:
\begin{itemize}
    \item the output of the auxiliary step started from the shifted point
    \(X_0+\chi(X_0)\) instead of \(X_0\), using the same Brownian
    motion; and
    \item the shifted endpoint \(X_h+\chi(X_h)\), which has the same
    distribution as the shifted starting point \(X_0+\chi(X_0)\), because
    \(X_0\) and \(X_h\) both have distribution \(\pi\).
\end{itemize}
If the first point is close to the second on average, then the auxiliary
step nearly preserves the distribution of \(X_0+\chi(X_0)\).  Up to terms
of order \(h\,\norm{\chi}_{L^2(\pi)}\), which come from evaluating
\(\nabla V\) at the shifted points rather than at \(X_0\) and \(X_h\),
the difference between the two points is
\[
    \chi(X_0)+(\widetilde X_h-X_h)-\chi(X_h)\,.
\]
Given the endpoint, reversibility of \(\pi\) gives
\(\E[\chi(X_0)\mid X_h]=P_h\chi(X_h)\), so the conditional mean of this
difference is \((P_h\chi+e_h-\chi)(X_h)\).  It vanishes exactly when
\(\chi\) solves the discrete Poisson equation
\[
    (\Id-P_h)\chi=e_h\,.
\]
Equivalently, \(\chi=\sum_{n=0}^\infty P_h^n e_h\) is the
total systematic error accumulated over all future steps, each step's
contribution \(e_h\) being transported by the exact dynamics.  Shifting
the starting point by \(\chi\) accounts for this accumulated bias at
once, so that the remaining conditional mean comes only from the terms of
order \(h\,\norm{\chi}_{L^2(\pi)}\) neglected above.  We call \(\chi\) the
\emph{corrector}.
Such Poisson equations have long been used to isolate
systematic discretization error
\cite{MattinglyStuartTretyakov2010,JohndrowMattingly2017,LuYeZhou2025}; a related discrete
corrector appears in the analysis of the unadjusted Langevin algorithm
in \cite{PedrottiWhalley2026}.
It motivates the reference distribution
\begin{equation}
    \nu\deq(\Id+\chi)_\#\pi\,.
    \label{eq:et-corrected-law}
\end{equation}
The spectral calculation in \Cref{subsec:et-corrector} gives
\(\norm\chi_{L^2(\pi)}\leqslant C\sqrt{dh^3}\), and hence
\(\W_2(\nu,\pi)\leqslant C\sqrt{dh^3}\).  At this shifted
reference distribution, the conditional mean of the error is only
\(O(\sqrt{dh^5})\). 

The centered fluctuation still has size \(O(\sqrt{dh^3})\).
However, as the next step supplies fresh Gaussian noise, this
fluctuation cost is only \(O(\sqrt{dh^5})\) in \(\W_2\).  The reason is
that a centered perturbation affects a Gaussian convolution to second
order; see \Cref{lem:et-gaussian-convolution}.  We use the convolution estimate of
\cite[Theorem~E.4 and Remark~E.6(ii)]{PaulinWhalley2026}{,
with the sharp constant proved directly in \Cref{lem:et-gaussian-convolution}};
\cite[Lemma~4.11]{HigherOrderRHMC} develops the same smoothing principle
for perturbations that can depend on the smoothing Gaussian.
The appendix verifies the bounds using Malliavin calculus, which is the main analytic ingredient.

The following elementary consequence of contractivity then accumulates
the two costs.

\begin{lemma}[Bias accumulation at a reference law]
\label{lem:et-local-framework}
Let \(0<h\leqslant1\), let \(\rho\) be as in
\eqref{eq:et-contraction-factor}, and let
\(\nu\in\mathcal P_2(\R^d)\).  Then
\begin{equation}
    \W_2(\widehat\pi_h^{\mathrm{Heun}},\nu)
    \leqslant\W_2(\nu K_h^{\mathrm{Heun}},\nu)
       +\frac{\W_2\bigl(\nu(K_h^{\mathrm{Heun}})^2,
          \nu K_h^{\mathrm{Heun}}\bigr)}{1-\rho}\,.
    \label{eq:et-local-framework}
\end{equation}
\end{lemma}

The proofs of
\Cref{prop:et-contractivity} and \Cref{lem:et-local-framework}
are in \Cref{subsec:et-auxiliary-proofs}. For the reference distribution
\eqref{eq:et-corrected-law}, \Cref{subsec:et-assembly} proves
\[
    \W_2(\nu K_h^{\mathrm{Heun}},\nu)\leqslant C\sqrt{dh^3}\,,
    \qquad
    \W_2\bigl(\nu(K_h^{\mathrm{Heun}})^2,\nu K_h^{\mathrm{Heun}}\bigr)
       \leqslant C\sqrt{dh^5}\,.
\]
The first bound pays for the full local error once.  The second pays for
its smaller effect after the next step and is multiplied by the
relaxation time \((1-\rho)^{-1}\lesssim\kappa/h\).
Together with \(\W_2(\nu,\pi)\leqslant C\sqrt{dh^3}\),
\Cref{lem:et-local-framework} proves the stationary bound.
\Cref{subsec:et-upper-completion} then gives the convergence and query
bounds in \Cref{thm:et-upper}.

\begin{remark}[The remaining dependence on the condition number]
\label{rem:et-condition-gap}
The bound \eqref{eq:et-complexity} contains
\(\kappa^{4/3}d^{1/3}\varepsilon^{-2/3}\).  The lower bound in
\Cref{thm:et-linear-kappa-lower} contains
\(\kappa d^{1/3}\varepsilon^{-2/3}\) for
\(d^{-1}\leqslant\varepsilon\leqslant\varepsilon_0\), so the exponents of \(d\) and
\(\varepsilon\) agree while a factor \(\kappa^{1/3}\) remains.
In the upper bound, this loss comes from accumulating the distance
between consecutive iterates by the triangle inequality.  Improving
that accumulation would require further cancellation across steps;
the present proof does not establish it.
\end{remark}

\section{Underdamped upper bound}
\label{sec:underdamped-upper}

Before proceeding, we note that contraction is measured in the same twisted
metric as the continuous underdamped Langevin dynamics in \cite[Theorem~5.3.6]{ChewiBook}.
Friction acts directly on momentum; the cross term in this metric
transfers its dissipation to position.  For \(z\deq (x,p)\), define
\[
    T_\gamma\deq
    \begin{bmatrix}
      \Id_d&0\\
      \Id_d&\frac{2}{\gamma}\,\Id_d
    \end{bmatrix}\,,
    \qquad
    M_\gamma
    \deq T_\gamma^\top T_\gamma
    =\begin{bmatrix}
       2\Id_d&\frac{2}{\gamma}\,\Id_d\\[0.5em]
       \frac{2}{\gamma}\,\Id_d&\frac4{\gamma^2}\,\Id_d
     \end{bmatrix}\,,
\]
and the hypocoercive norm
\(\norm z_{M_\gamma}^2\deq z^\top M_\gamma z
=\norm{T_\gamma z}^2\).
The associated Wasserstein distance is denoted by
\(\W_{2,M_\gamma}\).  In the normalized setting
\(\beta=1\) and \(\gamma=2\), we omit the subscript and write
\(T\), \(M\), and \(\W_{2,M}\).

Now, we show that the exponential midpoint method
\eqref{eq:EM-x}--\eqref{eq:EM-v} with \(\gamma=2\sqrt\beta\)
achieves $d^{1/4}$ query complexity in Wasserstein distance under the Hessian bounds
\eqref{eq:sc-smooth}. 

\begin{theorem}[Exponential midpoint upper bound]
\label{thm:uld-midpoint-main}
There are universal constants \(c,C>0\) such that, for
\(0<\sqrt\beta\,h\leqslant{c\kappa^{-1/4}}\), the kernel \(K_h^{\mathrm{ExpMid}}\)
has a unique invariant distribution \(\bm\pi_h^{\mathrm{ExpMid}}\) in
\(\mathcal P_2(\R^{2d})\), and
\begin{equation}
    \sqrt\alpha\,\W_2
       \bigl((\bm\pi_h^{\mathrm{ExpMid}})_x,\pi\bigr)
    \leqslant C\sqrt{\kappa d}\,(\sqrt\beta\,h)^2\,.
    \label{eq:uld-midpoint-position-bias}
\end{equation}
For the initialization \(\widehat X_0=x_\star\) and
\(\widehat P_0\sim\cN(0,\Id_d)\), independent of the driving Brownian
motion, for every integer
\(n\geqslant0\),
\begin{equation}
    \sqrt\alpha\,\W_2
       \bigl(\Law(\widehat X_n),\pi\bigr)
    \leqslant C\sqrt d\,\e^{-cn\sqrt\beta\,h/\kappa}
       +C\sqrt{\kappa d}\,(\sqrt\beta\,h)^2\,.
    \label{eq:uld-midpoint-finite-time-explicit}
\end{equation}
Consequently, for \(0<\varepsilon\leqslant\sqrt d/2\), the number
\(N_{\nabla V}\) of gradient queries required to attain the
accuracy criterion \eqref{eq:global-absolute-criterion} satisfies
\begin{equation}
    N_{\nabla V}\leqslant{C\,\frac{\kappa^{5/4}d^{1/4}}{\varepsilon^{1/2}}\,
       \log\frac d{\varepsilon^2}}\,.
    \label{eq:uld-midpoint-complexity}
\end{equation}
\end{theorem}
Just as in the overdamped case, there is a discrepancy in the condition number exponent, which cannot easily be rectified in the current proof.

The remainder of this section outlines the proof; the technical
components are carried out in \Cref{app:underdamped-proof}.
The propositions below use the coordinate normalization in \Cref{sec:normalization}. First, the following contractivity lemma is proved by some elementary coupling analysis.

\begin{proposition}[Numerical hypocoercive contraction]
\label{prop:uld-midpoint-numerical-contraction}
There is a universal constant \(c>0\) such that, for
\(0<h\leqslant{c\kappa^{-1/4}}\) and all
\(\mu,\nu\in\mathcal P_2(\R^{2d})\),
\[
    \W_{2,M}(\mu K_h^{\mathrm{ExpMid}},\nu K_h^{\mathrm{ExpMid}})
    \leqslant\Bigl(1-\frac{h}{4\kappa}\Bigr)\,
       \W_{2,M}(\mu,\nu)\,.
\]
\end{proposition}

The second ingredient is the one-step error at stationarity, which
carries the full order \(h^3\).

\begin{proposition}[One-step error at stationarity]
\label{prop:uld-midpoint-one-step}
There is a universal constant \(C>0\) such that, for
\(0<h\leqslant h_0\),
\[
    \W_{2,M}
       (\bm\pi K_h^{\mathrm{ExpMid}},\bm\pi)
    \leqslant C\sqrt d\,h^3\,.
\]
\end{proposition}

A fixed-point argument gives a unique invariant distribution and
\(\W_{2,M}(\bm\pi_h^{\mathrm{ExpMid}},\bm\pi)
\leqslant C\kappa\sqrt d\,h^2\): divide the one-step error by the
contraction gap. Iterating from
\(\delta_{x_\star}\otimes\cN(0,\Id_d)\)
gives \eqref{eq:uld-midpoint-finite-time-explicit}.

The main work is therefore \Cref{prop:uld-midpoint-one-step}.

\paragraph{The local error.}
Let \((X_t,P_t)_{0\leqslant t\leqslant h}\) be the underdamped Langevin
diffusion started at \((X_0,P_0)\sim\bm\pi\).  Since \(\bm\pi\) is
invariant, the \emph{endpoint} \(Z_h\deq(X_h,P_h)\), i.e., the exact
phase space position after one step, also has distribution \(\bm\pi\).
The exact solution formulas in \Cref{sec:underdamped-algorithms} express \(Z_h\)
through the two gradient integrals
\[
    \int_0^h a_{h-t}\nabla V(X_t)\,\mathrm d t
    \qquad\text{and}\qquad
    \int_0^h\e^{-2(h-t)}\nabla V(X_t)\,\mathrm d t\,.
\]
The exponential midpoint method approximates each integral by \(h\)
times its integrand at \(t=h/2\), with the predictor
\(X_{h/2}^{\mathrm{pred}}\) in place of the exact midpoint \(X_{h/2}\).
The predictor differs from \(X_{h/2}\) by \(O(\sqrt d\,h^3)\) in
\(L^2\), and the gradient enters with a coefficient of order \(h\), so
replacing the predictor by the exact midpoint changes the output by only
\(O(\sqrt d\,h^4)\).  With the exact midpoint, the method started at
\((X_0,P_0)\) and driven by the same Brownian motion outputs
\(Z_h+\mathcal E_h\), where the \emph{local error} \(\mathcal E_h\)
collects the errors of the midpoint rule for the two integrals; see
\eqref{eq:exact-midpoint-defect}.  Since \(Z_h\sim\bm\pi\),
\Cref{prop:uld-midpoint-one-step} reduces to the question of how far the
random displacement \(Z_h\mapsto Z_h+\mathcal E_h\) moves the distribution
\(\bm\pi\) in \(\W_2\).

The na\"{\i}ve bound only gives \(\norm{\mathcal E_h}_{L^2} \le C\sqrt d\,h^2\). Dividing by the contraction gap
\(h/\kappa\) gives only \(C\kappa\sqrt d\,h\) for the bias, one power
of \(h\) short of the theorem.

\paragraph{Conditional mean and centered fluctuation.}
As in the overdamped case, we split
\[
    \mathcal E_h=m_h(Z_h)+R_h\,,
    \qquad
    m_h(Z_h)\deq\E[\mathcal E_h\mid Z_h]\,.
\]
We discuss the \emph{centered fluctuation} \(R_h\) below; the main issue is the conditional mean.

The \emph{conditional mean} \(m_h\) is a deterministic vector field on
phase space, and what matters is not its size but its effect on
\(\bm\pi\).  For a vector field \(F\) on phase space, write
\(D\deq(\nabla_x,\nabla_p)\) and define the \emph{weighted divergence}
\(\delta_{\bm\pi}F\) by
\(\int\ip{D\varphi}{F}\,\mathrm d\bm\pi
=\int\varphi\,\delta_{\bm\pi}F\,\mathrm d\bm\pi\) for all smooth test
functions \(\varphi\).  If \(Z\sim\bm\pi\), a first-order Taylor
expansion gives
\[
    \E\varphi(Z+sF(Z))
    =\E\varphi(Z)+s\int\varphi\,\delta_{\bm\pi}F\,\mathrm d\bm\pi
       +O(s^2)\,,
    \qquad s\to0\,,
\]
so to first order the displacement changes the distribution only
through \(\delta_{\bm\pi}F\). We measure the first-order size of the
displacement by the negative Sobolev norm
\[
    \norm{\delta_{\bm\pi}F}_{\dot H^{-1}(\bm\pi)}
    \deq\sup\Bigl\{\int\varphi\,\delta_{\bm\pi}F\,\mathrm d\bm\pi :
       \norm{D\varphi}_{L^2(\bm\pi)}\leqslant1\Bigr\}\,,
\]
which equals the smallest \(L^2(\bm\pi)\) norm of a vector field with the
same weighted divergence as \(F\) (\Cref{subsec:uld-normalization}).  In
particular, it is at most \(\norm F_{L^2(\bm\pi)}\), but it can be much
smaller.  For small displacements, \(\W_2\) is comparable to this norm
\cite{Peyre2018}.

\paragraph{Cancellation in the weighted divergence.}
\Cref{prop:uld-midpoint-weighted-divergence} shows that
\(\norm{\delta_{\bm\pi}m_h}_{\dot H^{-1}(\bm\pi)}\leqslant C\sqrt d\,h^3\),
one power of \(h\) better than the bound on \(\norm{m_h}_{L^2(\bm\pi)}\).
We show that \(\delta_{\bm\pi}m_h\) is itself the
midpoint quadrature error \(\int_0^h\psi_t\,\mathrm d t-h\psi_{h/2}\) of a
curve \(t\mapsto\psi_t\) of functions on phase space, which starts at
\(\psi_0(x,p)=\ip p{\nabla V(x)}\) and evolves essentially by the
reversed dynamics; see \eqref{eq:uld-midpoint-divergence-quadrature}.  The
second time derivative of this curve involves third derivatives of
\(V\), which are not available pointwise.  In \(\dot H^{-1}(\bm\pi)\),
however, a third derivative of \(V\) is a derivative of the bounded
Hessian, and integration by parts moves this derivative onto the test
function \(\varphi\).  \Cref{lem:uld-midpoint-second-generator} makes this
precise using a Hermite expansion in the momentum, and shows that the
second time derivative is bounded by \(C\sqrt d\) in
\(\dot H^{-1}(\bm\pi)\).  Thus \(\psi_t\) is affine in \(t\) up to an
error \(O(\sqrt d\,t^2)\) in this norm, and the exactness of the midpoint
rule for affine functions gives the bound \(C\sqrt d\,h^3\).

\paragraph{From divergence bounds to transport.}
\Cref{lem:uld-midpoint-random-displacement} turns these estimates into a
Wasserstein bound.  For \(Y\sim\mu\) and a random displacement
\(U=m(Y)+R\) with \(\E[R\mid Y]=0\), it bounds
\(\W_2(\Law(Y+U),\mu)\) by
\(\norm{\delta_\mu m}_{\dot H^{-1}(\mu)}\) plus second-order terms,
which are products of derivatives of \(U\) with the size of \(U\) or
\(R\).  It extends the Gaussian transport estimate of
\cite[Lemma~4.11]{HigherOrderRHMC} to the non-Gaussian phase space
distribution \(\bm\pi\). We use Malliavin calculus estimates to bound these second-order terms.
Together with
\(\norm{\delta_{\bm\pi}m_h}_{\dot H^{-1}(\bm\pi)}\leqslant C\sqrt d\,h^3\),
this proves \Cref{prop:uld-midpoint-one-step}.  The construction and its
regularity requirements are given in
\Cref{subsec:uld-defect,subsec:uld-cancellation,subsec:uld-transport}.

At a high level, the bounds here and in \Cref{sec:overdamped-upper} are related, although the techniques are not exactly comparable. Both objects, instead of directly working in $L^2$, use a divergence-type quantity to analyze the error, which is then related to the error in $W_2$ via a divergence-transport lemma. The $\dot H^{-1}$ norm can be viewed as a linearization of the conditional mean's effect on the distribution, while the Stein kernel in \Cref{sec:overdamped-upper} can be viewed as a generalization of the former which also captures non-linear effects. However, the detailed quantification of their error contributions in \Cref{lem:uld-midpoint-random-displacement} and \Cref{lem:et-Gaussian-convolution-transport} do not lead to comparable estimates in general.

\section{Deterministic lower bounds}
\label{sec:deterministic-lower}

We treat deterministic Heun and underdamped exponential midpoint together.
Their stationary biases give the two deterministic lower bounds in
\Cref{tab:intro-comparison}. We first present the oscillatory
Heun construction and its condition number extension, then the Gaussian
covariance calculation for exponential midpoint. The stationary bias proofs follow the
same order in \Cref{app:deterministic-lower}.
\par

\subsection{Heun on oscillatory targets}
\label{subsec:overdamped-common-construction}

The Heun lower bound uses an oscillatory target: small oscillations of
wavelength \(\sqrt h\) and amplitude \(h\) produce a persistent error
while keeping the Hessian bounded, and a final coordinate is added to ensure the algorithm needs to take a fixed positive time before convergence. The construction depends on the step size.

Specifically, let \(d\geqslant2\), and define
\begin{equation}
    {V_{h,d}(x,z)
    \deq\sum_{i=1}^{d-1}\Bigl\{\frac{x_i^2}{2}
       +a_\star h\cos\frac{x_i}{\sqrt h}\Bigr\}
       +\frac{z^2}{2}\,,
    \qquad (x,z)\in\R^{d-1}\times\R\,.}
    \label{eq:overdamped-common-target}
\end{equation}
Here \(a_\star\in(0,1/4)\) is a fixed numerical constant.
It controls the oscillations and,
since
\[
    (1-a_\star)\Id_d\preceq\nabla^2V_{h,d}
       \preceq(1+a_\star)\Id_d\,,
\]
keeps the condition number bounded independently of \(d\) and \(h\).
Write \(\pi_h\) for the probability distribution on \(\R^d\) with
density proportional to \(\e^{-V_{h,d}}\).
We remark that this construction was used in~\cite{ChewiEtAl2021MALA} to provide upper bounds on the conductance of MALA\@.

\begin{proposition}[Heun lower bounds on an oscillatory target]
\label{prop:heun-oscillatory-lower}
There are universal \(c,h_0>0\) such that, for every \(d\geqslant2\)
and \(0<h\leqslant h_0\), the Heun sampler on
\eqref{eq:overdamped-common-target} has a unique invariant distribution
\(\widehat\pi_h^{\mathrm{Heun}}\in\mathcal P_2\), and
\begin{align}
 &\W_2(\pi_hK_h^{\mathrm{Heun}},\pi_h)\geqslant c\sqrt d\,h^{3/2}\,, \qquad
       &&\W_2(\widehat\pi_h^{\mathrm{Heun}},\pi_h)
 \geqslant c\sqrt d\,h^{3/2}\,,
 \label{eq:overdamped-common-ET}\\
 &\operatorname{KL}(\pi_hK_h^{\mathrm{Heun}}\mmid\pi_h)\geqslant cdh^2\,,\qquad
       &&\operatorname{KL}(\widehat\pi_h^{\mathrm{Heun}}\mmid\pi_h)
 \geqslant cdh^2\,.\notag
\end{align}
\end{proposition}

The Heun proof in \Cref{subsec:overdamped-bounded-proof} computes the
change in expectation of the \(1\)-Lipschitz function
\(\sqrt h\cos(\cdot/\sqrt h)\) after one update started from the target.
A Gaussian calculation shows this is of order \(h^{3/2}\), which is then used to lower bound the Wasserstein error at stationarity. Tensorization gives the dimension dependence. Pinsker's inequality applied
to the same observable divided by \(\sqrt h\) gives the KL bounds.

\subsection{The condition number dependence for Heun}
\label{subsec:overdamped-condition-number}

The preceding construction has bounded condition number. To capture the
cost of a large \(\kappa\), we couple the oscillatory coordinates to a
coordinate whose curvature is of order \(\kappa^{-1}\). The numerical
error in the former will have an effect on the latter which is amplified by the inverse curvature. Our potential consists of numerous coupled blocks, each containing at most \(h^{-1}\)
oscillatory coordinates, which ensures that we can keep the Hessian bounded.
To make the construction explicit, in normalized units a block
with step \(t\asymp\beta h\) has potential
\[
 \begin{aligned}
(x,y)&\longmapsto \frac{\norm x^2}{2}+\frac{y^2}{2K} +a\ell^2\bigl(2+\sin(y/\sqrt K)\bigr)
           \sum_{i=1}^m\cos(x_i/\ell)\,,\qquad \ell^2=\frac t\tau\,.
 \end{aligned}
\]
Here \(a>0\) is a small fixed numerical constant, \(\tau>1\) is a
large fixed numerical constant, \(K\asymp\kappa\), and \(mt\leqslant1\). Products of these
blocks give the dimension dependence.
Equation \eqref{eq:heun-product-potential} gives the full product
potential; the proof then rescales it to the prescribed \(\alpha,\beta\). A final coordinate, independent of all others
with the same small curvature, supplies a lower bound on relaxation time, so that the algorithm must run for some number of iterations before it is close to convergence.

\begin{theorem}[Condition number lower bound for Heun]
\label{thm:et-linear-kappa-lower}
There are universal \(c,C,\varepsilon_0>0\), \(d_0\geqslant4\), and
\(\kappa_0>1\) with the following property. Let \(d\geqslant d_0\),
\(\kappa\geqslant\kappa_0\), \(0<\varepsilon\leqslant\varepsilon_0\), and
let Heun use a fixed step with \(0<\beta h\leqslant c\).
There is an \(\alpha\)-strongly convex, \(\beta\)-smooth potential
\(V_h\) on \(\R^d\), with \(\beta/\alpha=\kappa\), and an initial
 distribution \(\nu_h\) such that
\[
    \E_{X_0\sim\nu_h}\bigl[V_h(X_0)-V_h(x_\star)\bigr]\leqslant Cd\,,
    \qquad x_\star\deq\argmin V_h\,.
\]
Writing \(\pi_h(\mathrm d x)\propto\e^{-V_h(x)}\,\mathrm d x\), for every
integer \(N\geqslant0\),
\(\sqrt\alpha\,\W_2(\nu_h(K_h^{\mathrm{Heun}})^N,\pi_h)
\leqslant\varepsilon\) implies
\begin{equation}
    N\geqslant {c\kappa
       \min\Bigl\{\frac{d^{1/3}}{\varepsilon^{2/3}}\,,
   \frac{d^{1/2}}{\varepsilon^{1/2}}\Bigr\}\log \frac{1}{\varepsilon}}\,.
    \label{eq:linear-kappa-endpoint-rate}
\end{equation}
In particular, in the moderate accuracy range
\(d^{-1}\leqslant\varepsilon\leqslant\varepsilon_0\),
\begin{equation}
    N\geqslant \frac{c\kappa d^{1/3}}{\varepsilon^{2/3}}
                 \log\frac1\varepsilon\,.
    \label{eq:linear-kappa-third-root}
\end{equation}
\end{theorem}

The estimates and the product construction are proved in
\Cref{subsec:heun-modulated-block-proof,subsec:heun-product-lower-proof}.

\subsection{Underdamped exponential midpoint}
\label{subsec:exponential-midpoint-lower}

For the underdamped Gaussian calculation, let us first consider the $\kappa = 1$ case. Then, we choose \(V(x)=\norm x^2/2\),
\(\alpha=\beta=1\), and \(\gamma=2\), so \(\pi=\cN(0,\Id_d)\).
More generally, when $\kappa \gg 1$, let \(d\geqslant2\), \(0<\alpha\leqslant\beta\),
\(\kappa\deq\beta/\alpha\geqslant2\), and \(\gamma=2\sqrt\beta\).
We consider the anisotropic Gaussian target with potential
\begin{equation}
 V(x)\deq\frac\beta2\,\sum_{i=1}^{d-1}x_i^2
                    +\frac\alpha2\,x_d^2\,,
 \label{eq:em-anisotropic-potential}
\end{equation}
to capture the dependence on $\kappa$ in our bounds.

On the quadratic target, exponential midpoint is a linear update with a
deterministic state matrix and additive Gaussian noise. Its invariant
distribution is therefore Gaussian, and the error in its position variance
gives both the isotropic lower bound and its anisotropic extension.

\begin{proposition}[Exponential midpoint on a Gaussian target]
\label{prop:exponential-midpoint-Gaussian-lower}
There are universal constants \(c>0\) and \(h_0\in(0,1]\) such that,
for the isotropic target \(V(x)=\norm x^2/2\), with
\(\alpha=\beta=1\) and \(\gamma=2\), for every \(d\geqslant1\) and
\(0<h\leqslant h_0\), the kernel \(K_h^{\mathrm{ExpMid}}\) has a unique
invariant distribution \(\bm\pi_h^{\mathrm{ExpMid}}\) with finite second
moment. Its position marginal is \(\cN(0,s_h\Id_d)\), where
\begin{equation*}
    s_h=1-\frac{h^2}{24}+O(h^3)\,.
\end{equation*}
The remainder is universal. In particular,
\begin{equation}
    \W_2\bigl((\bm\pi_h^{\mathrm{ExpMid}})_x,\pi\bigr)
       \geqslant c\sqrt d\,h^2\,.
    \label{eq:exponential-midpoint-lower}
\end{equation}

More generally, for the dynamics corresponding to~\eqref{eq:em-anisotropic-potential},
the kernel has a unique invariant distribution in
\(\mathcal P_2(\R^{2d})\) whenever
\(0<\sqrt\beta\,h\leqslant h_0\), and its position bias satisfies
\[
 \sqrt\alpha\,\W_2\bigl((\bm\pi_h^{\mathrm{ExpMid}})_x,\pi\bigr)
 =\frac{(\sqrt\beta\,h)^2}{48}\,
       \sqrt{\frac{d-1}{\kappa}+\frac1{\kappa^2}}\,
       \{1+O(\sqrt\beta\,h)\}\,.
\]
All constants and remainders are uniform in \(d,\alpha,\beta\).
Initialize
\[
 X_0=\alpha^{-1/2}e_d\,,\qquad
 P_0\sim\cN(0,\Id_d)\,,
\]
with momentum independent of future randomness, where \(e_d\) is the
final coordinate vector. This start satisfies
\(\E[V(X_0)-V(x_\star)]=1/2\). For a universal \(C>0\) and every
integer \(N\geqslant0\),
\begin{equation}
 \sqrt\alpha\,\W_2\bigl(\Law(\widehat X_N),\pi\bigr)
 \geqslant\max\Bigl\{
       c\sqrt{\frac d\kappa}\,(\sqrt\beta\,h)^2,
       \e^{-CN\sqrt\beta\,h/\kappa}\Bigr\}\,.
 \label{eq:em-anisotropic-finite-time}
\end{equation}
Consequently, for \(0<\varepsilon\leqslant1/2\), achieving
\(\sqrt\alpha\,\W_2(\Law(\widehat X_N),\pi)\leqslant\varepsilon\)
requires
\begin{equation}
 N\geqslant c\,\Bigl(
       \kappa+\frac{\kappa^{3/4}d^{1/4}}{\varepsilon^{1/2}}
       \Bigr)\log\frac1\varepsilon\,.
 \label{eq:em-anisotropic-iteration}
\end{equation}

\end{proposition}

The coordinates of curvature \(\beta\) impose the step size restriction
through their variance error, while the coordinate of curvature
\(\alpha\) requires elapsed time of order
\(\kappa\log(1/\varepsilon)\) in normalized units.
The resulting \(\kappa^{3/4}\) dependence in
\eqref{eq:em-anisotropic-iteration} leaves a discrepancy of
\(\kappa^{1/2}\) between the accuracy-dependent terms of this lower
bound and \Cref{thm:uld-midpoint-main}.

The proof in \Cref{subsec:exponential-midpoint-Gaussian-proof} computes the
invariant covariance of one position--momentum pair and applies the exact
Wasserstein formula for centered Gaussian distributions.

\section{Randomized lower bounds}
\label{sec:randomized-lower}

We illustrate the technique for the randomized lower bounds at a high level. To see where the factor \(d^{1/3}\) in the iteration bounds comes
from, consider RLMC on \(V=\norm\cdot^2/2\); the other randomized methods
behave in the same way.  Conditionally on the evaluation time \(\Theta\),
one step has the form
\[
    X^+=m(\Theta)\,X+\sqrt{q(\Theta)}\,G_d
\]
with scalar coefficients \(m(\Theta)\) and \(q(\Theta)\) and a fresh
\(G_d\sim\cN(0,\Id_d)\); see the proof of
\Cref{prop:gaussian-radial-lower}.  The crucial point is that a single
draw of \(\Theta\) is used for all \(d\) coordinates: every coordinate is
integrated with the same evaluation time.  Hence, conditionally on the
past draws of \(\Theta\), the iterate is exactly \(\cN(0,S\Id_d)\) for a
single random scalar \(S \perp G\) shared by all coordinates, and at stationarity
\(S\stackrel{\mathrm d}=m(\Theta)^2S'+q(\Theta)\), with \(S'\) an
independent copy of \(S\).

On average the quadrature is very accurate: \(\E S=1+O(h^3)\).  In a
single step, however, started from the target variance \(S=1\), the
iterate has conditional variance \(m(\Theta)^2+q(\Theta)\), and its error
\(m(\Theta)^2+q(\Theta)-1=h^2\,(1-2\Theta)+h^4\Theta^2\) has mean only
\(h^4/3\) but fluctuates by order \(h^2\).  Since \(1-\E m(\Theta)^2\approx2h\), the recursion for
\(S\) forgets its past after about \(1/h\) steps, so \(S\) accumulates
roughly \(1/h\) independent fluctuations of size \(h^2\).  This gives
\(\operatorname{Var}(S)\asymp h^4\cdot h^{-1}=h^3\).

Because \(S\) is shared, these fluctuations add up in the
norm.  Indeed, \(\norm X^2=S\,\norm{G_d}^2\approx Sd\), so \(\norm X\)
fluctuates by about \(\sqrt d\,h^{3/2}\), whereas under \(\cN(0,\Id_d)\)
the norm has variance at most \(1\) by the Gaussian Poincar\'e
inequality.  (If instead each coordinate used its own independent
evaluation time, the coordinate variances would fluctuate independently,
and their fluctuations would largely cancel in
\(\norm X\).)  Since the norm is \(1\)-Lipschitz,
no coupling can hide the extra spread of \(\norm X\), and \(\W_2\) is at
least of order \(\sqrt{1+dh^3}-1\asymp\min\{dh^3,\sqrt{dh^3}\}\); this
is \Cref{lem:gaussian-scale-perturbation}. Note that deterministic midpoint methods avoid this common scale obstruction.

The proof template is the same for all the randomized methods: derive
the fixed-point equation for the conditional covariance $S$, compute its
variance and centered fourth moment through \Cref{lem:scalar-linear-mixture}/\ref{lem:random-linear-mixture} via Taylor expansion of the equation for small $h$, and apply
\Cref{lem:gaussian-scale-perturbation} to translate these moments into a lower bound for the Wasserstein distance. The scalar calculation for RLMC
and PLMC comes first; RULMC and PULMC require
some matrix algebra instead. \Cref{app:randomized-lower} supplies these
calculations, followed by the unified proof of
\Cref{cor:midpoint-iteration-lower}.
\par

\subsection{Overdamped randomized midpoint and Poisson midpoint}
\label{subsec:overdamped-randomized-lower}

\begin{proposition}[Randomized overdamped lower bounds on a Gaussian target]
\label{prop:overdamped-common-construction}
There are universal \(c,h_0>0\) such that, for every \(d\geqslant1\) and
\(0<h\leqslant h_0\), RLMC and PLMC on \(V(x)=\norm x^2/2\)
have unique invariant distributions in \(\mathcal P_2\). Their errors satisfy
\begin{align}
    \W_2(\widehat\pi_h^{\mathrm{RLMC}},\cN(0,\Id_d))
       &\geqslant c\min\{dh^3,\sqrt{dh^3}\}\,,
       \label{eq:overdamped-common-RLMC}\\
    \W_2(\widehat\pi_h^{\mathrm{PLMC},k},\cN(0,\Id_d))
       &\geqslant c\min\{dh^3,\sqrt{dh^3}\}\,,\qquad k\geqslant2\,,\notag
\\
    \W_2(\widehat\pi_h^{\mathrm{PLMC},1},\cN(0,\Id_d))
       &\geqslant c\sqrt d\,h\,.
       \label{eq:overdamped-common-ULA}
\end{align}
The constants are uniform in \(k\). For \(k=1\), PLMC is the
Euler--Maruyama scheme (ULA).
\end{proposition}

At any fixed sufficiently small accuracy, the first two bounds
require \(h\lesssim d^{-1/3}\), while the \(k=1\) case requires
\(h\lesssim d^{-1/2}\). The latter recovers the Gaussian calculation
in~\cite{Wibisono2018}. For the randomized cases, the proof in
\Cref{subsec:randomized-overdamped-proof} combines moment estimates for
the linear update coefficients with
\Cref{lem:scalar-linear-mixture,lem:gaussian-scale-perturbation} to
obtain the Wasserstein lower bounds.

\subsection{Underdamped randomized midpoint and Poisson midpoint}
\label{subsec:underdamped-randomized-lower}

We use the same Gaussian target and underdamped normalization as in
\Cref{subsec:exponential-midpoint-lower}; its rescaling statement also
applies here. The scalar covariance calculation for RLMC and PLMC is
replaced by a matrix calculation for a representative position--momentum pair.
\par

\begin{proposition}[Randomized methods on a Gaussian target]
\label{prop:randomized-underdamped-Gaussian-lower}
There are universal \(c,h_0>0\) such that, for every \(d\geqslant1\) and
\(0<h\leqslant h_0\), RULMC and PULMC on \(V(x)=\norm x^2/2\)
have unique invariant distributions in \(\mathcal P_2\). Their errors satisfy
\begin{align}
    \W_2\bigl((\bm\pi_h^{\mathrm{RULMC}})_x,\cN(0,\Id_d)\bigr)
       &\geqslant c\min\{dh^3,\sqrt{dh^3}\}\,,
       &&\notag\\
    \W_2\bigl((\bm\pi_h^{\mathrm{PULMC},k})_x,\cN(0,\Id_d)\bigr)
       &\geqslant c\min\{dh^3,\sqrt{dh^3}\}\,,
       &&k\geqslant2\,,\notag\\
    \W_2\bigl((\bm\pi_h^{\mathrm{PULMC},1})_x,\cN(0,\Id_d)\bigr)
       &\geqslant c\sqrt d\,h\,.\notag
\end{align}
The constants are uniform in \(k\). For \(k=1\), PULMC is the
exponential Euler scheme (ULMC).
\end{proposition}

\Cref{subsec:randomized-underdamped-proof,subsec:poisson-underdamped-proof}
derive the fixed-point equation for these systems, which is then passed again through
\Cref{lem:random-linear-mixture} and \Cref{lem:gaussian-scale-perturbation} to convert this into a Wasserstein lower bound.

\subsection{Iteration lower bounds for all midpoint methods}
\label{subsec:lower-iterations}

We collect the iteration consequences of the deterministic bounds in
\Cref{sec:deterministic-lower} and the preceding randomized bounds in one
statement. In each construction, a displaced position coordinate forces
a positive elapsed time in order that the given coordinate converges, while the error in the remaining coordinates
forces a small step size.

\begin{corollary}[Iteration lower bounds for all midpoint methods]
\label{cor:midpoint-iteration-lower}
There are universal \(c,C,\varepsilon_0,h_0>0\) such that, for every
\(d\geqslant2\), \(0<h\leqslant h_0\), each method listed below,
and every Poisson grid size \(k\geqslant1\), there is an initial position distribution \(\nu_0\) satisfying
\[
 \E_{X_0\sim\nu_0}[V(X_0)-V(x_\star)]\leqslant Cd\,,
 \qquad \alpha\W_2^2(\nu_0,\pi)\leqslant Cd\,,
 \qquad x_\star\deq\argmin V\,,
\]
with the following property. Write \(\nu_N\) for the position distribution
after \(N\) steps. For every \(0<\varepsilon\leqslant\varepsilon_0\)
and integer \(N\geqslant0\), position accuracy
\(\sqrt\alpha\,\W_2(\nu_N,\pi)\leqslant\varepsilon\) requires:
\[
\begin{array}{l|c}
 \text{Method}&\text{Necessary iterations}\\ \hline
 \text{Heun}&N\geqslant cd^{1/3}\varepsilon^{-2/3}\\
 \text{Exponential midpoint}&N\geqslant cd^{1/4}\varepsilon^{-1/2}\\
 \text{RLMC or RULMC}
    &N\geqslant cd^{1/3}\varepsilon^{-1/3}\\
 \text{PLMC or PULMC with }k\geqslant2
    &N\geqslant cd^{1/3}\varepsilon^{-1/3}\\
 \text{ULA or ULMC (}k=1\text{)}
 &N\geqslant cd^{1/2}\varepsilon^{-1}
\end{array}
\]
Here, \(V\deq V_{h,d}\) from \eqref{eq:overdamped-common-target} for Heun,
and \(V\deq\norm \cdot^2/2\) for every other method; the condition numbers
are bounded by a universal constant. The randomized methods use
quadrature choices shared across coordinates, and the constants are
uniform in the Poisson grid size \(k\).

For all underdamped methods, \(\gamma=2\) and
\(P_0\sim\cN(0,\Id_d)\) is independent of \(X_0\) and future
randomness, and the initial positions are
\[
 \nu_0\deq\delta_{e_d}\quad\text{for exponential midpoint}\,,\qquad
 \nu_0\deq\cN(e_d,\Id_d)\quad\text{for RULMC and PULMC}\,.
\]
These choices are independent of \(h,k\); the latter has
\(\W_2(\nu_0,\pi)=1\). The overdamped initial distributions may depend
on the method, \(h\), and \(k\), but not on \(\varepsilon\) or \(N\).

Under the hypotheses of \Cref{thm:et-linear-kappa-lower}, its Heun target
and initialization satisfy the same two initialization bounds, and
\(\sqrt\alpha\,\W_2(\nu_N,\pi)\leqslant\varepsilon\) requires
\[
 N\geqslant c\kappa
 \min\{d^{1/3}\varepsilon^{-2/3},
                  d^{1/2}\varepsilon^{-1/2}\}\log\frac1\varepsilon\,.
\]
\end{corollary}

\par

\newpage
\appendix
\section{Proof of the overdamped upper bound}
\label{app:overdamped-proof}

We complete the deterministic Heun proof outlined in
\Cref{sec:overdamped-upper}.
We work throughout in the
normalized units \(\beta=1\), \(\alpha=\kappa^{-1}\), with the fixed
threshold \(h_0\) from \Cref{sec:normalization}.  In particular,
\(V\in C^2\) and \(\kappa^{-1}\Id_d\preceq\nabla^2V\preceq\Id_d\).
\Cref{subsec:et-upper-completion} translates the resulting estimate back
to the original units.  The analytic background for this and all other appendices is collected in
\Cref{app:toolbox}.

We explain the main idea intuitively. Firstly, we will work at a shifted measure $\nu$, which removes the leading bias term without too much cost, and consider the effect of the dynamics around this measure. The remaining fluctuation term is still too large when measured in $L^2$. To witness additional cancellations, we instead provide estimates on the conditional Stein kernel, which in turn requires detailed analysis of Schatten norms and Witten Laplacians. This is the most technical part of the proof and involves heavy technical machinery from Malliavin calculus and sectorial functional calculus. After doing this, we apply the technique of~\cite[Theorem E.4]{PaulinWhalley2026} to convert this back to a transport bound, which accumulates via Lemma~\ref{lem:et-local-framework}.

A detailed roadmap is also provided below, which follows the sequential presentation of the proof. 
We begin with \Cref{subsec:et-auxiliary-proofs}, which opens with some useful auxiliary estimates.
\begin{enumerate}
\item \emph{Correcting the conditional mean.}
\Cref{subsec:et-corrector} solves the discrete Poisson equation
\((\Id-P_h)\chi=e_h\) and shifts the target to
\(\nu\). This absorbs the leading conditional bias at
transport cost \(O(\sqrt d\,h^{3/2})\); the remaining conditional mean
is only \(O(\sqrt d\,h^{5/2})\).

\item \emph{A small conditional Stein kernel.}
Even after centering, the remaining fluctuation is still \(O(\sqrt d\,h^{3/2})\) in \(L^2\), which is too large.
\Cref{subsec:et-malliavin-stein} instead estimates its conditional Stein
kernel, which captures the smaller error relevant after Gaussian smoothing.

\item \emph{Corrector regularity from the bounded Hessian.}
The Stein kernel estimate requires Schatten-$4$ bounds on the corrector's
derivative. \Cref{subsec:et-S4-multiplier} proves the Witten Laplacian
multiplier estimate that supplies these bounds without higher derivatives
of \(V\); here, we need to invoke numerous Malliavin derivative bounds and a sectorial functional calculus argument.

\item \emph{Gaussian smoothing.}
\Cref{subsec:et-transport} converts the conditional Stein kernel bound into
a transport cost of \(O(\sqrt d\,h^{5/2})\) after convolution. This is the
step that improves on accumulating the size of the fluctuation itself.

\item \emph{Accumulation and the implementable method.}
\Cref{subsec:et-assembly} controls the predictor replacement and applies
\Cref{lem:et-local-framework}: contraction accumulates the smoothed error
with a factor of order \(\kappa/h\), giving invariant distribution bias
\(O(\kappa\sqrt d\,h^{3/2})\). \Cref{subsec:et-upper-completion} adds
initialization and returns to the original units.
\end{enumerate}

The main ideas have established precedents.  Mattingly, Stuart, and
Tretyakov~\cite{MattinglyStuartTretyakov2010} use the continuous Poisson
equation to analyze stationary discretization error, while discrete Poisson equations were used by \cite{JohndrowMattingly2017} and \cite{LuYeZhou2025} to study sampling
algorithms. The direct predecessor of our construction is the discrete
corrector developed by Pedrotti and Whalley~\cite{PedrottiWhalley2026}
for the unadjusted Langevin algorithm.
The Malliavin--Stein approach of Nourdin and Peccati is presented in
\cite{NourdinPeccati2012}, and the Stein kernel representation used here
is from Nourdin, Peccati, and Swan~\cite[Proposition~3.7]{NourdinPeccatiSwan2014}.
The independent Gaussian convolution bound is due to Paulin and
Whalley~\cite[Theorem~E.4]{PaulinWhalley2026}.
In \Cref{lem:et-gaussian-convolution}, we give a more direct
transport proof of the preceding result with sharp numerical constant.
The subsequent geometric accumulation is the standard contractive
coupling argument; see \cite[Chapter~4]{ChewiBook}.
Here the work is to combine these tools with the corrector and its
Schatten derivative estimates under \(C^2\) regularity.

\subsection{Auxiliary proofs for the Heun upper bound}
\label{subsec:et-auxiliary-proofs}

Throughout the appendix, \(X_t^x\) denotes the Langevin diffusion
\eqref{eq:langevin} started from \(x\).  We first record elementary
bounds on its trajectory.

\begin{lemma}[Displacement estimates]
\label{lem:et-displacement}
There is a universal constant \(C>0\) such that, for
\(0<t\leqslant h_0\) and \(x\in\R^d\),
\begin{equation}
    \E\norm{X_t^x-x}^2
    \leqslant C\,\bigl\{dt+t^2\,\norm{\nabla V(x)}^2\bigr\}\,;
    \label{eq:et-displacement-pointwise}
\end{equation}
if \(X_0\sim\pi\), then
\begin{equation}
    \E\norm{X_t-X_0}^2\leqslant Cdt\,.
    \label{eq:et-displacement-stationary}
\end{equation}
\end{lemma}

\begin{proof}
The pointwise estimate is \cite[Lemma~4.1.8]{ChewiBook} with deterministic
initial state, \(\beta=1\), and \(t\leqslant h_0\leqslant1/3\).
For a stationary initial state, apply the same lemma and use
\(\int\norm{\nabla V}^2\,\mathrm d \pi
=\int\Delta V\,\mathrm d \pi\leqslant d\).
\end{proof}

We next prove
\Cref{prop:et-contractivity} and \Cref{lem:et-local-framework},
stated in \Cref{sec:overdamped-upper}.

\begin{proof}[Proof of \Cref{prop:et-contractivity}]
Couple the steps from \(x\) and \(y\) with the same Brownian increment.
Let \(X_h^{\mathrm{pred}}\) and \(Y_h^{\mathrm{pred}}\) be their
predictors.  The fundamental theorem of calculus gives
\[
\begin{split}
    H_{\mathrm{init}}&\deq\int_0^1\nabla^2V\bigl(y+s\,(x-y)\bigr)\,\mathrm d s\,,\\
    H_{\mathrm{pred}}&\deq\int_0^1\nabla^2V\bigl(Y_h^{\mathrm{pred}}
       +s\,(X_h^{\mathrm{pred}}-Y_h^{\mathrm{pred}})\bigr)\,\mathrm d s\,,
\end{split}
\]
so \(\nabla V(x)-\nabla V(y)=H_{\mathrm{init}}\,(x-y)\) and
\(\nabla V(X_h^{\mathrm{pred}})-\nabla V(Y_h^{\mathrm{pred}})
=H_{\mathrm{pred}}\,(X_h^{\mathrm{pred}}-Y_h^{\mathrm{pred}})\).
The predictor difference is \((\Id_d-hH_{\mathrm{init}})(x-y)\).  Substitution into
the update \eqref{eq:et-step} therefore yields
\begin{equation}
\begin{split}
    X^+-Y^+
    &=\bigl\{\Id_d-\tfrac h2\,H_{\mathrm{init}}-\tfrac h2\,H_{\mathrm{pred}}\,(\Id_d-hH_{\mathrm{init}})\bigr\}(x-y)\\
    &=\tfrac12\,\bigl\{\Id_d+(\Id_d-hH_{\mathrm{pred}})(\Id_d-hH_{\mathrm{init}})\bigr\}(x-y)\,.
    \label{eq:et-contraction-factorization}
\end{split}
\end{equation}
Both \(H_{\mathrm{init}}\) and \(H_{\mathrm{pred}}\) are symmetric with eigenvalues in
\([\kappa^{-1},1]\).  If \(0<h\leqslant1\), the eigenvalues of
\(\Id_d-hH_{\mathrm{init}}\) and \(\Id_d-hH_{\mathrm{pred}}\) lie in
\([0,1-h/\kappa]\); their operator norms are therefore at most
\(1-h/\kappa\).  No commutation of \(H_{\mathrm{init}}\) and \(H_{\mathrm{pred}}\) is needed:
submultiplicativity and \eqref{eq:et-contraction-factorization} give
\[
    \norm{X^+-Y^+}
    \leqslant\frac{1+(1-h/\kappa)^2}{2}\,\norm{x-y}
    \leqslant\Bigl(1-\frac h{2\kappa}\Bigr)\,\norm{x-y}\,.
\]
Integrating this pathwise bound over an optimal coupling of the initial
states proves \eqref{eq:et-contraction-factor}.  Lipschitzness of the
force ensures that the kernel preserves \(\mathcal P_2(\R^d)\), so
Banach's fixed point theorem on the complete metric space
\((\mathcal P_2(\R^d),\W_2)\) gives the invariant distribution.
\end{proof}

\begin{proof}[Proof of \Cref{lem:et-local-framework}]
Let \(0<h\leqslant1\).  For \(j\geqslant1\), contractivity gives
\[
    \W_2\bigl(\nu(K_h^{\mathrm{Heun}})^{j+1},
       \nu(K_h^{\mathrm{Heun}})^j\bigr)
    \leqslant\rho^{j-1}
    \W_2\bigl(\nu(K_h^{\mathrm{Heun}})^2,\nu K_h^{\mathrm{Heun}}\bigr)\,.
\]
The triangle inequality from \(\nu\) to \(\nu(K_h^{\mathrm{Heun}})^N\)
and summing as a geometric series suffice to bound
\(\W_2(\nu,\nu(K_h^{\mathrm{Heun}})^N)\) by the right-hand side of
\eqref{eq:et-local-framework}, uniformly in \(N\).  Since
\(\W_2(\nu(K_h^{\mathrm{Heun}})^N,\widehat\pi_h^{\mathrm{Heun}})
\leqslant\rho^N\W_2(\nu,\widehat\pi_h^{\mathrm{Heun}})\to0\),
the claim follows.
\end{proof}

\subsection{The corrector and the discrete Poisson equation}
\label{subsec:et-corrector}

The Heun method averages the force at the starting point and
the predictor.  For the analysis, replace the predictor by the exact
endpoint and define
\begin{equation}
    \widetilde X_h^x
    \deq x-\frac h2\,\{\nabla V(x)+\nabla V(X_h^x)\}+\sqrt2\,B_h\,,
    \qquad
    \widetilde K_h(x,\cdot)\deq\Law(\widetilde X_h^x)\,.
    \label{eq:et-exact-reference}
\end{equation}
This auxiliary step uses the same Brownian motion as the exact diffusion.
\Cref{subsec:et-assembly} bounds the cost of replacing its exact endpoint
by the implementable predictor.

Start both processes at \(X_0\sim\pi\).  Their difference is the
trapezoidal quadrature error
\begin{equation}
    \mathcal E_h
    \deq\widetilde X_h-X_h
    =\int_0^h\nabla V(X_t)\,\mathrm d t
       -\frac h2\,\{\nabla V(X_0)+\nabla V(X_h)\}\,.
    \label{eq:et-quadrature-error}
\end{equation}
We construct a corrector for its conditional mean
\(e_h(X_h)\deq\E[\mathcal E_h\mid X_h]\).

Our sign convention is numerical minus exact in
\eqref{eq:et-quadrature-error}, so the integral term in \(e_h\) has a
positive sign. Conditionally on \(X_h\), reversibility
makes the reversed path \((X_{h-t})_{0\leqslant t\leqslant h}\) a
Langevin diffusion started from \(X_h\), so
\(\E[\nabla V(X_t)\mid X_h]=(P_{h-t}\nabla V)(X_h)\) for
\(0\leqslant t\leqslant h\).  Taking conditional means in
\eqref{eq:et-quadrature-error} term by term and substituting
\(t\mapsto h-t\) in the integral gives
\begin{equation}
    e_h
    =\int_0^hP_t(\nabla V)\,\mathrm d t
       -\frac h2\,\bigl\{\nabla V+P_h(\nabla V)\bigr\}\,,
    \label{eq:et-conditional-mean}
\end{equation}
as vector fields on \(\R^d\).
Invariance and \(\int\nabla V\,\mathrm d \pi=0\) show that
\(e_h\) integrates to zero against \(\pi\).
Define the corrector \(\chi\) as the solution of the discrete
Poisson equation
\begin{equation}
    (\Id-P_h)\chi=e_h\,,
    \qquad\text{equivalently}\qquad
    \chi=(\Id-P_h)^{-1}e_h\,.
    \label{eq:et-corrector}
\end{equation}
On the mean-zero subspace of \(L^2(\pi)\), the spectral gap gives
\(\norm{P_h}\leqslant\e^{-h/\kappa}<1\) for every \(h>0\).
Thus, the equation has a unique solution in that subspace, with the
Neumann series representation
\begin{equation}
    \chi=\sum_{n=0}^\infty P_h^n e_h
        =\sum_{n=0}^\infty P_{nh}e_h
    \qquad\text{in }L^2(\pi)\,.
    \label{eq:et-corrector-series}
\end{equation}
Indeed, iterating \(\chi=e_h+P_h\chi\) gives
\[
    \chi-\sum_{n=0}^{N}P_{nh}e_h=P_{(N+1)h}\chi\,,
    \qquad
    \norm{P_{(N+1)h}\chi}_{L^2(\pi)}
    \leqslant\e^{-(N+1)h/\kappa}\,\norm\chi_{L^2(\pi)}\longrightarrow0\,.
\]

On the mean-zero subspace, the semigroup identity
\[
    \int_0^hP_t\,\mathrm d t=(-L)^{-1}(\Id-P_h)
\]
allows us to solve \eqref{eq:et-corrector} directly.  Substitution in
\eqref{eq:et-conditional-mean}, followed by \(P_h=\e^{hL}\), gives
\begin{equation}
\begin{aligned}
    \chi
    &=\bigl\{(-L)^{-1}
       -\tfrac h2\,(\Id-P_h)^{-1}(\Id+P_h)\bigr\}(\nabla V)\\
    &=h\,\Bigl\{(-hL)^{-1}
       -\frac12\,\coth\Bigl(\frac{-hL}{2}\Bigr)\Bigr\}(\nabla V)=h\,r(-hL)(\nabla V)\,.
\end{aligned}
    \label{eq:et-corrector-multiplier}
\end{equation}
The last equality defines the scalar function
\[
    r(z)\deq\frac1z-\frac12\,\coth\frac z2\,,
    \quad r(0)\deq0\,.
\]

To estimate \(\norm{\chi}_{L^2(\pi)}^2\), we need the following estimate
of \(r\):
\[
    \lvert r(z)\rvert\leqslant C\min\{z,1\}
    \leqslant C\sqrt z\,,
    \qquad z\geqslant0\,.
\]
Indeed, \(r(z)=-z/12+O(z^3)\) near zero, while both \(z^{-1}\) and
\(\coth(z/2)\) are bounded for \(z\geqslant1\).  Combining this with
\eqref{eq:et-corrector-multiplier} and the spectral theorem gives
\begin{equation}
\begin{aligned}
    \norm{\chi}_{L^2(\pi)}^2
    &=h^2\sum_{j\in[d]}
       \norm{r(-hL)(\partial_jV)}_{L^2(\pi)}^2\leqslant Ch^3\sum_{j\in[d]}
       \norm{(-L)^{1/2}(\partial_jV)}_{L^2(\pi)}^2\\
    &=Ch^3\sum_{j\in[d]}
       \int\norm{\nabla\partial_jV}^2\,\mathrm d \pi=Ch^3\int\norm{\nabla^2V}_{\mathrm{HS}}^2\,\mathrm d \pi\leqslant Cdh^3\,.
\end{aligned}
    \label{eq:et-corrector-L2}
\end{equation}
The coupling \(X\mapsto X+\chi(X)\), with \(X\sim\pi\), now gives
\[
    \W_2^2((\Id+\chi)_\#\pi,\pi)
    \leqslant\norm\chi_{L^2(\pi)}^2\leqslant Cdh^3\,.
\]

\subsection{Malliavin--Stein kernel estimate}
\label{subsec:et-malliavin-stein}

We now couple the auxiliary step from
\(\nu=(\Id+\chi)_\#\pi\) to another sample from \(\nu\).
Start a stationary diffusion \((X_t)_{0\leqslant t\leqslant h}\) at
\(X_0\sim\pi\).  Let \(\overline X_t\) be the diffusion started at
\(X_0+\chi(X_0)\), using the same Brownian path, and set
\[
    \widetilde{\overline X}_h
    \deq X_0+\chi(X_0)
       -\frac h2\,\{\nabla V(X_0+\chi(X_0))
          +\nabla V(\overline X_h)\}+\sqrt2\,B_h
       \sim\nu\widetilde K_h\,.
\]
We retain \(\mathcal E_h=\widetilde X_h-X_h\) for the raw
quadrature error and write \(\mathcal E_h^{\mathrm{corr}}\) for the
error after shifting the initial distribution and the reference endpoint
by the corrector.  Since \(X_h+\chi(X_h)\sim\nu\), decompose this
corrected error into its conditional mean and centered fluctuation:
\begin{equation}
\begin{aligned}
    \mathcal E_h^{\mathrm{corr}}
    &\deq\widetilde{\overline X}_h-\{X_h+\chi(X_h)\}
      =m(X_h)+R\,,\\
    m(X_h)&\deq\E[\mathcal E_h^{\mathrm{corr}}\mid X_h]\,,
    \qquad \E[R\mid X_h]=0\,.
\end{aligned}
    \label{eq:et-mean-fluctuation}
\end{equation}
Using the diffusion equation for $X_h$ and adding and subtracting the term \(\frac h2\,\{\nabla V(X_0)+\nabla V(X_h)\}\),
we obtain from \eqref{eq:et-quadrature-error} the pathwise identity
\begin{equation}
\begin{split}
    \mathcal E_h^{\mathrm{corr}}={}&\chi(X_0)+\mathcal E_h-\chi(X_h)\\
       &-\frac h2\,\{\nabla V(X_0+\chi(X_0))-\nabla V(X_0)
          +\nabla V(\overline X_h)-\nabla V(X_h)\}\,.
    \label{eq:et-corrected-displacement}
\end{split}
\end{equation}
The next two lemmas bound its total size and its conditional mean.

\begin{lemma}[Second moment of the corrected local error]
\label{lem:et-defect-second-moment}
There is a universal constant \(C>0\) such that, for
\(0<h\leqslant h_0\),
\[
    \E\norm{\mathcal E_h^{\mathrm{corr}}}^2
    \leqslant Cdh^3\,.
\]
\end{lemma}

\begin{proof}
Stationarity and \eqref{eq:et-corrector-L2} control the two corrector
terms in \eqref{eq:et-corrected-displacement}.  The quadrature term
satisfies, by \(1\)-Lipschitzness and
\eqref{eq:et-displacement-stationary},
\[
    \norm{\mathcal E_h}_{L^2}
    \leqslant\int_0^h\norm{X_t-X_0}_{L^2}\,\mathrm d t
       +\frac h2\,\norm{X_h-X_0}_{L^2}
    \leqslant C\sqrt{dh^3}\,.
\]
Synchronous contraction of the exact diffusion
\cite[Lemma~5.1.6(1)]{ChewiBook} gives
\begin{equation}
    \norm{\overline X_h-X_h}
    \leqslant\e^{-h/\kappa}\,\norm{\chi(X_0)}\,.
    \label{eq:et-synchronous-contraction}
\end{equation}
Consequently, each force increment on the second line of
\eqref{eq:et-corrected-displacement} is bounded by
\(\norm{\chi(X_0)}\).  The triangle inequality in \(L^2\) now gives
\(\norm{\mathcal E_h^{\mathrm{corr}}}_{L^2}\leqslant
(2+h)\,\norm\chi_{L^2(\pi)}+C\sqrt{dh^3}\), proving the claim.
\end{proof}

\begin{lemma}[Conditional mean bound]
\label{lem:et-mean-bound}
For every \(h>0\),
\[
    \norm{m}_{L^2(\pi)}
    \leqslant h\,\norm{\chi}_{L^2(\pi)}\,.
\]
\end{lemma}

\begin{proof}
Reversibility and the discrete Poisson equation give
\[
    \E[\chi(X_0)+\mathcal E_h-\chi(X_h)\mid X_h]
    =P_h\chi(X_h)+e_h(X_h)-\chi(X_h)=0\,.
\]
Thus \eqref{eq:et-corrected-displacement} implies
\[
    m(X_h)=-\frac h2\,\E\bigl[
       \nabla V(X_0+\chi(X_0))-\nabla V(X_0)
       +\nabla V(\overline X_h)-\nabla V(X_h)\mid X_h\bigr]\,.
\]
Both increments are bounded by \(\norm{\chi(X_0)}\), by
\(1\)-Lipschitzness and \eqref{eq:et-synchronous-contraction}.
Conditional Jensen's inequality proves the result.
\end{proof}

The conditional mean therefore has size
\(\norm m_{L^2(\pi)}\leqslant C\sqrt{dh^5}\).  To obtain the same
transport cost for the centered fluctuation \(R\) after Gaussian
smoothing, we bound its conditional Stein kernel.

For a vector field \(\varphi\), let \(\nabla\varphi\) have columns
\(\nabla\varphi_j\).  In this convention, a Stein kernel of a centered
random vector \(R\) is a matrix field \(\tau_R\) satisfying
\begin{equation}
    \E\langle R\,,\varphi(R)\rangle
    =\E\operatorname{tr}\{\tau_R(R)\nabla\varphi(R)^\top\}
    \label{eq:et-stein-identity}
\end{equation}
for smooth test fields with bounded derivative.  We use the classical
Malliavin representation of this kernel
\cite[Proposition~3.7]{NourdinPeccatiSwan2014}, together with a Schatten
norm estimate that preserves the desired dependence on \(d\).
Background and notation are in
\Cref{subsec:tb-stein,subsec:tb-malliavin,subsec:tb-schatten}.

\begin{lemma}[Malliavin--Stein kernel estimate]
\label{lem:et-malliavin-stein}
Let \(R\in L^2(\Omega;\R^d)\) be a centered, Malliavin-differentiable
random vector over an isonormal Gaussian probability space with
Cameron--Martin space \(\mathcal H\), and suppose that
\(D_GR\in L^4(\Omega;\mathfrak{S}_4)\).  Then the law of \(R\) has a
Stein kernel \(\tau_R\) satisfying
\[
    \E\norm{\tau_R}_{\mathrm{HS}}^2
    \leqslant \norm{D_GR}_{L^4(\Omega;\mathfrak{S}_4)}^4\,,
\]
where \(D_GR:\mathcal H\to\R^d\) denotes the Malliavin derivative of
\(R\).
\end{lemma}

\begin{proof}
The representation in \cite[Proposition~3.7]{NourdinPeccatiSwan2014},
transposed to match \eqref{eq:et-stein-identity}, gives
\[
    \tau_R(R)=\E\bigl[(D_GR)
       (-D_G\mathcal L_{\mathrm{OU}}^{-1}R)^*\mid R\bigr]\,.
\]
Here, \(\mathcal L_{\mathrm{OU}}\) is the Gaussian
Ornstein--Uhlenbeck generator and \( (\mathcal Q_t)_{t\ge 0} \) its Mehler
semigroup.  The pseudoinverse and gradient commutation formulas give
\begin{equation}
    -D_G\mathcal L_{\mathrm{OU}}^{-1}R
    =\int_0^\infty\e^{-t}\mathcal Q_tD_GR\,\mathrm d t\,.
    \label{eq:et-ou-gradient-representation}
\end{equation}
Jensen's inequality in the Mehler representation makes
\(\mathcal Q_t\) a contraction on
\(L^4(\Omega;\mathfrak S_4)\).  Minkowski's inequality in
\eqref{eq:et-ou-gradient-representation} therefore gives
\[
    \norm{-D_G\mathcal L_{\mathrm{OU}}^{-1}R}_{L^4(\Omega;\mathfrak S_4)}
    \leqslant\norm{D_GR}_{L^4(\Omega;\mathfrak S_4)}\,.
\]
Conditional Jensen, the Schatten inequality
\(\norm{AB^*}_{\mathfrak S_2}\leqslant
\norm A_{\mathfrak S_4}\,\norm B_{\mathfrak S_4}\), and
Cauchy--Schwarz imply
\begin{align*}
    \E\norm{\tau_R}_{\mathrm{HS}}^2
    &\leqslant\E\bigl[
       \norm{D_GR}_{\mathfrak S_4}^2\,
       \norm{-D_G\mathcal L_{\mathrm{OU}}^{-1}R}_{\mathfrak S_4}^2\bigr]\leqslant\norm{D_GR}_{L^4(\Omega;\mathfrak S_4)}^4\,. \qedhere
\end{align*}
\end{proof}

We now return to the fluctuation \(R\) defined in
\eqref{eq:et-mean-fluctuation}.  By
\Cref{lem:et-malliavin-stein}, the desired conditional Stein kernel
bound reduces to controlling \(D_GR\) in
\(L^4(\Omega;\mathfrak{S}_4)\).  For each fixed endpoint \(X_h=y\), reversibility
identifies the law of the reversed path with that of a Langevin diffusion
started from \(y\).  We realize this diffusion on canonical Wiener space and
write \(D_G\) for the Malliavin derivative with respect to its driving
Brownian path.  In this realization,
\(D_Gm(X_h)=0\).  Differentiating the representation
\eqref{eq:et-corrected-displacement} of \(\mathcal E_h^{\mathrm{corr}}\) introduces
\(\nabla\chi(X_0)\), so we first record the required regularity
estimate for the corrector, proved later in
\Cref{subsec:et-S4-multiplier}.

\begin{lemma}[Corrector derivative]
\label{lem:et-S4-corrector}
There is a universal constant \(C>0\) such that, for every \(h>0\), the
corrector in \eqref{eq:et-corrector} satisfies
\[
    \norm{\nabla\chi}_{L^4(\pi;\mathfrak{S}_4)}
    \leqslant Cd^{1/4} h\,.
\]
\end{lemma}

\begin{proposition}[Conditional \(\mathfrak{S}_4\) derivative]
\label{prop:et-S4-derivative}
There is a universal constant \(C>0\) such that, for
\(0<h\leqslant h_0\), the fluctuation \(R\) in
\eqref{eq:et-mean-fluctuation} satisfies
\begin{equation}
    \norm{D_GR}_{L^4(\Omega;\mathfrak{S}_4)}
    \leqslant Cd^{1/4} h^{3/2}\,.
    \label{eq:et-DR-S4}
\end{equation}
The norm averages over both the endpoint \(X_h\sim\pi\) and the
Brownian realization with that endpoint fixed.
Consequently, it has conditional Stein kernels \(\tau_{X_h}\) such that
\begin{equation}
    \E\norm{\tau_{X_h}}_{\mathrm{HS}}^2
    \leqslant Cdh^6\,.
    \label{eq:et-Stein-kernel-bound}
\end{equation}
\end{proposition}

\begin{proof}
Condition on \(X_h=y\).  By reversibility, the conditional law of
\((X_{h-t})_{0\leqslant t\leqslant h}\) is the law of a Langevin diffusion
started from \(y\), which we realize on canonical Wiener space.  Write
\(\mathcal H\deq L^2([0,h];\R^d)\), let \(G\) be the associated isonormal
Gaussian process over \(\mathcal H\), and define
\[
    \mathcal J_t\deq D_GX_{h-t}:\mathcal H\to\R^d\,,
    \qquad 0\leqslant t\leqslant h\,.
\]
Thus, \(\mathcal J_t\) is the Malliavin derivative of the reversed path at
time \(t\).  To identify it, fix \(f\in\mathcal H\) and shift
the canonical Brownian path in the Cameron--Martin direction
\(t\mapsto\int_0^t f_s\,\mathrm d s\).  Differentiating the resulting Langevin SDE
with respect to this shift shows that
\[
    \frac{\mathrm d}{\mathrm d t}\,(\mathcal J_tf)
    =-\nabla^2V(X_{h-t})\mathcal J_tf+\sqrt2\,f_t\,,
    \qquad \mathcal J_0f=0\,.
\]
For \(0\leqslant s\leqslant t\leqslant h\), let \(J_{t,s}\) be the
homogeneous evolution operator, characterized by
\[
    \frac{\mathrm d}{\mathrm d t}\,J_{t,s}=-\nabla^2V(X_{h-t})J_{t,s}\,,
    \qquad J_{s,s}=\Id\,.
\]
For fixed \(s\) and \(v\in\R^d\), convexity of \(V\) gives
\[
    \frac{\mathrm d}{\mathrm d t}\,\norm{J_{t,s}v}^2
    =-2\,\ip{J_{t,s}v}{\nabla^2V(X_{h-t})J_{t,s}v}
    \leqslant0\,.
\]
Since \(J_{s,s}=\Id\), it follows that
\(\norm{J_{t,s}v}\leqslant\norm v\), and hence
\(\norm{J_{t,s}}_{\mathfrak{S}_\infty}\leqslant1\).  For every
\(f\in\mathcal H\), variation of constants now yields
\[
    \mathcal J_tf=\sqrt2\int_0^tJ_{t,s}f_s\,\mathrm d s\,.
\]
Cauchy--Schwarz and the bound on \(J_{t,s}\) give
\(\norm{\mathcal J_t}_{\mathfrak{S}_\infty}\leqslant\sqrt{2t}\).  The
integral kernel representation gives
\[
    \norm{\mathcal J_t}_{\mathfrak{S}_2}^2
    =2\int_0^t\norm{J_{t,s}}_{\mathfrak{S}_2}^2\,\mathrm d s
    \leqslant2dt\,.
\]
In summary,
\begin{equation}
    \norm{\mathcal J_t}_{\mathfrak{S}_\infty}\leqslant\sqrt{2t}\,,
    \qquad
    \norm{\mathcal J_t}_{\mathfrak{S}_2}\leqslant\sqrt{2dt}\,.
    \label{eq:et-K-op-HS}
\end{equation}
H\"older's inequality for the singular values therefore gives
\begin{equation}
    \norm{\mathcal J_t}_{\mathfrak{S}_4}
    \leqslant\norm{\mathcal J_t}_{\mathfrak{S}_\infty}^{1/2}\,
       \norm{\mathcal J_t}_{\mathfrak{S}_2}^{1/2}
    \leqslant\sqrt{2t}\,d^{1/4}\,.
    \label{eq:et-K-S4}
\end{equation}

Because \(X_h=y\) is held fixed when differentiating with respect to the
driving Brownian path,
\(D_G\{m(X_h)\}=D_G\{\chi(X_h)\}=0\) as they depend only on this endpoint.  Differentiating
\eqref{eq:et-corrected-displacement} gives
\[
\begin{split}
    D_GR={}&D_G\{\chi(X_0)\}+D_G\mathcal E_h \\
           &\qquad{} -\frac h2\,D_G\{\nabla V(X_0+\chi(X_0))-\nabla V(X_0)\}
      -\frac h2\,D_G\{\nabla V(\overline X_h)-\nabla V(X_h)\}\,.
\end{split}
\]
We estimate these four terms in turn.  The Malliavin chain rule first gives
\[
    D_G\{\chi(X_0)\}
    =\nabla\chi(X_0)^\top D_GX_0
    =\nabla\chi(X_0)^\top\mathcal J_h\,.
\]
Using \eqref{eq:et-K-op-HS}, then averaging over the reverse noise
and the stationary endpoint, gives
\begin{equation}
    \norm{D_G\{\chi(X_0)\}}_{L^4(\Omega;\mathfrak S_4)}
    \leqslant\sqrt{2h}\,\norm{\nabla\chi}_{L^4(\pi;\mathfrak S_4)}
    \leqslant Cd^{1/4} h^{3/2}\,,
    \label{eq:et-u-conditional-derivative}
\end{equation}
where the last inequality is \Cref{lem:et-S4-corrector}.

We next bound \(D_G\mathcal E_h\).  Applying the Malliavin chain rule to the
formula for \(\mathcal E_h=\widetilde X_h-X_h\) in
\eqref{eq:et-quadrature-error} gives
\[
\begin{split}
    D_G\mathcal E_h
    ={}&\int_0^h\nabla^2V(X_t)\mathcal J_{h-t}\,\mathrm d t
       -\frac h2\,\nabla^2V(X_0)\mathcal J_h\,.
\end{split}
\]
This uses \(\mathcal J_0=D_GX_h=0\) because the endpoint \(X_h=y\) is held
fixed.  Using
\(\norm{\nabla^2V}_{\mathrm{op}}\leqslant1\), Minkowski's inequality,
and \eqref{eq:et-K-S4}, we now obtain
\[
\begin{split}
    \norm{D_G\mathcal E_h}_{L^4(\Omega;\mathfrak{S}_4)}
    &\leqslant \int_0^h\norm{\mathcal J_{h-t}}_{\mathfrak{S}_4}\,\mathrm d t
       +\frac h2\,\norm{\mathcal J_h}_{\mathfrak{S}_4}\leqslant\sqrt2\,d^{1/4}\,
       \Bigl\{\int_0^h\sqrt{h-t}\,\mathrm d t+\frac h2\,\sqrt h\Bigr\}\leqslant Cd^{1/4} h^{3/2}\,.
\end{split}
\]

For the term
\(D_G\{\nabla V(X_0+\chi(X_0))-\nabla V(X_0)\}\), the chain rule gives
\[
\begin{split}
    &D_G\{\nabla V(X_0+\chi(X_0))-\nabla V(X_0)\}=
       \bigl\{\nabla^2V(X_0+\chi(X_0))
          (\Id+\nabla\chi(X_0)^\top)-\nabla^2V(X_0)\bigr\}\mathcal J_h\,.
\end{split}
\]
Using \(\norm{\nabla^2V}_{\mathrm{op}}\leqslant1\),
\Cref{lem:et-S4-corrector}, \eqref{eq:et-K-op-HS}, and
\eqref{eq:et-K-S4}, we obtain
\[
\begin{split}
    &\frac h2\,\norm{D_G\{\nabla V(X_0+\chi(X_0))-\nabla V(X_0)\}}_{L^4(\Omega;\mathfrak{S}_4)}\leqslant
    Cd^{1/4} h^{3/2}\,.
\end{split}
\]

It remains to bound
\(D_G\{\nabla V(\overline X_h)-\nabla V(X_h)\}\).  Return to the
forward time coordinate, let \(\overline X_t\) be the diffusion started from
\(X_0+\chi(X_0)\) with the same Brownian path, and put
\(\eta_t\deq\overline X_t-X_t\).  Then \(D_GX_t=\mathcal J_{h-t}\) and adding and subtracting $\nabla^2 V(\overline X_t) \mathcal J_{h-t}$ gives
\[
\begin{split}
    \frac{\mathrm d}{\mathrm d t}\,D_G\eta_t
    ={}&-\nabla^2V(\overline X_t)D_G\eta_t-\{\nabla^2V(\overline X_t)-\nabla^2V(X_t)\}\mathcal J_{h-t}\,,
    \qquad
    D_G\eta_0=\nabla\chi(X_0)^\top\mathcal J_h\,.
\end{split}
\]
The fundamental solution of the homogeneous equation is a contraction.
Since
\(\norm{\nabla^2V(\overline X_t)-\nabla^2V(X_t)}_{\mathrm{op}}
\leqslant2\),
variation of constants, \Cref{lem:et-S4-corrector},
\eqref{eq:et-K-op-HS}, and \eqref{eq:et-K-S4} bound the initial
variation by \(Cd^{1/4} h^{3/2}\) and the forcing at time \(t\) by
\(Cd^{1/4} \sqrt{h-t}\). Integrating the latter gives
\[
\begin{split}
    \norm{D_G\eta_h}_{L^4(\Omega;\mathfrak{S}_4)}
    &\leqslant C\sqrt h\,\norm{\nabla\chi}_{L^4(\pi;\mathfrak{S}_4)}
    +2\int_0^h\norm{\mathcal J_{h-t}}_{\mathfrak{S}_4}\,\mathrm d t\leqslant Cd^{1/4} h^{3/2}\,.
\end{split}
\]
Because \(X_h\) is fixed under the conditioning,
\(D_G\overline X_h=D_G\eta_h\).  Thus,
\begin{equation}
\begin{split}
    &\frac h2\,\norm{D_G\{\nabla V(\overline X_h)-\nabla V(X_h)\}}_{L^4(\Omega;\mathfrak{S}_4)}\leqslant C d^{1/4} h^{5/2}\,.
\end{split}
\label{eq:et-terminal-force-S4-derivative}
\end{equation}
Combining \eqref{eq:et-u-conditional-derivative}--\eqref{eq:et-terminal-force-S4-derivative} and using \(h\leqslant h_0\leqslant1\) proves
\eqref{eq:et-DR-S4}.

For $\pi$-almost every endpoint \(y\), \Cref{lem:et-malliavin-stein} gives a Stein kernel
\(\tau_y\) satisfying
\[
    \E\bigl[\norm{\tau_y}_{\mathrm{HS}}^2\mid X_h=y\bigr]
    \leqslant
    \E\bigl[\norm{D_GR}_{\mathfrak{S}_4}^4
       \mid X_h=y\bigr]\,.
\]
Since \(X_h\sim\pi\), integrating this inequality over the endpoint and
using the tower property gives
\[
\begin{split}
    \E\norm{\tau_{X_h}}_{\mathrm{HS}}^2
    &\leqslant
       \E\norm{D_GR}_{\mathfrak{S}_4}^4=\norm{D_GR}_{L^4(\Omega;\mathfrak{S}_4)}^4\leqslant
       \bigl(Cd^{1/4} h^{3/2}\bigr)^4
     =C^4dh^6\,,
\end{split}
\]
where the inequality is \eqref{eq:et-DR-S4}.  Renaming the universal
constant proves \eqref{eq:et-Stein-kernel-bound}.
\end{proof}

\subsection{A Schatten-\texorpdfstring{$4$}{4} multiplier estimate for the Witten Laplacian}
\label{subsec:et-S4-multiplier}

The remaining input to \Cref{prop:et-S4-derivative} is
\(\norm{\nabla\chi}_{L^4(\pi;\mathfrak S_4)}\lesssim d^{1/4} h\).
The Witten semigroup describes the evolution of spatial gradients under
Langevin dynamics; \Cref{subsec:tb-witten,subsec:tb-interpolation} recall
the background. The argument has three steps.
\Cref{lem:et-witten-endpoints} gives real time contraction on
\(L^\infty(\pi;\mathfrak S_\infty)\) and complex time contraction on
\(L^2(\pi;\mathfrak S_2)\).
\Cref{lem:et-witten-analyticity} interpolates these estimates to
\(L^4(\pi;\mathfrak S_4)\) and bounds the resolvent there.
\Cref{lem:et-S4-multiplier} then applies the general multiplier estimate of \Cref{lem:tb-sectorial-multiplier} to bound the
multiplier in \(\nabla\chi=h\,r(-hL^{(1)})\,\nabla^2V\).
The exponent \(4\) is needed because the Stein kernel is a product of two
operators in \(\mathfrak S_4\), and therefore belongs to \(\mathfrak S_2\).

We use the $1$-form generator and its semigroup
\[
    L^{(1)}\deq L-\nabla^2V\,,\qquad
    L=\Delta-\nabla V\cdot\nabla\,,\qquad P_t^{(1)}\deq\e^{tL^{(1)}}\,,
\]
acting columnwise on matrix fields.  The Feynman--Kac representation
\cite[Proposition~3.2, equation~(3.2)]{Huguet2021} is
\[
    (P_t^{(1)}M)(x)
    =\E\bigl[(J_t^x)^\top M(X_t^x)\bigr]\,,\qquad
    J_t^x\deq D_xX_t^x\,.
\]
Here, \(J_t^x\) is the Jacobian with respect to the starting point.
It satisfies \(\partial_tJ_t^x=-\nabla^2V(X_t^x)J_t^x\), hence
\(\norm{J_t^x}_{\mathrm{op}}\leqslant\e^{-t/\kappa}\).
\Cref{subsec:tb-witten} motivates the generator and transpose via the
scalar chain rule and explains how the cited representation applies.
We now record the resulting norm estimates.

\begin{lemma}[Contraction estimates for the Witten semigroup]
\label{lem:et-witten-endpoints}
For every \(t\geqslant0\) and \(2\leqslant p\leqslant\infty\),
\begin{equation}
    \norm{P_t^{(1)}M}_{L^p(\pi;\mathfrak{S}_p)}
    \leqslant \e^{-t/\kappa}\,\norm M_{L^p(\pi;\mathfrak{S}_p)}\,.
    \label{eq:et-witten-real-contraction}
\end{equation}
On \(L^2(\pi;\mathfrak{S}_2)\), \(-L^{(1)}\) is positive self-adjoint and
\begin{equation}
\begin{split}
    \ip M{-L^{(1)}M}_{L^2(\pi;\mathfrak{S}_2)}
    ={}&\int\norm{\nabla M}_{\mathrm{HS}}^2\,\mathrm d \pi+\int\sum_{j\in[d]}
          \ip{M_j}{\nabla^2V\,M_j}\,\mathrm d \pi\,.
\end{split}
\label{eq:et-witten-quadratic-form}
\end{equation}
Consequently,
\begin{equation}
    \norm{\e^{zL^{(1)}}M}_{L^2(\pi;\mathfrak{S}_2)}
    \leqslant\norm M_{L^2(\pi;\mathfrak{S}_2)}
    \qquad\text{whenever }\operatorname{Re}z\geqslant0\,.
    \label{eq:et-witten-L2-complex}
\end{equation}
\end{lemma}

\begin{proof}
The Feynman--Kac formula and the ideal property of Schatten norms give
\[
    \norm{P_t^{(1)}M(x)}_{\mathfrak S_p}
    \leqslant\e^{-t/\kappa}\,\E\norm{M(X_t^x)}_{\mathfrak S_p}\,.
\]
Jensen's inequality and stationarity give
\eqref{eq:et-witten-real-contraction} for finite \(p\); the essential
supremum gives the case \(p=\infty\).  Integration by parts proves
\eqref{eq:et-witten-quadratic-form} on smooth compactly supported
fields.  Closing this symmetric form defines the non-negative
self-adjoint operator \(-L^{(1)}\), and the spectral theorem proves
\eqref{eq:et-witten-L2-complex}.
\end{proof}

The preceding lemma supplies the contraction estimates. We now interpolate
its \(p=\infty\) and \(p=2\) bounds to obtain the
complex time estimate needed for the multiplier.  The interpolation
and sectorial calculus facts used here are recalled in
\Cref{subsec:tb-interpolation}.

\begin{lemma}[Dimension-free analyticity on \(L^4(\pi;\mathfrak{S}_4)\)]
\label{lem:et-witten-analyticity}
The semigroup \(\e^{zL^{(1)}}\) extends to a contractive analytic semigroup on
\(L^4(\pi;\mathfrak{S}_4)\) in the sector \(\lvert \arg z \rvert<\uppi/4\).  In
particular, \(-L^{(1)}\) is sectorial on \(L^4(\pi;\mathfrak{S}_4)\):
its spectrum is contained in the closed sector
\(\{\lvert\arg z\rvert\leqslant\uppi/4\}\), and for each fixed
\(\vartheta\in(\uppi/4,\uppi/2)\) there is a constant
\(C_\vartheta\) such that
\begin{equation}
    \norm{(z+L^{(1)})^{-1}}_{L^4(\pi;\mathfrak{S}_4)\to L^4(\pi;\mathfrak{S}_4)}
    \leqslant\frac{C_\vartheta}{\lvert z\rvert}
    \qquad\text{whenever }\lvert\arg z\rvert\geqslant\vartheta\,.
    \label{eq:et-witten-resolvent}
\end{equation}
\end{lemma}

\begin{proof}
The interpolation identity for Bochner--Schatten spaces
(\Cref{subsec:tb-interpolation}) gives
\begin{equation}
    [L^\infty(\pi;\mathfrak{S}_\infty),L^2(\pi;\mathfrak{S}_2)]_{1/2}
    =L^4(\pi;\mathfrak{S}_4)\,,
    \label{eq:et-matrix-interpolation}
\end{equation}
isometrically~\cite{BerghLofstrom1976}.

Fix \(\lvert\theta\rvert<\uppi/4\) and, for \(t>0\), consider on the strip
\(0\leqslant\operatorname{Re}\zeta\leqslant1\) the analytic family
\[
\mathcal F(\zeta)\deq\exp(t\e^{2\boldsymbol i\theta\zeta}L^{(1)})\,.
\]
If \(\zeta=\boldsymbol iy\), then \(t\e^{2\boldsymbol i\theta\zeta}=t\e^{-2\theta y}>0\), so
\Cref{lem:et-witten-endpoints} shows that \(\mathcal F(\boldsymbol iy)\) is a
contraction on \(L^\infty(\pi;\mathfrak{S}_\infty)\).
If \(\zeta=1+\boldsymbol iy\), then
\(t\e^{-2\theta y}\e^{2\boldsymbol i\theta}\) has positive real part, so
\eqref{eq:et-witten-L2-complex} shows that
\(\mathcal F(1+\boldsymbol iy)\) is a contraction on
\(L^2(\pi;\mathfrak{S}_2)\).
Stein's interpolation theorem~\cite{Stein1956} at \(\zeta=1/2\) and
\eqref{eq:et-matrix-interpolation} yield
\[
    \norm{\exp(t\e^{\boldsymbol i\theta}L^{(1)})}_{L^4(\pi;\mathfrak{S}_4)\to L^4(\pi;\mathfrak{S}_4)}
    \leqslant1\,.
\]
The argument is first applied on the dense compatible core
\(L^\infty(\pi;\mathfrak{S}_\infty)\cap L^2(\pi;\mathfrak{S}_2)\) and then
extended to \(L^4(\pi;\mathfrak{S}_4)\).  Strong continuity on
\(L^4(\pi;\mathfrak{S}_4)\) follows from the same density argument.
This proves the first assertion.

The resolvent of \(-L^{(1)}\) is the Laplace transform of the
semigroup, and analyticity lets the integration ray be rotated by any
angle below \(\uppi/4\), as in \eqref{eq:tb-rotated-laplace}.  Choosing
the rotation as in \Cref{subsec:tb-interpolation} gives
\eqref{eq:et-witten-resolvent} with
\(C_\vartheta=1/\sin((\vartheta-\uppi/4)/2)\).
\end{proof}

Sectoriality puts the contour calculus of
\Cref{subsec:tb-interpolation} at our disposal, and with it the
multiplier bound, uniform in the step size; the proof reduces to
checking the hypotheses of \Cref{lem:tb-sectorial-multiplier}.

\begin{lemma}[Uniform \(\mathfrak{S}_4\) Witten--Laplacian multiplier]
\label{lem:et-S4-multiplier}
There is a universal constant \(C>0\) such that for every $h > 0$,
\begin{equation}
    \norm{r(-hL^{(1)})}_{L^4(\pi;\mathfrak{S}_4)\to L^4(\pi;\mathfrak{S}_4)}
    \leqslant C\,.
    \label{eq:et-S4-multiplier}
\end{equation}
\end{lemma}

\begin{proof}
We apply \Cref{lem:tb-sectorial-multiplier} on
\(\mathcal X=L^4(\pi;\mathfrak S_4)\) with \(\mathcal T=-L^{(1)}\) and
\(f=r\), which also defines \(r(-hL^{(1)})\) on this space.  By
\Cref{lem:et-witten-analyticity}, \(-L^{(1)}\) is sectorial of angle
\(\uppi/4\), with constants depending only on the angle.
Fix \(\vartheta\deq3\uppi/8\). It remains to verify the hypotheses
on \(r\) in the sector
\(\Sigma_\vartheta\deq\{z\in\mathbb C\setminus\{0\}:\lvert\arg z\rvert<\vartheta\}\),
recalled in \Cref{subsec:tb-interpolation}, with \(r(0)=0\), \(r(\infty)=-1/2\), and
\(s=1\).

\emph{Holomorphy.}  The poles of \(\coth(z/2)\) are at
\(2\uppi\boldsymbol ik\), \(k\in\mathbb Z\).  The pole at zero cancels
with \(z^{-1}\), since \(\tfrac12\coth(z/2)=z^{-1}+z/12+O(z^3)\); the
other poles lie on the imaginary axis, outside \(\Sigma_\vartheta\).
Thus \(r\) is holomorphic in \(\Sigma_\vartheta\) and in the disc
\(\{\lvert z\rvert<2\uppi\}\).

\emph{Behavior at zero.}  Since \(r\) is holomorphic in
\(\{\lvert z\rvert<2\uppi\}\) with \(r(0)=0\), there is a universal
constant \(C\) such that \(\lvert r(z)\rvert\leqslant C\,\lvert z\rvert\)
for \(\lvert z\rvert\leqslant1\).

\emph{Behavior at infinity.}  The identity
\(\coth(z/2)=1+2/(\e^z-1)\) gives
\[
    r(z)+\frac12=\frac1z-\frac1{\e^z-1}\,.
\]
For \(z\in\Sigma_\vartheta\), we have
\(\operatorname{Re}z\geqslant\lvert z\rvert\cos\vartheta\), so
\(\lvert\e^z-1\rvert\geqslant\e^{\operatorname{Re}z}-1
\geqslant\lvert z\rvert\cos\vartheta\).  Hence
\(\lvert r(z)+1/2\rvert\leqslant(1+1/\cos\vartheta)/\lvert z\rvert\).

\Cref{lem:tb-sectorial-multiplier} now bounds
\(\norm{r(-hL^{(1)})}_{L^4(\pi;\mathfrak S_4)\to L^4(\pi;\mathfrak S_4)}\)
uniformly in \(h>0\) by a universal constant, which proves
\eqref{eq:et-S4-multiplier}.
\end{proof}

We can now differentiate the corrector and prove the deferred estimate.

\begin{proof}[Proof of \Cref{lem:et-S4-corrector}]
Gradient intertwining gives \(\nabla P_t=P_t^{(1)}\nabla\).  Differentiate
the finite partial sums in \eqref{eq:et-corrector-series} and use
\eqref{eq:et-witten-real-contraction} to pass to the limit in
\(L^4(\pi;\mathfrak{S}_4)\).  Since $\nabla\chi=h\,r(-hL^{(1)})\,\nabla^2V$,
we can apply \Cref{lem:et-S4-multiplier} and the pointwise bound
\(\norm{\nabla^2V}_{\mathfrak{S}_4}\leqslant d^{1/4}\).
\end{proof}

\subsection{\texorpdfstring{\(\W_2\)}{W\_2} estimates after Gaussian convolution}
\label{subsec:et-transport}

We next convert the conditional Stein kernel bound into a transport
estimate.  An independent Gaussian convolution bound of this
form follows from \cite[Theorem~E.4 and Remark~E.6(ii)]{PaulinWhalley2026}.
The following direct proof sharpens its numerical constant.
Its hypotheses require only a centered perturbation with a
square-integrable Stein kernel; no density or Poincar\'e inequality for
the perturbation is needed.

\begin{lemma}[Gaussian convolution inequality]
\label{lem:et-gaussian-convolution}
Let \(R\in L^2(\Omega;\R^d)\) be centered, suppose that its law
admits a square-integrable Stein kernel \(\tau_R\), and let
\(\xi\sim\cN(0,\sigma^2\Id_d)\) be independent of \(R\).  Then,
\[
    \W_2^2\bigl(\Law(R+\xi),\Law(\xi)\bigr)
    \leqslant{\frac1{4\sigma^2}}\,\E\norm{\tau_R}_{\mathrm{HS}}^2\,.
\]
\end{lemma}

\begin{proof}
Set \(\mathsf X_\theta\deq\xi+\theta R\), for
\(0\leqslant\theta\leqslant1\).
For \(\varphi\in C_c^\infty(\R^d)\), the Stein identity
\eqref{eq:et-stein-identity}, followed by Gaussian integration by parts,
gives
\[
\begin{aligned}
    \frac{\mathrm d}{\mathrm d\theta}\,\E\varphi(\mathsf X_\theta)
    &=\E\ip R{\nabla\varphi(\mathsf X_\theta)}
      =\theta\,\E\operatorname{tr}\{\tau_R\nabla^2\varphi(\mathsf X_\theta)\}\\
    &=\frac\theta{\sigma^2}\,\E\ip{\tau_R\,\xi}{\nabla\varphi(\mathsf X_\theta)}\,.
\end{aligned}
\]
Here the first integration by parts holds with \(\xi\) fixed, and the
second with \(R\) fixed.  Thus
\(v_\theta(\mathsf X_\theta)\deq\frac\theta{\sigma^2}\,\E[\tau_R\,\xi\mid\mathsf X_\theta]\)
is a velocity for the continuity equation of \(\Law(\mathsf X_\theta)\).
Conditional Jensen's inequality and independence imply
\[
    \norm{v_\theta}_{L^2(\Law(\mathsf X_\theta))}^2
    \leqslant\frac{\theta^2}{\sigma^4}\,\E\norm{\tau_R\,\xi}^2
    =\frac{\theta^2}{\sigma^2}\,\E\norm{\tau_R}_{\mathrm{HS}}^2\,.
\]
The metric speed bound
\cite[Theorem~8.3.1]{AmbrosioGigliSavare2008} therefore yields
\[
    \W_2\bigl(\Law(R+\xi),\Law(\xi)\bigr)
    \leqslant\int_0^1\norm{v_\theta}_{L^2(\Law(\mathsf X_\theta))}\,\mathrm d\theta
    \leqslant\frac1{2\sigma}\,
       \bigl(\E\norm{\tau_R}_{\mathrm{HS}}^2\bigr)^{1/2}\,.\qedhere
\]
\end{proof}

\begin{remark}[Comparison with Paulin--Whalley]
Theorem~E.4 and Remark~E.6(ii) of \cite{PaulinWhalley2026} give the
same squared Wasserstein bound with coefficient \( (\log2)/\sigma^2\),
by combining a relative entropy estimate with Gaussian Talagrand.
The direct transport proof gives \(1/(4\sigma^2)\) under the same
square-integrable Stein kernel assumption; their theorem additionally
controls relative entropy.  The numerical constant \(1/4\) is sharp:
in dimension one, for independent \(R\sim\cN(0,\varepsilon)\) and
\(\xi\sim\cN(0,\sigma^2)\), one may take \(\tau_R=\varepsilon\), and
\[
    \W_2^2\bigl(\Law(R+\xi),\Law(\xi)\bigr)
    =\bigl(\sqrt{\sigma^2+\varepsilon}-\sigma\bigr)^2
    \sim\frac{\varepsilon^2}{4\sigma^2}\,,
    \qquad \varepsilon\downarrow0\,.
\]
\end{remark}

The next lemma carries this Gaussian convolution estimate through one
full auxiliary step.  Its additional term pays for the change in the
two force evaluations caused by perturbing the initial point.
For the model update \(x\mapsto x+\sqrt{2h}\,G\), the comparison is simply
between \(z+R+\sqrt{2h}\,G\) and \(z+\sqrt{2h}\,G\), so the preceding
lemma gives the Stein kernel cost divided by \(h\). Restoring the drift
adds an \(O(h\,\norm R_2)\) cost.

\begin{lemma}[\(\W_2\) estimate for one auxiliary step]
\label{lem:et-Gaussian-convolution-transport}
There is a universal constant \(C>0\) with the following property.
Let \(0<h\leqslant h_0\) and \(z\in\R^d\), and let
\(R\in L^2(\Omega;\R^d)\) be centered and independent of the Brownian
motion driving the auxiliary step \eqref{eq:et-exact-reference}, with
a law admitting a square-integrable Stein kernel \(\tau_R\).  Then,
\begin{equation}
    \W_2^2\bigl(\Law(\widetilde X_h^{z+R}),
       \Law(\widetilde X_h^{z})\bigr)
    \leqslant C\,\Bigl\{
       \frac{\E\norm{\tau_R}_{\mathrm{HS}}^2}{h}
       +h^2\,\E\norm R^2
      \Bigr\}\,.
\label{eq:et-Gaussian-convolution-transport}
\end{equation}
The same estimate holds conditionally on a preceding \(\sigma\)-field
when \(z\) is measurable with respect to it and the Brownian motion of
the auxiliary step is fresh and independent of that field and \(R\).
\end{lemma}

\begin{proof}
Let the auxiliary step from \(z+R\) be driven by a Brownian motion
\(B\) independent of \(R\), and the step from \(z\) by a second
Brownian motion \(\widetilde B\); the joint law of
\((R,B,\widetilde B)\) is constructed as follows.  Take an optimal
\(\W_2\) coupling of the endpoint variables \(R+\sqrt2B_h\) and
\(\sqrt2\widetilde B_h\); by \Cref{lem:et-gaussian-convolution},
applied with \(\sigma^2=2h\), its cost satisfies
\begin{equation}
    \E\norm{\delta}^2
    \leqslant{\frac1{8h}}\,\E\norm{\tau_R}_{\mathrm{HS}}^2\,,
    \qquad
    \delta\deq R+\sqrt2B_h-\sqrt2\widetilde B_h\,.
    \label{eq:et-endpoint-coupling}
\end{equation}
First sample \((R,B_h)\) conditionally on the coupled value of
\(R+\sqrt2B_h\).  This gives a joint law of
\((R,B_h,\widetilde B_h)\) with the required marginals.
To construct the full paths, take a standard Brownian bridge
\((b_t)_{0\leqslant t\leqslant h}\), independent of this triple, and set
\[
    B_t\deq\frac th\, B_h+b_t\,,\qquad
    \widetilde B_t\deq\frac th\,\widetilde B_h+b_t\,,
    \qquad 0\leqslant t\leqslant h\,.
\]
Each path is marginally a Brownian motion.  Their common bridge cancels
in the difference, giving
\(B_t-\widetilde B_t=\frac th\,(B_h-\widetilde B_h)\).
Thus, their drift-free parts differ by
\[
    R+\sqrt2B_t-\sqrt2\widetilde B_t
    =\Bigl(1-\frac th\Bigr)\,R+\frac th\,\delta\,,
\]
whose norm never exceeds \(\norm R+\norm{\delta}\).  By
\(1\)-Lipschitzness of \(\nabla V\) and Gr\"{o}nwall's inequality, the
diffusions \(X_t^{z+R}\) and \(X_t^{z}\), driven by \(B\) and by
\(\widetilde B\) respectively, therefore satisfy, for
\(0\leqslant t\leqslant h\leqslant h_0\),
\[
    \norm{X_t^{z+R}-X_t^{z}}
    \leqslant C\,\bigl(\norm R+\norm{\delta}\bigr)\,.
\]
The outputs of the two auxiliary steps differ by
\[
    \widetilde X_h^{z+R}-\widetilde X_h^{z}
    =\delta
    -\frac h2\,\bigl\{\nabla V(z+R)-\nabla V(z)\bigr\}
    -\frac h2\,\bigl\{\nabla V(X_h^{z+R})-\nabla V(X_h^{z})\bigr\}\,,
\]
and the two preceding bounds, with \(h\leqslant1\), give
\[
    \E\norm{\widetilde X_h^{z+R}-\widetilde X_h^{z}}^2
    \leqslant C\,\bigl\{\E\norm{\delta}^2+h^2\,\E\norm R^2\bigr\}\,.
\]
Together with \eqref{eq:et-endpoint-coupling}, this proves
\eqref{eq:et-Gaussian-convolution-transport}.  Conditional
disintegration proves the final assertion.
\end{proof}

\subsection{Proof of the invariant distribution bound}
\label{subsec:et-assembly}

We apply \Cref{lem:et-local-framework} to the shifted distribution
\(\nu\) and transfer the required estimates from \(\widetilde K_h\)
to the implementable kernel.

\begin{proposition}[Invariant distribution bias]
\label{prop:et-invariant-bias}
There is a universal constant \(C>0\) such that, for
\(0<h\leqslant h_0\),
\begin{equation}
    \kappa^{-1}\W_2^2(\widehat\pi_h^{\mathrm{Heun}},\pi)
    \leqslant C\kappa dh^3\,.
    \label{eq:et-direct-invariant-bound}
\end{equation}
\end{proposition}

\begin{proof}
Fix \(0<h\leqslant h_0\) and apply
\Cref{lem:et-local-framework} with the reference law
\(\nu=(\Id+\chi)_\#\pi\) from \eqref{eq:et-corrected-law}.  We verify the
two distances on the right-hand side of
\eqref{eq:et-local-framework}.

The coupling \eqref{eq:et-mean-fluctuation} and
\Cref{lem:et-defect-second-moment} immediately give
\begin{equation}
    \W_2^2(\nu\widetilde K_h,\nu)\leqslant Cdh^3\,.
    \label{eq:et-corrected-one-step-cost}
\end{equation}
For the distance between consecutive iterates, apply a fresh auxiliary
step to the coupled starting points
\[
    X_h+\chi(X_h)\sim\nu\,,
    \qquad
    X_h+\chi(X_h)+m(X_h)+R\sim\nu\widetilde K_h\,.
\]
Condition on \(X_h\).  Removing the mean \(m(X_h)\) by synchronous
coupling costs at most \(C\,\norm m_{L^2(\pi)}^2\leqslant Cdh^5\)
in squared distance, by \Cref{lem:et-mean-bound} and
\eqref{eq:et-corrector-L2}.  For the centered fluctuation \(R\),
\Cref{lem:et-Gaussian-convolution-transport} and
\eqref{eq:et-Stein-kernel-bound} give, after averaging over the endpoint,
\[
    C\,\bigl\{h^{-1}\,\E\norm{\tau_{X_h}}_{\mathrm{HS}}^2
       +h^2\,\E\norm R^2\bigr\}\leqslant Cdh^5\,.
\]
Here \(\E\norm R^2\leqslant\E\norm{\mathcal E_h^{\mathrm{corr}}}^2\leqslant Cdh^3\)
by conditional centering and \Cref{lem:et-defect-second-moment}.
The triangle inequality therefore yields
\begin{equation}
    \W_2^2(\nu\widetilde K_h^2,\nu\widetilde K_h)
    \leqslant Cdh^5\,.
    \label{eq:et-lag-one-cost}
\end{equation}

We next replace \(\widetilde K_h\) by the Heun kernel
\(K_h^{\mathrm{Heun}}\) defined by
\eqref{eq:et-predictor}--\eqref{eq:et-step}.  From a starting point \(x\),
drive the exact diffusion and the auxiliary step by the same Brownian path,
and use its endpoint \(B_h\) in the Heun step. Recall from
\eqref{eq:et-exact-reference} that
\(\widetilde X_h^x=x-\tfrac h2\,\{\nabla V(x)+\nabla V(X_h^x)\}+\sqrt2B_h\).
Thus
\[
    X^+-\widetilde X_h^x
    =-\frac h2\,\bigl\{\nabla V(X_h^{\mathrm{pred}})
       -\nabla V(X_h^x)\bigr\}\,.
\]
Moreover,
\[
    X_h^{\mathrm{pred}}-X_h^x
    =\int_0^h\bigl\{\nabla V(X_t^x)-\nabla V(x)\bigr\}\,\mathrm d t\,.
\]
Jensen's inequality, \(1\)-Lipschitzness of \(\nabla V\), and
\eqref{eq:et-displacement-pointwise} therefore give
\begin{equation}
    \E\norm{X^+-\widetilde X_h^x}^2
    \leqslant C\,\bigl\{dh^5+h^6\,\norm{\nabla V(x)}^2\bigr\}\,.
    \label{eq:et-implementation-local}
\end{equation}
Since \(\nu=(\Id+\chi)_\#\pi\), integration by parts and
\eqref{eq:et-corrector-L2} imply
\[
\begin{split}
    \int\norm{\nabla V}^2\,\mathrm d \nu
    &=\int\norm{\nabla V(x+\chi(x))}^2\,\pi(\mathrm d x)\leqslant2\int\norm{\nabla V}^2\,\mathrm d \pi
       +2\,\norm{\chi}_{L^2(\pi)}^2
     \leqslant Cd\,.
\end{split}
\]
The definition of \(\widetilde X_h^x\) and
\eqref{eq:et-displacement-pointwise} then give
\[
    \int\E\norm{\widetilde X_h^x-x}^2\,\nu(\mathrm d x)
    \leqslant C\,\Bigl\{dh+h^2\int\norm{\nabla V}^2\,\mathrm d \nu\Bigr\}
    \leqslant Cdh\,.
\]
Using \(1\)-Lipschitzness once more,
\begin{align*}
    \int\norm{\nabla V}^2\,\mathrm d(\nu\widetilde K_h)
    &\leqslant2\int\norm{\nabla V}^2\,\mathrm d \nu
       +2\int\E\norm{\widetilde X_h^x-x}^2\,\nu(\mathrm d x)
    \leqslant Cd\,.
\end{align*}
Integrating \eqref{eq:et-implementation-local} first against \(\nu\) and
then against \(\nu\widetilde K_h\) proves
\begin{equation}
\begin{split}
    \W_2^2(\nu K_h^{\mathrm{Heun}},\nu\widetilde K_h)
    &\leqslant Cdh^5\,,\qquad
    \W_2^2((\nu\widetilde K_h)K_h^{\mathrm{Heun}},
       (\nu\widetilde K_h)\widetilde K_h)
    \leqslant Cdh^5\,.
\end{split}
    \label{eq:et-kernel-comparison-costs}
\end{equation}

The triangle inequality and
\eqref{eq:et-corrected-one-step-cost} transfer the first cost to
\(K_h^{\mathrm{Heun}}\).  For the second cost, insert
\((\nu\widetilde K_h)K_h^{\mathrm{Heun}}\),
\(\nu\widetilde K_h^2\), and \(\nu\widetilde K_h\) between
\(\nu(K_h^{\mathrm{Heun}})^2\) and \(\nu K_h^{\mathrm{Heun}}\).
Contractivity bounds the first inserted distance by
\(\rho\W_2(\nu K_h^{\mathrm{Heun}},\nu\widetilde K_h)\).
Thus \eqref{eq:et-lag-one-cost} and
\eqref{eq:et-kernel-comparison-costs} prove
\[
\begin{split}
    \W_2(\nu K_h^{\mathrm{Heun}},\nu)&\leqslant C\sqrt{dh^3}\,, \qquad
    \W_2(\nu(K_h^{\mathrm{Heun}})^2,\nu K_h^{\mathrm{Heun}})
       \leqslant C\sqrt{dh^5}\,.
\end{split}
\]

Substitute these estimates into \Cref{lem:et-local-framework} and use
\(1-\rho\geqslant h/(2\kappa)\):
\[
    \W_2(\widehat\pi_h^{\mathrm{Heun}},\nu)
    \leqslant C\sqrt{dh^3}+C\kappa h^{-1}\sqrt{dh^5}
    \leqslant C\kappa\sqrt{dh^3}\,.
\]
Finally, \(\W_2(\nu,\pi)\leqslant\norm\chi_{L^2(\pi)}
\leqslant C\sqrt{dh^3}\), so the triangle inequality proves
\eqref{eq:et-direct-invariant-bound}.
\end{proof}

\subsection{Completion of the overdamped upper bound}
\label{subsec:et-upper-completion}

\begin{proof}[Proof of \Cref{thm:et-upper}]
The normalization in \Cref{sec:normalization} turns
\eqref{eq:et-direct-invariant-bound} into
\eqref{eq:et-stationary-upper}.  Existence and uniqueness follow from
\Cref{prop:et-contractivity} and its rescaling.

Contractivity and invariance give, for a universal \(c_0>0\),
\[
    \W_2(\mu(K_h^{\mathrm{Heun}})^N,
       \widehat\pi_h^{\mathrm{Heun}})
    \leqslant \e^{-c_0N\alpha h}
       \W_2(\mu,\widehat\pi_h^{\mathrm{Heun}})\,.
\]
Consequently,
\[
\begin{split}
    \W_2(\mu(K_h^{\mathrm{Heun}})^N,\pi)
    &\leqslant \e^{-c_0N\alpha h}\,
       \bigl\{\W_2(\mu,\pi)
          +\W_2(\pi,\widehat\pi_h^{\mathrm{Heun}})\bigr\}+\W_2(\widehat\pi_h^{\mathrm{Heun}},\pi)\,.
\end{split}
\]
Squaring this inequality and applying
\eqref{eq:et-stationary-upper} proves
\eqref{eq:et-finite-time-upper}, after changing the universal constants.

Choose a sufficiently small
universal \(c_1>0\) and set
\[
    \beta h=c_1h_0\min\Bigl\{1,
       \Bigl(\frac{\varepsilon^2}{\kappa d}\Bigr)^{1/3}\Bigr\}\,,
    \qquad
    N=\Bigl\lceil C\kappa\,(\beta h)^{-1}
       \log\Bigl(\e+\frac{\alpha\W_2^2(\mu,\pi)}{\varepsilon^2}\Bigr)\Bigr\rceil\,.
\]
The two terms in \eqref{eq:et-finite-time-upper} are then each at most
\(\varepsilon^2/2\), proving the query bound for arbitrary
initialization.  If \(\mu=\delta_{x_\star}\), strong convexity and
integration by parts give
\[
    \alpha\,\E\norm{X-x_\star}^2
    \leqslant\E\ip{\nabla V(X)}{X-x_\star}=d\,,
    \qquad X\sim\pi\,.
\]
For \(0<\varepsilon\leqslant\sqrt d/2\), we have \(d/\varepsilon^2\geqslant4\),
and hence \(\log(\e+d/\varepsilon^2)\leqslant2\log(d/\varepsilon^2)\).
Since
\(\kappa\,(\beta h)^{-1}\lesssim
\kappa+\kappa^{4/3}d^{1/3}\varepsilon^{-2/3}\),
this gives \eqref{eq:et-complexity}.

Each iteration uses two gradient queries, which is absorbed in the
universal constant.
\end{proof}

\section{Proof of the underdamped upper bound}
\label{app:underdamped-proof}

We prove the ingredients and assemble the argument outlined in
\Cref{sec:underdamped-upper}.  Throughout this appendix, except
\Cref{subsec:uld-completion}, the Langevin dynamics are normalized as in
\Cref{sec:normalization}: \(V\in C^2\),
\(\kappa^{-1}\Id_d\preceq\nabla^2V\preceq\Id_d\), and \(\gamma=2\).
Unless another measure is specified, \(L^2\) norms are under the
stationary phase space distribution, with auxiliary Gaussian noise
included for quantities depending on the path.  We use \(\mathrm{HS}\)
and \(\mathrm{op}\) for Hilbert--Schmidt and operator norms,
respectively.

The goal is a one step Wasserstein error of \(C\sqrt d\,h^3\) from
stationarity, which contraction turns into the second-order invariant
distribution bound. The essential plan is to separate the error into a conditional mean and a centered remainder portion. The former would dominate the error when analyzed in $L^2$. However, if one works instead in the negative Sobolev norm, then one witnesses additional cancellations via careful Malliavin estimates, which can then be translated into Wasserstein guarantees via a new heat interpolation lemma; this is combined with standard estimates for the remainder in order to conclude. 

More concretely, we proceed in four steps, which roughly follow the sequential flow of this appendix.
\begin{enumerate}
\item \emph{Conditional mean and centered fluctuation.}
For a stationary exact endpoint \(Z_h\sim\bm\pi\),
\Cref{subsec:uld-defect} decomposes the exact midpoint local error as
\(\mathcal E_h\eqqcolon m_h(Z_h)+R_h\), with
\(m_h(Z_h)\deq\E[\mathcal E_h\mid Z_h]\) and \(\E[R_h\mid Z_h]=0\).

\item \emph{Cancellation in the conditional mean.}
The weighted divergence of \(m_h\) determines its first-order effect on
the distribution. The key point is that the negative Sobolev norm measures the smallest
\(L^2\) size of a vector field possessing this divergence, detecting
cancellations invisible in the size of \(m_h\).
Although \(\norm{m_h}_{L^2(\bm\pi)}\leqslant C\sqrt d\,h^2\),
\Cref{subsec:uld-cancellation} proves
\(\norm{\delta_{\bm\pi}m_h}_{\dot H^{-1}(\bm\pi)}
\leqslant C\sqrt d\,h^3\) using its action on a Hermite decomposition. We mention that
\Cref{subsec:uld-normalization} contains some initial details about the negative Sobolev norm.

\item \emph{From mean and fluctuation to transport.}
\Cref{subsec:uld-transport} translates negative Sobolev bounds on the conditional mean, along with other $L^2$ error and derivative bounds, into errors measured in the transport cost. These give bounds of
\(O(\sqrt d\,h^3)\) and \(O(\sqrt d\,h^{7/2})\), respectively.

\item \emph{Predictor replacement and contraction.}
\Cref{subsec:uld-upper-auxiliary-proofs} takes care of one additional minor detail of replacing
the exact midpoint gradient by the gradient evaluated at the predictor.  It then proves the one-step estimate and numerical contraction. \Cref{subsec:uld-completion} combines them and returns to
the original units, completing \Cref{thm:uld-midpoint-main}.
\end{enumerate}

The separation of a local error into its conditional mean and a
centered fluctuation is the organizing principle of local error
analysis; see \cite[Section~5.1.1 and Lemma~5.1.2]{ChewiBook}.  For
small perturbations, \(\W_2\) is comparable to a weighted
\(\dot H^{-1}\) norm \cite{Peyre2018}; the minimum flux
characterization of this norm is recalled in \Cref{subsec:tb-sobolev},
and \Cref{lem:uld-midpoint-random-displacement} proves the finite
displacement bound used here.
The same lemma extends the Gaussian transport estimate of
\cite[Lemma~4.11]{HigherOrderRHMC} to the stationary phase space
distribution.

\subsection{Normalization and the negative Sobolev norm on phase space}
\label{subsec:uld-normalization}

The invariant distribution is
\[
    \bm\pi(\mathrm d x,\mathrm d p)\deq\pi(\mathrm d x)\,\cN(0,\Id_d)(\mathrm d p)\,,
    \qquad
    \pi(\mathrm d x)\propto\e^{-V(x)}\,\mathrm d x\,.
\]
The exact underdamped generator and its adjoint in \(L^2(\bm\pi)\) are
\begin{equation}
\begin{aligned}
    \mathcal L
    &\deq p\cdot\nabla_x-\nabla V\cdot\nabla_p
       +2\,(\Delta_p-p\cdot\nabla_p)\,,\\
    \mathcal L^\dagger
    &\deq-p\cdot\nabla_x+\nabla V\cdot\nabla_p
       +2\,(\Delta_p-p\cdot\nabla_p)\,.
\end{aligned}
    \label{eq:uld-generators}
\end{equation}
The Hamiltonian part is skew-adjoint, whereas the momentum
Ornstein--Uhlenbeck part is self-adjoint.  Thus \(\mathcal L^\dagger\)
generates the stationary reverse dynamics.  We write \( (\e^{t\mathcal L})_{t\ge 0} \)
for the forward semigroup and \( (\e^{t\mathcal L^\dagger})_{t\ge 0} \) for its
\(L^2(\bm\pi)\)-adjoint.  For a vector field \(F\),
\[
    \norm F_{L^2(\bm\pi)}^2
    \deq\int\norm{F(x,p)}^2\,\bm\pi(\mathrm d x,\mathrm d p)\,;
\]
if \(F\) depends only on \(x\), this equals \(\norm F_{L^2(\pi)}^2\).

Let \(D\deq(\nabla_x,\nabla_p)\), and let
\(\delta_{\bm\pi}\deq D^*\) be the adjoint in
\(L^2(\bm\pi)\), characterized by
\[
    \int\ip{D\varphi}{A}\,\mathrm d \bm\pi
    =\int\varphi\,\delta_{\bm\pi}A\,\mathrm d \bm\pi\,.
\]
For an \(\R^d\)-valued field \(F\) on phase space, let
\(\delta_x\) and \(\delta_p\) be the weighted divergences in the
position and momentum variables, respectively:
\[
\begin{aligned}
    \delta_xF&\deq-\nabla_x\cdot F+\ip{\nabla V}{F}\,,
    &\delta_pF&\deq-\nabla_p\cdot F+\ip pF\,,\\
    \delta_{\bm\pi}(F^x,F^p)&=\delta_xF^x+\delta_pF^p\,.
\end{aligned}
\]
The non-negative weighted Laplacian on phase space is
\begin{equation}
    \mathcal A\deq\delta_{\bm\pi}D
    =\mathcal A_x+\mathcal A_p\,,
    \qquad
    \mathcal A_x\deq-\Delta_x+\nabla V\cdot\nabla_x\,,
    \qquad
    \mathcal A_p\deq-\Delta_p+p\cdot\nabla_p\,.
    \label{eq:uld-weighted-laplacian}
\end{equation}
The spatial adjoint \(\delta_\pi\deq\nabla_x^*\) in \(L^2(\pi)\)
acts on a vector field \(F\) by
\(\delta_\pi F=-\nabla_x\cdot F+\ip{\nabla V}{F}\), and
\(\mathcal A_x=\delta_\pi\nabla_x\).
Thus \(\delta_xF=\delta_\pi F\) when \(F\) depends only on
position; we retain \(\delta_\pi\) for such fields.
The operator \(\mathcal A\) differs from the generator \(\mathcal L\).
In particular, \(\mathcal A_p\) is the negative of the momentum
Ornstein--Uhlenbeck generator and acts as multiplication by \(n\) on
the \(n\)-th Wiener chaos in \(p\); see \Cref{subsec:tb-malliavin}
for the normalized Hermite basis and this eigenvalue identity.

Let \(\dot H^1(\bm\pi)\) be the homogeneous Sobolev space given by
completion of smooth functions modulo constants for the norm
\(\norm{D\varphi}_2\).  Its continuous dual is
\(\dot H^{-1}(\bm\pi)\), with norm
\[
    \norm r_{\dot H^{-1}(\bm\pi)}
    \deq\sup_{\norm{D\varphi}_2\leqslant1}\ip r\varphi\,.
\]
The Riesz representation theorem gives
\begin{equation}
    \norm r_{\dot H^{-1}(\bm\pi)}
    =\norm{\mathcal A^{-1/2}r}_2
    =\norm{D\mathcal A^{-1}r}_2
    =\inf_{\delta_{\bm\pi}A=r}\,\norm A_2\,.
    \label{eq:uld-midpoint-Hminus1-variational}
\end{equation}
The minimizing field \(D\mathcal A^{-1}r\) is the orthogonal projection
of any field \(A\) satisfying \(\delta_{\bm\pi}A=r\) onto the closure
of gradients; see \cite[Theorem~1.3.20]{ChewiBook} and
\Cref{subsec:tb-sobolev}.

\begin{lemma}[Short-time stability in \(\dot H^{-1}\)]
\label{lem:uld-midpoint-Hminus1-stability}
There is a universal constant \(C>0\) such that, for every
\(r\in\dot H^{-1}(\bm\pi)\) and \(0\leqslant t\leqslant1\),
\[
    \norm{\e^{t\mathcal L^\dagger} r}_{\dot H^{-1}(\bm\pi)}
    \leqslant C\,\norm r_{\dot H^{-1}(\bm\pi)}\,.
\]
Moreover, \( (\e^{t\mathcal L^\dagger})_{t\ge 0} \) extends to a strongly continuous semigroup on
\(\dot H^{-1}(\bm\pi)\).
\end{lemma}

\begin{proof}
First let \(r\deq\delta_{\bm\pi}A\) with \(A\) smooth.  Write
\( (\Phi_t(z))_{t\ge 0} \) for the exact forward underdamped Langevin flow and
\(J_t(z)\deq D_z\Phi_t(z)\), with the Brownian path argument suppressed.
Then,
\[
    D\e^{t\mathcal L}\varphi(z)
    =\E\bigl[J_t(z)^\top D\varphi(\Phi_t(z))\bigr]\,.
\]

The drift Jacobian has universally bounded
operator norm because
\(0\preceq\nabla^2V\preceq\Id_d\), so
\(\norm{J_t}_{\mathrm{op}}\leqslant\e^{Ct}\).
Jensen's inequality and the bound on \(J_t\) give
\[
    \norm{D\e^{t\mathcal L}\varphi}_{L^2(\bm\pi)}^2
    \le \e^{2Ct}\,\norm{D\varphi}_{L^2(\bm\pi)}^2\,.
\]
Duality immediately gives
\[
    \norm{\e^{t\mathcal L^\dagger}r}_{\dot H^{-1}(\bm\pi)}
    \leqslant\e^{Ct}\,\norm r_{\dot H^{-1}(\bm\pi)}\,.
\]

Divergences of smooth compactly supported vector
fields are dense by \eqref{eq:uld-midpoint-Hminus1-variational}.  On these divergences,
\(L^2(\bm\pi)\)-strong continuity and Poincar\'e's inequality imply
strong continuity in \(\dot H^{-1}(\bm\pi)\).  The displayed bound
extends it to the completion.
\end{proof}

\subsection{Conditional mean of the exact midpoint local error}
\label{subsec:uld-defect}

As in \Cref{subsec:et-corrector,subsec:et-malliavin-stein}, we condition
the quadrature error on the endpoint of a stationary exact step and
separate its conditional mean from its centered fluctuation.  The
Gaussian derivative estimates in
\Cref{lem:uld-midpoint-local-differential-bounds} adapt the argument of
\Cref{prop:et-S4-derivative} to the underdamped dynamics: time reversal
uses a momentum flip, and noise reaches the position through the
momentum.  The conditional mean will be controlled by the negative
Sobolev estimate in \Cref{subsec:uld-cancellation}.

Let \((X_t,P_t)_{0\leqslant t\leqslant h}\) be a stationary underdamped Langevin
process and put
\(Z_h\deq(X_h,P_h)\).  Applying midpoint quadrature to the integrals of
\(a_{h-t}\nabla V(X_t)\) and \(\e^{-2(h-t)}\nabla V(X_t)\) in the exact
position and momentum formulas gives the midpoint local error
\(\mathcal E_h\), taken as numerical minus exact:
\begin{equation}
\begin{aligned}
    \mathcal E_h^x
    &\deq\int_0^h a_{h-t}\nabla V(X_t)\,\mathrm d t
       -h a_{h/2}\nabla V(X_{h/2})\,,\\
    \mathcal E_h^p
    &\deq\int_0^h \e^{-2(h-t)}\nabla V(X_t)\,\mathrm d t
       -h\e^{-h}\nabla V(X_{h/2})\,,\\
    \mathcal E_h&\deq(\mathcal E_h^x,\mathcal E_h^p)\,.
\end{aligned}
    \label{eq:exact-midpoint-defect}
\end{equation}
Decompose the local error into its conditional mean and centered
fluctuation:
\[
    m_h(Z_h)\deq\E[\mathcal E_h\mid Z_h]\,,
    \qquad
    R_h\deq\mathcal E_h-m_h(Z_h)\,.
\]
The numerical endpoint using the exact midpoint is $Z_h+m_h(Z_h)+R_h$.
To describe the randomness in \(R_h\), fix the endpoint
\(Z_h=z=(x,p)\).  By stationarity, the conditional law of
\((X_{h-t},P_{h-t})_{0\leqslant t\leqslant h}\) is the diffusion
generated by \(\mathcal L^\dagger\), started at \(z\).
The formulas in \eqref{eq:uld-generators} show that changing the sign
of momentum transforms these reverse dynamics into the forward
underdamped dynamics.  Thus
\((X_{h-t},-P_{h-t})_{0\leqslant t\leqslant h}\), conditionally on
\(Z_h=z\), has the law of the forward diffusion started at \((x,-p)\).

We can therefore construct this conditional path on canonical Wiener
space by solving the forward SDE from \((x,-p)\), driven by a fresh
standard Brownian motion \(\widetilde B\), and then undoing the momentum
flip.  Sampling \(Z_h\sim\bm\pi\) independently of \(\widetilde B\)
gives the required joint law of the endpoint and the reversed path.
Represent this Brownian noise by the isonormal Gaussian process
\[
    \mathcal H\deq L^2([0,h];\R^d)\,,
    \qquad
    G(f)\deq\int_0^h\ip{f_s}{\mathrm d\widetilde B_s}\,,
    \qquad f\in\mathcal H\,.
\]
Write \(\Omega\) for the probability space carrying the independent
pair \((Z_h,G)\); see \Cref{subsec:tb-malliavin} for the Gaussian
differentiation conventions.

Substituting the reconstructed positions into
\eqref{eq:exact-midpoint-defect} makes \(\mathcal E_h\) a function of
\((Z_h,G)\).  Averaging over \(G\) with \(Z_h=z\) fixed gives
\(m_h(z)\), so subtracting this average yields
\[
    R_h=R_h(Z_h,G)\,,
    \qquad
    \E_G[R_h(z,G)]=0
    \quad\text{for \(\bm\pi\)-almost every }z\,,
\]
where \(\E_G\) denotes expectation over the Gaussian noise alone.

We collect the local estimates.  Here
\(D_z\) differentiates this realization with respect to the terminal
state while holding \(G\) fixed, and \(D_G\) is the Malliavin derivative
with respect to \(G\) while holding the terminal state fixed.
In particular, \(D_GR_h=D_G\mathcal E_h\), since \(m_h(Z_h)\) depends
only on the terminal state.
The operator and Hilbert--Schmidt norms below are those used
in the unified displacement bound
\eqref{eq:uld-midpoint-random-displacement}.

\begin{lemma}[Local differential estimates]
\label{lem:uld-midpoint-local-differential-bounds}
With Euclidean norms on phase space, there is a universal constant
\(C>0\) such that, for every \(0<h\leqslant h_0\),
\begin{align}
    \norm{m_h}_2
    &\leqslant C\sqrt d\,h^2\,,
    &
    \norm{D_zm_h}_{L^\infty(\bm\pi;\mathrm{op})}
    &\leqslant Ch\,,
    \label{eq:uld-midpoint-mean-bounds}\\
    \norm{D_zR_h}_{L^2(\Omega;\mathrm{HS})}
    &\leqslant C\sqrt d\,h\,,
    &
    \norm{D_GR_h}_{L^\infty(\Omega;\mathrm{op})}
    &\leqslant Ch^{5/2}\,,
    \notag\\
    \norm{D_GR_h}_{L^2(\Omega;\mathrm{HS})}
    &\leqslant C\sqrt d\,h^{5/2}\,.
    \notag
\end{align}

The full local error also satisfies
\[
    \norm{D_z\mathcal E_h}_{L^\infty(\Omega;\mathrm{op})}
    \leqslant Ch\,.
\]

\end{lemma}

\begin{proof}
Stationarity and the adjoint semigroup give
\[
    \E[\nabla V(X_s)\mid Z_h]
    =\bigl(\e^{(h-s)\mathcal L^\dagger}\nabla V\bigr)(Z_h)\,,
    \qquad 0\leqslant s\leqslant h\,.
\]
Apply this conditional expectation to the formulas for
\(\mathcal E_h^x\) and \(\mathcal E_h^p\) in
\eqref{eq:exact-midpoint-defect}.  In those two integrals, write \(s\)
for forward time and then set \(t\deq h-s\).  This gives
\begin{equation}
\begin{aligned}
    m_h
    &=\int_0^hA_t\,\mathrm d t-hA_{h/2}
    =-\int_0^{h/2}t\dot A_t\,\mathrm d t
       +\int_{h/2}^h(h-t)\dot A_t\,\mathrm d t\,,
\end{aligned}
    \label{eq:uld-midpoint-mean-semigroup}
\end{equation}
where
\[
    g_t\deq\e^{t\mathcal L^\dagger}(\nabla V)\,,
    \qquad
    A_t\deq(a_tg_t,\e^{-2t}g_t)\,.
\]
The second equality in \eqref{eq:uld-midpoint-mean-semigroup} follows
by integration by parts on \([0,h/2]\) and \([h/2,h]\).
We bound
\(\norm{\dot A_t}_{L^2(\bm\pi)}\) by controlling \(g_t\) and its time derivative.
By \eqref{eq:uld-generators},
\(\mathcal L^\dagger(\nabla V)=-\nabla^2V\,p\), since \(\nabla V\)
depends only on \(x\).  Weighted integration by parts and averaging
\(p\sim\cN(0,\Id_d)\), respectively, give
\[
\begin{aligned}
    \norm{\nabla V}_{L^2(\pi)}^2
    &=\E_\pi\operatorname{tr}(\nabla^2V)\leqslant d\,,
    \qquad \norm{\mathcal L^\dagger(\nabla V)}_{L^2(\bm\pi)}^2
    =\E_\pi\operatorname{tr}\bigl((\nabla^2V)^2\bigr)\leqslant d\,.
\end{aligned}
\]
Applying the chain rule to \(\nabla V\) along the reverse position path
and taking conditional expectations gives, in \(L^2(\bm\pi)\),
\[
    g_t-\nabla V
    =\int_0^t\e^{s\mathcal L^\dagger}
       \mathcal L^\dagger(\nabla V)\,\mathrm d s\,.
\]
The \(L^2\) bounds on \(\nabla V\) and
\(\mathcal L^\dagger(\nabla V)\), together with strong continuity of
\(\e^{t\mathcal L^\dagger}\), justify differentiation:
\(\dot g_t=\e^{t\mathcal L^\dagger}\mathcal L^\dagger(\nabla V)\).
Contractivity of the semigroup gives
\(\norm{g_t}_{L^2(\bm\pi)}+\norm{\dot g_t}_{L^2(\bm\pi)}\leqslant2\sqrt d\).
Since \(\dot a_t=\e^{-2t}\), the product rule yields
\[
    \dot A_t=\bigl(\e^{-2t}g_t+a_t\dot g_t,\,
       \e^{-2t}\,(\dot g_t-2g_t)\bigr)\,,
    \qquad
    \norm{\dot A_t}_{L^2(\bm\pi)}\leqslant C\sqrt d
    \quad(0\leqslant t\leqslant h\leqslant h_0\leqslant1)\,.
\]
Substituting into \eqref{eq:uld-midpoint-mean-semigroup} gives
\(\norm{m_h}_2\leqslant C\sqrt d\,h^2\), as required in
\eqref{eq:uld-midpoint-mean-bounds}.

We next prove the terminal state derivative bounds for \(m_h\) and
\(R_h\), keeping \(G\) fixed.  Differentiating the SDE for
\((X_{h-t},-P_{h-t})\) with respect to \(z=(x,p)\) gives
\[
\begin{gathered}
    \frac{\mathrm d}{\mathrm d t}
    \begin{bmatrix}D_zX_{h-t}\\-D_zP_{h-t}\end{bmatrix}
    =\begin{bmatrix}
       0&\Id_d\\
       -\nabla^2V(X_{h-t})&-2\Id_d
     \end{bmatrix}
     \begin{bmatrix}D_zX_{h-t}\\-D_zP_{h-t}\end{bmatrix}\,,\qquad
    \begin{bmatrix}D_zX_h\\-D_zP_h\end{bmatrix}
    =\begin{bmatrix}\Id_d&0\\0&-\Id_d\end{bmatrix}\,.
\end{gathered}
\]
The initial matrix comes from the starting point \((x,-p)\) of the
reversed path after the momentum flip.  Since
\(\norm{\nabla^2V}_{\mathrm{op}}\leqslant1\), the coefficient matrix
in this variational equation has operator norm at most a universal
constant.  Gr\"{o}nwall's inequality therefore gives, uniformly in \(z\) and \(G\),
\[
    \norm{D_zX_{h-t}}_{\mathrm{op}}
       +\norm{D_zP_{h-t}}_{\mathrm{op}}
    \leqslant C\e^{Ct}\leqslant C\,,
    \qquad 0\leqslant t\leqslant h\leqslant h_0\leqslant1\,.
\]
The chain rule then bounds
\(D_z[\nabla V(X_{h-t})]=\nabla^2V(X_{h-t})D_zX_{h-t}\)
in operator norm by \(C\).  Applying this bound to each integral and
midpoint term in \eqref{eq:exact-midpoint-defect}, with \(t\) denoting
reverse time, yields
\[
\begin{aligned}
    \norm{D_z\mathcal E_h^x}_{\mathrm{op}}
    &\leqslant C\,\Bigl(\int_0^h a_t\,\mathrm d t+h a_{h/2}\Bigr)
       \leqslant Ch^2\,,\\
    \norm{D_z\mathcal E_h^p}_{\mathrm{op}}
    &\leqslant C\,\Bigl(\int_0^h\e^{-2t}\,\mathrm d t+h\e^{-h}\Bigr)
       \leqslant Ch\,.
\end{aligned}
\]
Here we used \(a_t\leqslant t\) and \(\e^{-2t}\leqslant1\).
Combining the two components gives
\(\norm{D_z\mathcal E_h}_{\mathrm{op}}\leqslant Ch\).
The uniform derivative bound and independence of \(G\) from the
terminal state justify differentiation under \(\E_G\):
\[
    D_zm_h(z)=\E_G[D_z\mathcal E_h(z,G)]\,,
    \qquad
    D_zR_h(z,G)=D_z\mathcal E_h(z,G)-D_zm_h(z)\,.
\]
Consequently,
\[
    \norm{D_zm_h}_{L^\infty(\bm\pi;\mathrm{op})}\leqslant Ch\,,
    \qquad
    \norm{D_zR_h}_{L^\infty(\Omega;\mathrm{op})}\leqslant Ch\,.
\]
Since \(D_zR_h\) is a \(2d\times2d\) matrix, its Hilbert--Schmidt
norm is at most \(\sqrt{2d}\) times its operator norm.  Taking the
\(L^2\) norm over \((Z_h,G)\) therefore gives
\[
    \norm{D_zR_h}_{L^2(\Omega;\mathrm{HS})}
    \leqslant\sqrt{2d}\,
       \norm{D_zR_h}_{L^\infty(\Omega;\mathrm{op})}
    \leqslant C\sqrt d\,h\,.
\]

We now estimate the Malliavin derivative, keeping the terminal state
fixed.  For a unit direction \(f\in\mathcal H\), let
\(\xi_t\deq D_GX_{h-t}f\) and \(\eta_t\deq-D_GP_{h-t}f\).
These variations of the path after the momentum flip satisfy
\[
    \dot\xi_t=\eta_t\,,
    \qquad
    \dot\eta_t=-2\eta_t-\nabla^2V(X_{h-t})\,\xi_t+2f_t\,,
    \qquad
    \xi_0=\eta_0=0\,.
\]
Integrating these equations gives
\(\eta_t=2\int_0^t\e^{-2(t-s)}f_s\,\mathrm d s
-\int_0^t\e^{-2(t-s)}\,\nabla^2V(X_{h-s})\,\xi_s\,\mathrm d s\) and
\(\xi_t=\int_0^t\eta_s\,\mathrm d s\).
Cauchy--Schwarz bounds
\(\norm{2\int_0^t\e^{-2(t-s)}f_s\,\mathrm d s}\) by \(2\sqrt t\).
For \(0\leqslant t\leqslant h\leqslant h_0\), the integral involving
\(\nabla^2V(X_{h-s})\,\xi_s\) is bounded by
\(Ct^2\sup_{0\leqslant s\leqslant t}\,\norm{\eta_s}\) and can be
absorbed in the estimate for \(\sup_{0\leqslant s\leqslant t}\,\norm{\eta_s}\), giving
\[
    \norm{D_GP_{h-t}}_{\mathrm{op}}\leqslant Ct^{1/2}\,,
    \qquad
    \norm{D_GX_{h-t}}_{\mathrm{op}}\leqslant Ct^{3/2}\,.
\]
Since \(D_GP_{h-t}\) and \(D_GX_{h-t}\) take values in \(\R^d\),
they have rank at most \(d\), so their Hilbert--Schmidt norms are
at most \(\sqrt d\) times the operator norm bounds above.
Differentiating the formulas for \(\mathcal E_h^p\) and
\(\mathcal E_h^x\) in \eqref{eq:exact-midpoint-defect} gives
\(\norm{D_G\mathcal E_h^p}_{\mathrm{op}}\leqslant Ch^{5/2}\) and
\(\norm{D_G\mathcal E_h^x}_{\mathrm{op}}\leqslant Ch^{7/2}\).
The extra power of \(h\) in the position component comes from
\(a_t=(1-\e^{-2t})/2\leqslant t\).  Since
\(D_GR_h=D_G\mathcal E_h\), this proves the operator and
Hilbert--Schmidt bounds for \(D_GR_h\).
\end{proof}

\subsection{\texorpdfstring{\(\dot H^{-1}\)}{H\textasciicircum\{-1\}} estimate for the weighted divergence}
\label{subsec:uld-cancellation}

We prove the weighted divergence bound for the conditional mean,
\(\norm{\delta_{\bm\pi}m_h}_{\dot H^{-1}(\bm\pi)}\leqslant C\sqrt d\,h^3\).
The extra power of \(h\) comes from cancellation of the constant and
linear terms under midpoint quadrature.  We bound the
quadrature remainder in \(\dot H^{-1}(\bm\pi)\) using only the bounded
Hessian assumption.

The proof uses an auxiliary estimate on
\((\mathcal L^\dagger)^2\zeta\), where
\(\zeta(x,p)\deq\ip p{\nabla V(x)}\).
This estimate controls the second-order semigroup term; we establish
it first.

\begin{lemma}[Second action of the adjoint generator]
\label{lem:uld-midpoint-second-generator}
There is a universal \(C>0\) such that
\[
    \norm{(\mathcal L^\dagger)^2\zeta}_{\dot H^{-1}(\bm\pi)}
    \leqslant C\sqrt d\,.
\]
\end{lemma}

\begin{proof}
We first bound the generator on functions with bounded momentum
Hermite degree, then apply duality.  Let \(\Pi_{\leqslant r}\) be the
orthogonal projection onto momentum Hermite degrees at most
\(r\in\mathbb N_0\).  Here the momentum Hermite basis
\(H_\ell\) is the normalized basis defined in
\Cref{subsec:tb-malliavin}; the range of \(\Pi_{\leqslant r}\) is the
span of \(M_\ell(x)H_\ell(p)\) with \(\lvert\ell\rvert\leqslant r\).
With the component weighted divergences defined in
\Cref{subsec:uld-normalization},
\[
    \mathcal L
    =\delta_p\nabla_x-\delta_x\nabla_p-2\mathcal A_p\,.
\]
For a smooth function \(\psi\) of momentum degree at most \(r\), the
weighted divergence estimate \eqref{eq:tb-weighted-divergence}, in
the Gaussian and spatial variables respectively, gives
\[
\begin{aligned}
    \norm{\delta_p\nabla_x\psi}_{L^2(\bm\pi)}^2
    &\leqslant(r+1)\,\norm{\nabla_x\psi}_{L^2(\bm\pi)}^2\,,\\
    \norm{\delta_x\nabla_p\psi}_{L^2(\bm\pi)}^2
    &\leqslant r\,\norm{\nabla_x\psi}_{L^2(\bm\pi)}^2
       +\norm{\nabla_p\psi}_{L^2(\bm\pi)}^2\,,\\
    \norm{\mathcal A_p\psi}_{L^2(\bm\pi)}^2
    &\leqslant r\,\norm{\nabla_p\psi}_{L^2(\bm\pi)}^2\,.
\end{aligned}
\]
Indeed, the first two estimates use
\(\norm{\nabla_p\nabla_x\psi}_{L^2(\bm\pi)}^2
\leqslant r\,\norm{\nabla_x\psi}_{L^2(\bm\pi)}^2\), and the last
follows from \(\mathcal A_pH_\ell=\lvert\ell\rvert H_\ell\).
Taking square roots, the triangle inequality bounds
\(\norm{\mathcal L\psi}_{L^2(\bm\pi)}\) by
\((\sqrt{r+1}+\sqrt r)\,\norm{\nabla_x\psi}_{L^2(\bm\pi)}
+(1+2\sqrt r)\,\norm{\nabla_p\psi}_{L^2(\bm\pi)}\).
Since \(\Pi_{\leqslant r}\) contracts the full gradient seminorm,
\[
    \norm{\mathcal L\Pi_{\leqslant r}\varphi}_{L^2(\bm\pi)}
    \leqslant C\sqrt{r+1}\,\norm{D\varphi}_{L^2(\bm\pi)}\,,
\]
where \(C\) is a universal constant.
For \(f\in L^2(\bm\pi)\) of momentum degree at most \(r\), the
generator changes the momentum degree by at most one.  Consequently,
\[
\begin{aligned}
    \bigl\lvert\ip{\mathcal L^\dagger f}{\varphi}\bigr\rvert
    &=\bigl\lvert\ip f{\mathcal L\Pi_{\leqslant r+1}\varphi}_{L^2(\bm\pi)}\bigr\rvert
    \leqslant C\sqrt{r+2}\,\norm f_{L^2(\bm\pi)}\,
       \norm{D\varphi}_{L^2(\bm\pi)}\,.
\end{aligned}
\]
Thus, \(\norm{\mathcal L^\dagger f}_{\dot H^{-1}(\bm\pi)}
\leqslant C\sqrt{r+2}\,\norm f_{L^2(\bm\pi)}\), with the left side
interpreted distributionally.
Apply this estimate with \(r=2\) and
\(f\deq\mathcal L^\dagger\zeta\).  Direct calculation gives
\begin{equation}
    \mathcal L^\dagger\zeta
    =-\bigl\{\ip p{\nabla^2V\,p}
       -\operatorname{tr}(\nabla^2V)\bigr\}
       +\delta_\pi(\nabla V)-2\zeta\,.
    \label{eq:uld-midpoint-first-generator}
\end{equation}
The three terms are Hermite polynomials in momentum of exact degrees \(2\), \(0\), and \(1\), so
they are orthogonal.  Gaussian moments,
\(\norm{\nabla V}_{L^2(\pi)}^2\leqslant d\), and
\eqref{eq:tb-weighted-divergence} give
\[
\begin{aligned}
    \norm{\ip p{\nabla^2V\,p}
       -\operatorname{tr}(\nabla^2V)}_{L^2(\bm\pi)}^2
    &=2\,\norm{\nabla^2V}_{L^2(\pi;\mathrm{HS})}^2\leqslant2d\,,\\
    \norm{\delta_\pi(\nabla V)}_{L^2(\pi)}^2
    &=\norm{\nabla^2V}_{L^2(\pi;\mathrm{HS})}^2
       +\E_\pi\ip{\nabla V}{\nabla^2V\,\nabla V}\leqslant2d\,,\\
    \norm\zeta_{L^2(\bm\pi)}^2
    &=\norm{\nabla V}_{L^2(\pi)}^2\leqslant d\,.
\end{aligned}
\]
It follows that \(\norm f_{L^2(\bm\pi)}\leqslant C\sqrt d\), which
proves the lemma by the dual estimate.  Approximation of the spatial
coefficients of \(f\) in \(L^2(\pi)\), keeping their momentum degrees
fixed, extends that estimate to the present \(C^2\) potential.
\end{proof}

By \eqref{eq:uld-midpoint-mean-bounds}, the conditional mean satisfies
\(\norm{m_h}_{L^2(\bm\pi)}\leqslant C\sqrt d\,h^2\).  The following
proposition gains an additional power of \(h\) when bounding its
weighted divergence in \(\dot H^{-1}(\bm\pi)\).

\begin{proposition}[\(\dot H^{-1}\) estimate for the weighted divergence]
\label{prop:uld-midpoint-weighted-divergence}
There is a universal constant \(C>0\) such that, for
every \(0<h\leqslant h_0\),
\[
    \norm{\delta_{\bm\pi}m_h}_{\dot H^{-1}(\bm\pi)}
    \leqslant C\sqrt d\,h^3\,.
\]
\end{proposition}

\begin{proof}
We express \(\delta_{\bm\pi}m_h\) as a midpoint quadrature error and separate
the adjoint semigroup contribution from a remainder.
From \eqref{eq:uld-midpoint-mean-semigroup},
\begin{equation}
    \delta_{\bm\pi}m_h=\int_0^h\psi_t\,\mathrm d t-h\psi_{h/2}\,,
    \qquad
    \psi_t\deq a_t\delta_xg_t+\e^{-2t}\delta_pg_t\,.
    \label{eq:uld-midpoint-divergence-quadrature}
\end{equation}
The commutator identities
\begin{align*}
    \delta_x\mathcal L^\dagger F-\mathcal L^\dagger\delta_xF
    &=\delta_p(\nabla^2V\,F)\,, \qquad
    \delta_p\mathcal L^\dagger F-\mathcal L^\dagger\delta_pF
    =-\delta_xF+2\delta_pF
\end{align*}
involve only \(\nabla^2V\).  Combining them with
\(\dot a_t=\e^{-2t}\), \(\frac{\mathrm d}{\mathrm d t}\,\e^{-2t}=-2\e^{-2t}\), and
\(\dot g_t=\mathcal L^\dagger g_t\) gives
\[
    \dot\psi_t-\mathcal L^\dagger\psi_t
    =a_t\delta_p(\nabla^2V\,g_t)\,,
    \qquad
    \psi_0=\zeta\,.
\]
Hence, in \(\dot H^{-1}(\bm\pi)\),
\[
    \psi_t=\e^{t\mathcal L^\dagger}\zeta+\mathcal R_t\,,
    \qquad
    \mathcal R_t
    \deq\int_0^t\e^{(t-s)\mathcal L^\dagger}
       \bigl\{a_s\delta_p(\nabla^2V\,g_s)\bigr\}\,\mathrm d s\,.
\]
Since
\(\delta_{\bm\pi}(0,\nabla^2V\,g_s)=\delta_p(\nabla^2V\,g_s)\),
\eqref{eq:uld-midpoint-Hminus1-variational} and the
\(L^2(\bm\pi)\) contractivity of the adjoint semigroup give
\[
    \norm{\delta_p(\nabla^2V\,g_s)}_{\dot H^{-1}(\bm\pi)}
    \leqslant\norm{\nabla^2V\,g_s}_{L^2(\bm\pi)}
    \leqslant\norm{g_s}_{L^2(\bm\pi)}
    \leqslant\norm{\nabla V}_{L^2(\pi)}\leqslant\sqrt d\,.
\]
By \Cref{lem:uld-midpoint-Hminus1-stability} and \(a_s\leqslant s\),
\[
    \norm{\mathcal R_t}_{\dot H^{-1}(\bm\pi)}
    \leqslant C\sqrt d\,t^2\,.
\]
Integrating the bound on \(\norm{\mathcal R_t}_{\dot H^{-1}(\bm\pi)}\) over
\([0,h]\) and evaluating it at \(h/2\) gives
\begin{equation*}
    \Bigl\lVert \int_0^h\mathcal R_t\,\mathrm d t-h\mathcal R_{h/2} \Bigr\rVert
       _{\dot H^{-1}(\bm\pi)}
    \leqslant C\sqrt d\,h^3\,.
\end{equation*}

It remains to bound the contribution of \(\e^{t\mathcal L^\dagger}\zeta\)
to \(\delta_{\bm\pi}m_h\).  We use the second-order integral expansion in
\(\dot H^{-1}(\bm\pi)\):
\begin{equation}
    \e^{t\mathcal L^\dagger}\zeta
    =\zeta+t\mathcal L^\dagger\zeta
       +\int_0^t(t-s)\e^{s\mathcal L^\dagger}
          (\mathcal L^\dagger)^2\zeta\,\mathrm d s\,.
    \label{eq:uld-midpoint-extrapolation-Duhamel}
\end{equation}
To justify \eqref{eq:uld-midpoint-extrapolation-Duhamel}, first note
that \eqref{eq:uld-midpoint-first-generator} can be written as
\[
 \mathcal L^\dagger\zeta
 =-\{\ip p{\nabla^2V\,p}-\operatorname{tr}(\nabla^2V)\}
       +\delta_\pi(\nabla V)-2\zeta\,.
\]
Gaussian moments and \eqref{eq:tb-weighted-divergence} give
\[
 \norm{\delta_\pi(\nabla V)}_{L^2(\pi)}^2
 =\E_\pi\operatorname{tr}((\nabla^2V)^2)
   +\E_\pi\ip{\nabla V}{\nabla^2V\nabla V}\leqslant2d\,,
 \qquad \norm{\mathcal L^\dagger\zeta}_{L^2(\bm\pi)}\leqslant C\sqrt d\,.
\]
\Cref{lem:uld-midpoint-second-generator} gives
\((\mathcal L^\dagger)^2\zeta\in\dot H^{-1}(\bm\pi)\).
Testing against a smooth function in the generator core and
integrating the weak generator identity twice proves the expansion.
The flow derivative bound in the proof of
\Cref{lem:uld-midpoint-Hminus1-stability} controls the gradient of
the forward semigroup in these pairings; approximation in the
generator core extends the identity to \(\dot H^{-1}(\bm\pi)\).

The bound on \(\mathcal R_t\),
\eqref{eq:uld-midpoint-extrapolation-Duhamel}, and
\Cref{lem:uld-midpoint-Hminus1-stability,lem:uld-midpoint-second-generator}
give the affine approximation
\begin{equation}
    \norm{\psi_t-\zeta-t\mathcal L^\dagger\zeta}
       _{\dot H^{-1}(\bm\pi)}
    \leqslant C\sqrt d\,t^2\,,
    \qquad 0\leqslant t\leqslant h\,.
    \label{eq:uld-midpoint-affine-remainder}
\end{equation}
The midpoint rule is exact on the affine function
\(t\mapsto\zeta+t\mathcal L^\dagger\zeta\), so subtracting this function
from both the integrand and the midpoint evaluation in
\eqref{eq:uld-midpoint-divergence-quadrature} leaves the error unchanged:
\[
\begin{aligned}
    \delta_{\bm\pi}m_h
    ={}&\int_0^h\bigl(\psi_t-\zeta-t\mathcal L^\dagger\zeta\bigr)
       \,\mathrm d t-h\bigl(\psi_{h/2}-\zeta-\tfrac h2\mathcal L^\dagger\zeta\bigr)\,.
\end{aligned}
\]
Applying \eqref{eq:uld-midpoint-affine-remainder} to these two remainders
and using the triangle inequality gives
\[
    \norm{\delta_{\bm\pi}m_h}_{\dot H^{-1}(\bm\pi)}
    \leqslant C\sqrt d\,\Bigl(\int_0^h t^2\,\mathrm d t
       +h\,(h/2)^2\Bigr)
    \leqslant C\sqrt d\,h^3\,.\qedhere
\]

\end{proof}

\subsection{Transport bounds for the conditional mean and fluctuation}
\label{subsec:uld-transport}

We convert the conditional mean cancellation and centered fluctuation
bounds into transport bounds. We first remove a divergence-free
field, which keeps only the part of the mean that moves the distribution; the
remaining terms are products of the local error and its derivatives.
The following lemma extends the Gaussian transport estimate of
\cite[Lemma~4.11]{HigherOrderRHMC} to a non-Gaussian base distribution.

For a base distribution \(\mu\) on \(\R^q\), write \(D\) for the
gradient and \(\delta_\mu\) for its \(L^2(\mu)\) adjoint.  Define
\(\dot H^{-1}(\mu)\) by duality against the gradient seminorm
\(\norm{D\varphi}_{L^2(\mu)}\), as in
\Cref{subsec:uld-normalization} for \(\dot H^{-1}(\bm\pi)\).

\begin{lemma}[Transport bound for a random displacement]
\label{lem:uld-midpoint-random-displacement}
Let \(\mu\in\mathcal P_2(\R^q)\) have density proportional to
\(\e^{-\mathcal V_\mu}\), where \(\mathcal V_\mu\in C^2(\R^q)\) and
\(\nabla^2\mathcal V_\mu\preceq\Lambda\Id_q\) for some \(\Lambda>0\).
On a probability space \(\Omega\), let \(Y\sim\mu\) and let \(G\)
be an independent isonormal Gaussian process with separable
Cameron--Martin space \(\mathcal H\).
Let \(U(Y,G)\in L^2(\Omega;\R^q)\) be weakly differentiable in \(Y\)
and Malliavin-differentiable in \(G\), with $D_YU,D_GU\in L^\infty(\Omega;\mathrm{op})$.
Each derivative holds the other variable fixed, and norms over
\(\Omega\) include both variables.  Define
\(m(Y)\deq\E[U\mid Y]\) and \(R\deq U-m(Y)\).  Then,
\begin{equation}
\begin{aligned}
    \W_2\bigl(\Law(Y+U),\mu\bigr)
    &\leqslant\norm{\delta_\mu m}_{\dot H^{-1}(\mu)}
       +\frac12\,\norm{D_YU}_{L^\infty(\Omega;\mathrm{op})}\,
          \norm U_{L^2(\Omega)}\\
    &\qquad+\frac12\,\norm{D_GU}_{L^\infty(\Omega;\mathrm{op})}\,
       \bigl(\norm{D_YR}_{L^2(\Omega;\mathrm{HS})}^2
          +\Lambda\,\norm R_{L^2(\Omega)}^2\bigr)^{1/2}\,.
\end{aligned}
\label{eq:uld-midpoint-random-displacement}
\end{equation}
\end{lemma}
\begin{proof}
We subtract a divergence-free field on the product of the base and
Gaussian spaces from the velocity of the linear interpolation.
Let \(w\) be the orthogonal projection of \(m\) onto the closure of
gradients in \(L^2(\mu;\R^q)\).  By
\Cref{subsec:tb-sobolev},
\[
    \delta_\mu w=\delta_\mu m\,,
    \qquad
    \norm w_{L^2(\mu)}=\norm{\delta_\mu m}_{\dot H^{-1}(\mu)}\,.
\]
The weighted divergence estimate \eqref{eq:tb-weighted-divergence}
in the base variable gives
\[
    \norm{\delta_\mu R}_{L^2(\Omega)}^2
    \leqslant\norm{D_YR}_{L^2(\Omega;\mathrm{HS})}^2
       +\Lambda\,\norm R_{L^2(\Omega)}^2\,.
\]
Here \(\delta_\mu R\) is defined by closure from smooth fields.
Since \(\E[R\mid Y]=0\), it also integrates to zero in the Gaussian
variable for almost every \(Y\).  Write \(\delta_G\) for the adjoint
of \(D_G\).  The Gaussian spectral gap and the inverse on functions
that integrate to zero, recalled in \Cref{subsec:tb-malliavin}, give
\[
\begin{gathered}
    F\deq-D_G(-\mathcal L_{\mathrm{OU}})^{-1}\delta_\mu R\,,
    \qquad \delta_GF=-\delta_\mu R\,,
    \qquad \norm F_{L^2(\Omega;\mathcal H)}
    \leqslant\norm{\delta_\mu R}_{L^2(\Omega)}\,.
\end{gathered}
\]
Thus, the product field \((U-w,F)\) has zero weighted divergence:
\[
    \delta_\mu(U-w)+\delta_GF
    =\delta_\mu(m-w)+\delta_\mu R-\delta_\mu R=0\,.
\]
In this formula and below, \(w\) is evaluated at \(Y\).
Set \(\mathsf X_\theta\deq Y+\theta U\), for \(0\leqslant\theta\leqslant1\).
For a smooth compactly supported test function \(\varphi\), test the
preceding divergence identity against
\(\varphi_\theta(Y,G)\deq\varphi(Y+\theta U)\).  By the adjoint definitions of the
divergences,
\[
    0=\E\bigl[\ip{D_Y\varphi_\theta}{U-w}
       +\ip{D_G\varphi_\theta}{F}_{\mathcal H}\bigr]\,.
\]
The chain rule then gives
\[
    0=\E\ip{D\varphi(\mathsf X_\theta)}
       {(\Id_q+\theta D_YU)\,(U-w)+\theta D_GU\,F}\,.
\]
Subtracting this identity from the derivative of
\(\E\varphi(\mathsf X_\theta)\) gives
\[
    \frac{\mathrm d}{\mathrm d\theta}\,\E\varphi(\mathsf X_\theta)
    =\E\ip{D\varphi(\mathsf X_\theta)}{w-\theta D_YU\,(U-w)-\theta D_GU\,F}\,.
\]
Consequently, a velocity for the continuity equation is
\[
    v_\theta(\mathsf X_\theta)\deq
       \E\bigl[w-\theta D_YU\,(U-w)-\theta D_GU\,F\mid \mathsf X_\theta\bigr]\,.
\]
Projection and conditional centering give
\[
    \norm{U-w}_{L^2(\Omega)}^2
    =\norm U_{L^2(\Omega)}^2-\norm w_{L^2(\mu)}^2
    \leqslant\norm U_{L^2(\Omega)}^2\,.
\]
Conditional Jensen's inequality therefore bounds the velocity by
\[
\begin{aligned}
    \norm{v_\theta}_{L^2(\Law(\mathsf X_\theta))}
    &\leqslant\norm w_{L^2(\mu)}
       +\theta\,\norm{D_YU}_{L^\infty(\Omega;\mathrm{op})}\,
          \norm U_{L^2(\Omega)}
    +\theta\,\norm{D_GU}_{L^\infty(\Omega;\mathrm{op})}\,
       \norm F_{L^2(\Omega;\mathcal H)}\,.
\end{aligned}
\]
The metric speed bound
\cite[Theorem~8.3.1]{AmbrosioGigliSavare2008}, integrated over
\([0,1]\), proves \eqref{eq:uld-midpoint-random-displacement}.
All integration by parts identities extend to the stated weak
derivatives by the Sobolev chain rule and closure: the composition
\(\varphi(Y+\theta U)\) has square integrable derivatives in both variables,
and \(F\) belongs to the domain of \(\delta_G\).
\end{proof}

\begin{remark}[Comparison with the Picard HMC transport estimate]
The proof of \Cref{lem:uld-midpoint-random-displacement} recovers and
strengthens \cite[Lemma~4.11]{HigherOrderRHMC}.  In its notation, let
\(Z\sim\cN(0,\Id_m)\) and \(G\sim\cN(0,\Id_n)\) be independent,
let \(f(z,g)\) have Lipschitz constants \(A\) in \(z\) and \(B\) in
\(g\), and write \(b(z)\deq\E_Gf(z,G)\) and
\(R\deq f(Z,G)-b(Z)\).  Suppose \(R\) takes values in a fixed
subspace \(E\subseteq\R^m\), with orthogonal projection \(\Pi_E\),
and let \(L\) be a linear map into a Hilbert space.  Then
\[
\begin{aligned}
    &\W_2\bigl(\Law(L\,(Z+f(Z,G))),\Law(L\,(Z+b(Z)))\bigr)
    \leqslant\sqrt m\,\Bigl\{
       \frac12\,\norm{L\Pi_E}_{\mathrm{op}}B^2
       +\frac32\,\norm L_{\mathrm{op}}AB\Bigr\}\,.
\end{aligned}
\]
This improves the coefficient \(5\) of the mixed term in that lemma
to \(3/2\).  Indeed, use the field \(F\) from the preceding proof
with Gaussian measure \(\mu=\cN(0,\Id_m)\), and interpolate by
\(\mathsf X_\theta\deq Z+b(Z)+\theta R\), for \(0\leqslant\theta\leqslant1\).
Integration by parts with the divergence-free field
\((R,F)\) allows us to subtract
\(D_Z(L\mathsf X_\theta)R+D_G(L\mathsf X_\theta)F\) from the original velocity \(LR\)
without changing the continuity equation.  The chain rule gives the
resulting velocity, before conditioning on \(L\mathsf X_\theta\):
\[
    -L\,(Db+\theta D_zR)\,R-\theta LD_GR\,F\,.
\]
Since \(Db+\theta D_zR=(1-\theta)\,Db+\theta D_zf\), its operator norm is at most
\(A\).  Gaussian Poincar\'e and conditional variance give, on a joint probability space $\Omega$ for \((Z,G)\),
\[
    \norm R_{L^2(\Omega)}\leqslant\sqrt m\,B\,,
    \qquad
    \norm{D_zR}_{L^2(\Omega;\mathrm{HS})}\leqslant\sqrt m\,A\,,
    \qquad
    \norm F_{L^2(\Omega;\R^n)}\leqslant\sqrt m\,\sqrt{A^2+B^2}\,.
\]
Also, \(D_GR\) takes values in \(E\), so
\(\norm{LD_GR}_{\mathrm{op}}
\leqslant\norm{L\Pi_E}_{\mathrm{op}}B\).
Integrating the velocity bound and using
\(\sqrt{A^2+B^2}\leqslant A+B\) proves the claim.
\end{remark}

\subsection{Auxiliary proofs for the exponential midpoint upper bound}
\label{subsec:uld-upper-auxiliary-proofs}

We first combine the transport bounds for the conditional mean and
centered fluctuation with the predictor replacement estimate to prove
\Cref{prop:uld-midpoint-one-step}.  We then establish
\Cref{prop:uld-midpoint-numerical-contraction}.

\begin{proof}[Proof of \Cref{prop:uld-midpoint-one-step}]
Use \Cref{lem:uld-midpoint-random-displacement} with
\(Y=Z_h\), \(U=\mathcal E_h\), \(\mu=\bm\pi\), and \(\Lambda=1\).
Here the base potential on phase space is
\(\mathcal V_{\bm\pi}(x,p)=V(x)+\norm p^2/2\), whose Hessian is
bounded above by \(\Id_{2d}\).

We have
\(D_z\mathcal E_h=D_zm_h+D_zR_h\).  The first bound below is proved
directly for this full derivative in
\Cref{lem:uld-midpoint-local-differential-bounds}.
Since \(m_h(z)\) is independent of \(G\), we have \(D_Gm_h=0\) and
\(D_G\mathcal E_h=D_GR_h\), so the second bound follows from the
Gaussian derivative estimate in the same lemma.  Thus,

\[
    \norm{D_z\mathcal E_h}_{L^\infty(\Omega;\mathrm{op})}
    \leqslant Ch\,,
    \qquad
    \norm{D_G\mathcal E_h}_{L^\infty(\Omega;\mathrm{op})}
    \leqslant Ch^{5/2}\,.
\]
Gaussian Poincar\'e at fixed \(Z_h\), followed by the
Hilbert--Schmidt derivative estimate, yields
\[
    \norm{R_h}_{L^2(\Omega)}
    \leqslant\norm{D_GR_h}_{L^2(\Omega;\mathrm{HS})}
    \leqslant C\sqrt d\,h^{5/2}\,.
\]
Together with \eqref{eq:uld-midpoint-mean-bounds} and conditional
centering, this gives
\[
    \norm{\mathcal E_h}_{L^2(\Omega)}^2
    =\norm{m_h}_{L^2(\bm\pi)}^2+\norm{R_h}_{L^2(\Omega)}^2
    \leqslant Cd h^4\,.
\]
Finally,
\(\norm{D_zR_h}_{L^2(\Omega;\mathrm{HS})}\leqslant C\sqrt d\,h\)
by \Cref{lem:uld-midpoint-local-differential-bounds}, and
\(\norm{\delta_{\bm\pi}m_h}_{\dot H^{-1}(\bm\pi)}
\leqslant C\sqrt d\,h^3\) by
\Cref{prop:uld-midpoint-weighted-divergence}.  Substitution into
\eqref{eq:uld-midpoint-random-displacement} gives directly
\[
    \W_2\bigl(\Law(Z_h+\mathcal E_h),\bm\pi\bigr)
    \leqslant C\sqrt d\,(h^3+h^{7/2})
    \leqslant C\sqrt d\,h^3\,.
\]
Now, we want to replace the exact path in our analysis with the predictor. Compare the exact midpoint update \(Z_h+\mathcal E_h\) with
the implementable update \eqref{eq:EM-x}--\eqref{eq:EM-v}, using the
same initial state and Brownian path.  The predictor
\eqref{eq:uld-predictor} satisfies
\[
    X_{h/2}-X_{h/2}^{\mathrm{pred}}
    =-\int_0^{h/2}a_{h/2-s}\,
       \bigl\{\nabla V(X_s)-\nabla V(X_0)\bigr\}\,\mathrm d s\,.
\]
Since \(a_t\leqslant t\), \(\nabla V\) is \(1\)-Lipschitz, and
stationarity gives \(\norm{X_s-X_0}_2\leqslant\sqrt d\,s\),
\[
    \norm{X_{h/2}-X_{h/2}^{\mathrm{pred}}}_2
    \leqslant\sqrt d\int_0^{h/2}(h/2-s)\,s\,\mathrm d s
    =\frac{\sqrt d\,h^3}{48}\,.
\]
The force coefficients in the position and momentum updates are
\(h a_{h/2}\) and \(h\e^{-h}\), respectively.  This coupling therefore gives
\begin{equation*}
\begin{aligned}
    \W_2\bigl(\bm\pi K_h^{\mathrm{ExpMid}},\Law(Z_h+\mathcal E_h)\bigr)
    &\leqslant h\sqrt{a_{h/2}^2+\e^{-2h}}\,
       \norm{X_{h/2}-X_{h/2}^{\mathrm{pred}}}_2
    \leqslant C\sqrt d\,h^4\,.
\end{aligned}
\end{equation*}
Finally, the triangle inequality and
\(\norm{T}_{\mathrm{op}}\leqslant C\), for the metric transformation
defined in \Cref{sec:underdamped-upper}, give
\[
    \W_{2,M}(\bm\pi K_h^{\mathrm{ExpMid}},\bm\pi)
    \leqslant C\,\bigl\{
       \W_2\bigl(\Law(Z_h+\mathcal E_h),\bm\pi\bigr)
       +\sqrt d\,h^4\bigr\}
    \leqslant C\sqrt d\,h^3\,.\qedhere
\]
\end{proof}

A direct expansion with an unweighted \(O(h^2)\) remainder
would restrict the step size to \(h\lesssim\kappa^{-1}\).
Instead, we keep the dissipation associated with the Hessian at the
predictor, which reduces the possible loss in squared distance to
\(O(h^5)\) and yields \(h\lesssim\kappa^{-1/4}\).
The sharpness construction following the proof shows that this range
cannot be improved by uniform contraction in the metric \(M\) under
only the Hessian bounds.

\begin{proof}[Proof of \Cref{prop:uld-midpoint-numerical-contraction}]
Use a synchronous coupling and put
\(\Delta x\deq x-x'\), \(\Delta p\deq p-p'\), and
\(\Delta z\deq(\Delta x,\Delta p)\).
Let \(y,y'\) be the predictor positions
\eqref{eq:uld-predictor} from \((x,p)\) and \((x',p')\), using the
same Brownian path.  Define the Hessians averaged along the segments
joining the initial and predictor positions by
\[
\begin{aligned}
 H_{\mathrm{init}}&\deq\int_0^1\nabla^2V(x'+s\Delta x)\,\mathrm d s\,,\\
 H_{\mathrm{pred}}&\deq\int_0^1\nabla^2V(y'+s(y-y'))\,\mathrm d s\,.
\end{aligned}
\]
Each lies between \(\kappa^{-1}\Id_d\) and \(\Id_d\);
\(H_{\mathrm{pred}}\) may depend on the noise.
The predictor difference and the output differences satisfy exactly
\[
 \begin{aligned}
 g&\deq H_{\mathrm{pred}}\,
       \{\Delta x+a_{h/2}\Delta p-c_{h/2}H_{\mathrm{init}}\, \Delta x\}\,,\\
 \Delta x^+&=\Delta x+a_h\Delta p-h a_{h/2}g\,,\qquad
 \Delta p^+=\e^{-2h}\Delta p-h\e^{-h}g\,.
 \end{aligned}
\]
Since \(a_t=(1-\e^{-2t})/2\),
\begin{equation}
 \begin{split}
 \norm{(\Delta x^+,\Delta p^+)}_M^2-\norm{\Delta z}_M^2
 &=-\frac{1-\e^{-4h}}2\,\norm{\Delta p}^2
   -2h\,\Bigl\langle\Delta x+\frac{1+\e^{-3h}}2\,\Delta p, g\Bigr\rangle
 \\
   &\qquad+\frac{h^2\,(1+\e^{-2h})}2\,\norm g^2\,.
 \end{split}
 \label{eq:uld-midpoint-exact-coupling-energy}
\end{equation}
The following quantity will be useful when computing the dissipation,
\[
 \mathcal D\deq2\,\norm{\Delta p}^2+
     2\,\ip{\Delta x+\Delta p}{H_{\mathrm{pred}}\,\Delta x}\,.
\]
Use the normalized metric transformation
\(T(\Delta x,\Delta p)=(\Delta x,\Delta x+\Delta p)\), and set
\(\xi\deq T\Delta z\), so that
\(\norm\xi^2=\norm{\Delta z}_M^2\).
The deterministic part of the update in the \(T\)-coordinates is
\[
 A\deq T\begin{bmatrix}
       0&\Id_d\\
       -H_{\mathrm{pred}}&-2\Id_d
     \end{bmatrix}T^{-1}
   =\begin{bmatrix}
       -\Id_d&\Id_d\\
       \Id_d-H_{\mathrm{pred}}&-\Id_d
     \end{bmatrix}\,.
\]
Its symmetric part is
\[
 \frac{A+A^\top}{2}
 =\begin{bmatrix}
    -\Id_d&\Id_d-\frac12\,H_{\mathrm{pred}}\\
    \Id_d-\frac12\,H_{\mathrm{pred}}&-\Id_d
   \end{bmatrix}\,.
\]
For each eigenvalue \(\lambda\in[\kappa^{-1},1]\) of
\(H_{\mathrm{pred}}\), the corresponding two-dimensional block
has eigenvalues \(-1\pm(1-\lambda/2)\), namely
\(-\lambda/2\) and \(-2+\lambda/2\). Hence
\((A+A^\top)/2\preceq-\Id_{2d}/(2\kappa)\),
\[
 \mathcal D=-2\,\ip{\xi}{A\xi}
 \geqslant\kappa^{-1}\,\norm\xi^2
 =\kappa^{-1}\,\norm{\Delta z}_M^2\,.
\]
We also have
\[
 \mathcal D\geqslant\norm{\Delta p}^2+
           \ip{\Delta x}{H_{\mathrm{pred}}\,\Delta x}\,.
\]
To see this, subtract its right-hand side from
\(\mathcal D\) and complete the square using
\((H_{\mathrm{pred}})^2\preceq
H_{\mathrm{pred}}\).

First substitute
\(\bar g\deq H_{\mathrm{pred}}\,(\Delta x+a_{h/2}\Delta p)\)
for \(g\) in the right-hand side of
\eqref{eq:uld-midpoint-exact-coupling-energy}.
For brevity, write
\(a\deq a_{h/2}\) and
\(b\deq(1+\e^{-3h})/2\). Then
\[
 \ip{\Delta x+b\Delta p}{\bar g}
 =\ip{\Delta x}{H_{\mathrm{pred}}\Delta x}
  +(a+b)\,\ip{\Delta p}{H_{\mathrm{pred}}\Delta x}
  +ab\,\ip{\Delta p}{H_{\mathrm{pred}}\Delta p}\,.
\]
Since \((1-\e^{-4h})/2=2h+O(h^2)\) and \(a+b=1+O(h)\), the
leading terms in \eqref{eq:uld-midpoint-exact-coupling-energy} give
\(-h\mathcal D\); the last term in the preceding display contributes
\(-2hab\,\ip{\Delta p}{H_{\mathrm{pred}}\Delta p}\leqslant0\).
The coefficient errors are controlled by the preceding lower bound on
\(\mathcal D\), since \(H_{\mathrm{pred}}^2\preceq H_{\mathrm{pred}}\) gives
\[
 \bigl|\ip{\Delta p}{H_{\mathrm{pred}}\Delta x}\bigr|
 \leqslant\frac12\bigl(\norm{\Delta p}^2+\ip{\Delta x}{H_{\mathrm{pred}}\Delta x}\bigr)
 \leqslant\frac12\,\mathcal D\,.
\]
The positive term involving \(\norm{\bar g}^2\) is also controlled by
\(\mathcal D\):
\[
\begin{aligned}
 \norm{\bar g}^2
 &\leqslant2\,\norm{H_{\mathrm{pred}}\Delta x}^2
       +2a^2\,\norm{H_{\mathrm{pred}}\Delta p}^2\\
 &\leqslant2\,\ip{\Delta x}{H_{\mathrm{pred}}\Delta x}
       +2a^2\,\norm{\Delta p}^2
 \leqslant C\mathcal D\,.
\end{aligned}
\]

Thus, this substitution gives an upper bound
\(-h\mathcal D+Ch^2\mathcal D\).

The predictor correction satisfies
\[
 g-\bar g=-c_{h/2}H_{\mathrm{pred}}
                          H_{\mathrm{init}}\,\Delta x\,,\qquad
 \norm{g-\bar g}\leqslant Ch^2\,\norm{\Delta x}\,.
\]
We also have
\[
 \Bigl\lvert\ip{\Delta x+\frac{1+\e^{-3h}}2\,\Delta p}{g-\bar g}\Bigr\rvert
 \leqslant Ch^2\sqrt{\mathcal D}\,\norm{\Delta x}\,.
\]
Together with \(\norm{\bar g}\leqslant C\sqrt{\mathcal D}\),
\eqref{eq:uld-midpoint-exact-coupling-energy} consequently yields
\[
 \begin{split}
 \norm{(\Delta x^+,\Delta p^+)}_M^2-\norm{\Delta z}_M^2
 &\leqslant-h\mathcal D+Ch^2\mathcal D
       +Ch^3\sqrt{\mathcal D}\,\norm{\Delta x}+Ch^6\,\norm{\Delta x}^2\\
 &\leqslant-\frac34\,h\mathcal D
       +Ch^5\,\norm{\Delta z}_M^2\,.
 \end{split}
\]
Here the last line uses Young's inequality and a universal upper bound
on \(h\). If \(h\leqslant c\kappa^{-1/4}\), choose \(c\) small
enough to absorb \(Ch^5\) into \(h/(4\kappa)\). We obtain
\[
 \norm{(\Delta x^+,\Delta p^+)}_M^2
 \leqslant\Bigl(1-\frac h{2\kappa}\Bigr)\,
                      \norm{\Delta z}_M^2\,.
\]
This holds for every common noise realization. Integrating over an
optimal initial coupling and taking square roots gives the stated
contraction factor \(1-h/(4\kappa)\).
\end{proof}

\begin{remark}[Sharpness of the step size condition]
\label{rem:uld-midpoint-sharpness}
The range \(h\leqslant c\kappa^{-1/4}\) cannot be improved in general:
when the averaged Hessians at the initial and predictor positions do
not commute, the exponential midpoint kernel can expand \(\W_{2,M}\) by
a term of order \(h^5\).  Choose a smooth convex function
\(\psi:\R\to\R\) such that \(0\leqslant\psi''\leqslant1/4\), with
\(\psi''=1/4\) near zero and \(\psi''=0\) on \([2,\infty)\).  For
\(0<\alpha\leqslant1/4\), let
\[
 V_\alpha(x)\deq\frac\alpha2\,\norm x^2+\frac{x_2^2}{8}
                   +\psi(x_1+x_2)\,,\qquad x\in\R^2\,.
\]
Then, \(\alpha\Id_2\preceq\nabla^2V_\alpha\preceq\Id_2\), and
\(\nabla^2V_\alpha\) is constant on the strip near \(x_1+x_2=0\) where
\(\psi''=1/4\), and on the half-plane \(x_1+x_2\geqslant2\).  Let
\(e_1\) be the first coordinate vector, and let
\[
 \Delta x\deq\begin{bmatrix}
  1\\[0.4em]
  \displaystyle\frac{c_{h/2}}{8\,\{1-c_{h/2}\,(\alpha+1/2)\}}
 \end{bmatrix}\,.
\]
Start the synchronous coupling at \((0,L e_1)\) and
\((\varepsilon_0\,\Delta x,L e_1)\), where \(\varepsilon_0>0\) is small
enough that the segment joining the initial positions lies in the strip and $L$ is a constant.
For either initial position
\(x\in\{0,\varepsilon_0\Delta x\}\), the predictor
\eqref{eq:uld-predictor} is
\[
 X_{h/2}^{\mathrm{pred}}(x)
 =x+a_{h/2}L e_1-c_{h/2}\nabla V_\alpha(x)+\eta\,,
\]
where \(\eta\) is the common Gaussian noise, whose law does not depend
on \(L\). In particular, the predictor's coordinate sum satisfies
\[
 \sum_{i=1}^2\bigl(X_{h/2}^{\mathrm{pred}}(x)\bigr)_i
 =a_{h/2}L+\sum_{i=1}^2
       \bigl(x_i-c_{h/2}\partial_iV_\alpha(x)+\eta_i\bigr)\,.
\]
For fixed \(h>0\), we have \(a_{h/2}>0\), while the remaining terms
do not depend on \(L\). Thus both predictors lie in the half-plane
\(\{y\in\R^2:y_1+y_2\geqslant2\}\) with probability tending to one
as \(L\to\infty\). By convexity of this half-plane, the segment
joining them also lies there.
On this event, the averaged Hessians are
\[
 H_{\mathrm{init}}
 =\begin{bmatrix}\alpha+1/4&1/4\\1/4&\alpha+1/2\end{bmatrix}\,,\qquad
 H_{\mathrm{pred}}
 =\begin{bmatrix}\alpha&0\\0&\alpha+1/4\end{bmatrix}\,,
\]
and, since the momentum difference vanishes, the output difference is
\(\varepsilon_0\,(\Delta x-h a_{h/2}g,-h\e^{-h}g)\) with
\(g\deq H_{\mathrm{pred}}\,
(\Id_2-c_{h/2}H_{\mathrm{init}})\,\Delta x\).
The gradient is \(1\)-Lipschitz, so dominated convergence shows that
the difference of the output means tends to this vector.

For \(0<h\leqslant1\), the second component of
\((\Id_2-c_{h/2}H_{\mathrm{init}})\,\Delta x\) is
\(-c_{h/2}/8\), and its first component lies in \((0,1]\).
Using \eqref{eq:uld-midpoint-exact-coupling-energy} with \(\Delta p=0\)
and \(c_{h/2}=(h-1+\e^{-h})/4\geqslant h^2/12\) gives
\[
 \begin{split}
 &\norm{(\Delta x-h a_{h/2}g,-h\e^{-h}g)}_M^2-\norm{(\Delta x,0)}_M^2\\
 &\qquad=-2h\,\ip{\Delta x}g+
              \frac{h^2\,(1+\e^{-2h})}2\,\norm g^2
 \geqslant-2\alpha h+\frac{h c_{h/2}^2}{128}
 \geqslant-2\alpha h+\frac{h^5}{18432}\,.
 \end{split}
\]
The \(\W_{2,M}\) distance between the output distributions is at least
the \(M\)-norm of the difference of their means.  Hence, if
\(\alpha<h^4/36864\), it exceeds the initial distance
\(\varepsilon_0\,\norm{(\Delta x,0)}_M\) for sufficiently large finite
\(L\).  Taking \(\alpha=\kappa^{-1}\) shows that a uniform contraction
range must satisfy \(h\lesssim\kappa^{-1/4}\).
\end{remark}

\subsection{Completion of the underdamped upper bound}
\label{subsec:uld-completion}

\begin{proof}[Proof of \Cref{thm:uld-midpoint-main}]
We first return to the original parameters
\(\alpha,\beta\) and \(\gamma=2\sqrt\beta\).  Let
\[
    \overline V(y)\deq V(y/\sqrt\beta)\,,
    \qquad
    \mathsf S_\beta(x,p)\deq(\sqrt\beta\,x,p)\,.
\]
Write \(\overline K_s^{\mathrm{ExpMid}}\) for the exponential midpoint
kernel with potential \(\overline V\), friction \(2\), and step
\(s\).  Let \(\overline{\bm\pi}\) be the phase space target
associated with \(\overline V\).
The update formulas \eqref{eq:EM-x}--\eqref{eq:EM-v} give
\[
    \mathsf S_{\beta\#}(\mu K_h^{\mathrm{ExpMid}})
    =(\mathsf S_{\beta\#}\mu)\,\overline K_{\sqrt\beta\,h}^{\mathrm{ExpMid}}\,,
    \qquad
    \mathsf S_{\beta\#}\bm\pi=\overline{\bm\pi}\,.
\]
The definition of \(M_\gamma\) also gives
\[
    \norm{\mathsf S_\beta z}_{M}
    =\sqrt\beta\,\norm z_{M_\gamma}\,.
\]
For \(0<\sqrt\beta\,h\leqslant{c\kappa^{-1/4}}\), with the universal
constant \(c\) chosen so that \(c\leqslant h_0\),
\Cref{prop:uld-midpoint-numerical-contraction,prop:uld-midpoint-one-step}
therefore give, for all \(\mu,\nu\in\mathcal P_2(\R^{2d})\),
\[
\begin{aligned}
    \W_{2,M_\gamma}(\mu K_h^{\mathrm{ExpMid}},\nu K_h^{\mathrm{ExpMid}})
    &\leqslant\Bigl(1-\frac{\sqrt\beta\,h}{4\kappa}\Bigr)\,
       \W_{2,M_\gamma}(\mu,\nu)\,,\\
    \sqrt\beta\,\W_{2,M_\gamma}
       (\bm\pi K_h^{\mathrm{ExpMid}},\bm\pi)
    &\leqslant C\sqrt d\,(\sqrt\beta\,h)^3\,.
\end{aligned}
\]

Since \(\nabla V\) is globally Lipschitz, the update formulas show
that \(K_h^{\mathrm{ExpMid}}\) preserves finite second moments.
Together with the contraction bound, completeness of
\(\mathcal P_2(\R^{2d})\) under \(\W_{2,M_\gamma}\)
\cite[Proposition~1.3.15]{ChewiBook} allows us to apply Banach's fixed
point theorem.  This gives a unique invariant distribution
\(\bm\pi_h^{\mathrm{ExpMid}}\in\mathcal P_2(\R^{2d})\).
Dividing the stationary one step error by the contraction gap gives
\[
    \sqrt\beta\,\W_{2,M_\gamma}
       (\bm\pi_h^{\mathrm{ExpMid}},\bm\pi)
    \leqslant C\kappa\sqrt d\,(\sqrt\beta\,h)^2\,.
\]

Position projection is \(1\)-Lipschitz in the
\(M_\gamma\)-norm.  Multiplying the preceding bound by
\(\sqrt\alpha/\sqrt\beta=\kappa^{-1/2}\) proves
\eqref{eq:uld-midpoint-position-bias}.

For \(\widehat X_0=x_\star\) and
\(\widehat P_0\sim\cN(0,\Id_d)\), strong convexity and
integration by parts give
\(\E_\pi\norm{X-x_\star}^2\leqslant d/\alpha\).
Couple the initial momentum identically with the independent Gaussian
momentum under \(\bm\pi\). The definition of \(M_\gamma\) then bounds the
initial squared distance by
\[
    \beta\,\W_{2,M_\gamma}
       (\delta_{x_\star}\otimes\cN(0,\Id_d),\bm\pi)^2
    \leqslant2\beta\,\E_\pi\norm{X-x_\star}^2
    \leqslant2\kappa d\,.
\]
Iterating the rescaled contraction and stationary error bounds against
\(\bm\pi\) yields
\[
    \sqrt\beta\,\W_{2,M_\gamma}
       \bigl((\delta_{x_\star}\otimes\cN(0,\Id_d))
                    (K_h^{\mathrm{ExpMid}})^n,\bm\pi\bigr)
    \leqslant C\sqrt{\kappa d}\,\e^{-cn\sqrt\beta\,h/\kappa}
       +C\kappa\sqrt d\,(\sqrt\beta\,h)^2\,.
\]
After projection to position and multiplication by
\(\sqrt\alpha/\sqrt\beta=\kappa^{-1/2}\), this proves
\eqref{eq:uld-midpoint-finite-time-explicit}.
To reach the accuracy criterion \eqref{eq:global-absolute-criterion},
choose
\[
    \sqrt\beta\,h={c\kappa^{-1/4}d^{-1/4}\varepsilon^{1/2}}\,, \qquad
    n=\Bigl\lceil
       C\,\frac\kappa{\sqrt\beta\,h}\,
       \log\Bigl(\e+\frac d{\varepsilon^2}\Bigr)
      \Bigr\rceil\,,
\]
with \(c\) sufficiently small and \(C\) sufficiently large.
Since \(\varepsilon\leqslant\sqrt d/2\), this step lies in the
range \(\sqrt\beta\,h\leqslant c\kappa^{-1/4}\).
The step choice makes the stationary term in
\eqref{eq:uld-midpoint-finite-time-explicit} at most
\(\varepsilon/2\); the iteration count makes the transient term at
most \(\varepsilon/2\).
Since
\[
    \frac\kappa{\sqrt\beta\,h}
    {=\frac1c\,\kappa^{5/4}d^{1/4}\varepsilon^{-1/2}}\,,
\]
and since \(d/\varepsilon^2\geqslant4\) gives
\(\log(\e+d/\varepsilon^2)\leqslant2\log(d/\varepsilon^2)\),
the stated choice of \(n\) has the order claimed in
\eqref{eq:uld-midpoint-complexity}.
Each step evaluates the gradient at the current point and at the
predictor, so the query count is at most \(2n\).  Absorbing this
factor of two into the universal constant proves
\eqref{eq:uld-midpoint-complexity}.
\end{proof}

\section{Proofs of the deterministic lower bounds}
\label{app:deterministic-lower}

This appendix proves the deterministic quadrature lower bounds in
\Cref{sec:deterministic-lower}, which are separated into the three theorems in that section (overdamped with $\kappa = O(1)$, $\kappa \gg 1$ respectively, and the underdamped case).
\begin{enumerate}
\item \emph{$\kappa = O(1)$ lower bound.} Recall that in this construction, we have a quadratic potential perturbed by an oscillation of amplitude $h$ and wavelength $\sqrt h$. This yields a high-frequency error, which is captured by
\Cref{lem:et-one-step-defect}. This lemma computes the expectation change of
\(\sqrt h\cos(\cdot/\sqrt h)\) after one Heun update started from the target.
Invariance and contractivity turn this change into the stationary
\(h^{3/2}\) lower bound in
\Cref{subsec:common-overdamped-proof}.

\item \emph{Coupling to a weakly confined coordinate.}
Generalizing the above to the $\kappa \gg 1$ case, we consider, in addition to the oscillatory coordinates, an additional coordinate with small curvature, coupled to the others. \Cref{lem:heun-modulated-block} shows how the oscillation in the other coordinates shifts the mean of this final, weakly confined coordinate, which creates the leading discretization error.
We leverage this phenomenon to prove
\Cref{thm:et-linear-kappa-lower} in
\Cref{subsec:heun-product-lower-proof}.

\item \emph{A Gaussian obstruction for exponential midpoint.}
\Cref{subsec:exponential-midpoint-Gaussian-proof} performs some elementary calculations in the underdamped case to obtain the invariant covariance. Its position
variance can be solved explicitly, and differs from the target's by order \(h^2\); the Gaussian
Wasserstein formula then proves
\Cref{prop:exponential-midpoint-Gaussian-lower}. This is generalized, via a simple anisotropic quadratic, to the case where $\kappa \gg 1$.
\end{enumerate}
For the application of these bounds to get iteration lower bounds in
\Cref{cor:midpoint-iteration-lower}, the proof is deferred to
\Cref{subsec:lower-iteration-proof}, which combines these bias estimates to obtain iteration lower bounds for some initialization.

\subsection{Heun at bounded condition number}
\label{subsec:overdamped-bounded-proof}

The oscillation wavelength \(\sqrt h\) matches the Brownian displacement
in one step, while its amplitude \(h\) keeps the Hessian bounded. We
compute the change of a Lipschitz observable in one update from the target;
this avoids expanding the numerical invariant distribution.
For a probability distribution \(\mu\), write
\(\mu(f)\deq\int f\,\mathrm d\mu\).

\subsubsection{The expectation change in one update}
\label{subsec:et-one-step-proof}

The following lemma supplies the scalar stationary obstruction.

\begin{lemma}[Oscillatory expectation change for Heun]
\label{lem:et-one-step-defect}
Fix a numerical \(0<a_\star<1/4\), and let
\[
 V(x)\deq\frac{x^2}{2}+a_\star h\cos(x/\sqrt h)\,,\qquad
 \pi\propto\e^{-V}\,,\qquad f(x)\deq\sqrt h\cos(x/\sqrt h)\,.
\]
For all sufficiently small \(h>0\), the Heun kernel \(K_h^{\mathrm{Heun}}\)
has a unique invariant distribution \(\mu\) with finite second moment, and
\[
 \W_2(\mu,\pi)\geqslant
 \frac12\,\bigl\lvert\pi(K_h^{\mathrm{Heun}}f-f)\bigr\rvert
 \geqslant ch^{3/2}\,.
\]
The constants are numerical once \(a_\star\) is fixed.
\end{lemma}
\begin{proof}
For fixed \(g=\sqrt{2h}\,G\), put \(E_g(x)\deq x-hV'(x)+g\).
The Heun update is \(\{x+E_g(E_g(x))\}/2\). Since
\(V''\in[1-a_\star,1+a_\star]\), for \(h\,(1+a_\star)\leqslant1\)
its Lipschitz constant is at most
\(\{1+(1-h\,(1-a_\star))^2\}/2<1\).
Also, $K_h^{\rm Heun}(0,\cdot)$ has finite second moment, so contraction gives the
unique invariant distribution in \(\mathcal P_2\). In particular,
\(K_h^{\mathrm{Heun}}f\) is \(1\)-Lipschitz. Invariance therefore gives
\[
 \lvert\pi(K_h^{\mathrm{Heun}}f-f)\rvert
 =\lvert(\pi-\mu)(K_h^{\mathrm{Heun}}f-f)\rvert\leqslant2\W_2(\mu,\pi)\,.
\]

Let \(Z,G\) be independent standard Gaussians and define
\begin{equation*}
 A_h^{\mathrm{Heun}}\deq1-h+\frac{h^2}{2}\,,\qquad
 \sigma_h^{\mathrm{Heun}}\deq\sqrt{2h}\,\Bigl(1-\frac h2\Bigr)\,.
\end{equation*}
Write
\[
 \theta\deq Z/\sqrt h\,,\qquad
 U\deq(A_h^{\mathrm{Heun}}Z+\sigma_h^{\mathrm{Heun}}G)/\sqrt h\,,\qquad
 W\deq((1-h)\,Z+\sqrt{2h}\,G)/\sqrt h\,.
\]
The exact output of Heun started at $Z$, divided by \(\sqrt h\), is
\[
 U+\frac{a_\star h}{2}\,
       \{(1-h)\sin\theta+\sin(W+a_\star h\sin\theta)\}\,.
\]
Since
\[
 \lvert\sin(W+a_\star h\sin\theta)-\sin W\rvert\leqslant a_\star h\,,
 \qquad
 \bigl\lvert\e^{-a_\star h\cos\theta}-1+a_\star h\cos\theta\bigr\rvert
       \leqslant Ca_\star^2h^2\,,
\]
the target normalizing constant satisfies
\[
 \E\e^{-a_\star h\cos\theta}
    =1-a_\star h\e^{-1/(2h)}+O(h^2)=1+O(h^2)\,.
\]
The Gaussian variances are
\[
 \begin{split}
 \operatorname{Var}(U-\theta)
   &=(A_h^{\mathrm{Heun}}-1)^2/h+2\,(1-h/2)^2=2+O(h)\,,\\
 \operatorname{Var}(U-W)&=h^3/4+h^2/2\,,\\
 \operatorname{Var}(U+\theta)&\geqslant(A_h^{\mathrm{Heun}}+1)^2/h\geqslant c/h\,,\\
 \operatorname{Var}(U+W)&\geqslant(A_h^{\mathrm{Heun}}+1-h)^2/h\geqslant c/h\,.
 \end{split}
\]
Thus, the Gaussian characteristic function and the sine and cosine
product identities give
\[
 \begin{split}
 \E[\cos U\cos\theta]&=\tfrac12\,\e^{-1}+O(h)\,,\qquad
 \E[\sin U\sin\theta]=\tfrac12\,\e^{-1}+O(h)\,,\\
 \E[\sin U\sin W]&=\tfrac12+O(h)\,,\qquad
 \E\cos U\,,\ \E\cos\theta=O(\e^{-c/h})\,.
 \end{split}
\]
For \(\varphi(x)\deq\cos(x/\sqrt h)\), Taylor's inequality
\(\lvert\cos(s+t)-\cos s+t\sin s\rvert\leqslant t^2/2\) yields
\[
 \begin{split}
 (1+O(h^2))\,\pi(K_h^{\mathrm{Heun}}\varphi)
  &=\E\cos U-a_\star h\,\E[\cos U\cos\theta]\\
   &\qquad -\frac{a_\star h}{2}\,
     \{(1-h)\,\E[\sin U\sin\theta]+\E[\sin U\sin W]\}+O(h^2)\\
     &=-\frac{a_\star h}{4}\,(1+3\e^{-1})+O(h^2)\,,\\
 (1+O(h^2))\,\pi(\varphi)
 &=\E\cos\theta-a_\star h\,\E\cos^2\theta+O(h^2)
   =-\frac{a_\star h}{2}+O(h^2)\,.
 \end{split}
\]
Thus,
\(\pi(K_h^{\mathrm{Heun}}f-f)
 =-a_\star\,(3\e^{-1}-1)\,h^{3/2}/4+O(h^{5/2})\).
The leading coefficient is non-zero, proving the claim.
\end{proof}

\subsubsection{Comparison with a Gaussian target}
\label{subsec:heun-gaussian-proof}

On a purely Gaussian target, the same affine recursion has only an
\(h^2\) bias. This explains why the oscillatory perturbation is needed
for the larger \(h^{3/2}\) Heun lower bound.

\begin{proposition}[Gaussian lower bound for Heun]
\label{prop:heun-gaussian-lower}
For every \(d\geqslant1\) and \(0<h\leqslant h_0\), Heun on
\(V=\norm \cdot^2/2\) has a unique invariant distribution with finite second
moment, and
\begin{equation}
    \W_2(\widehat\pi_h^{\mathrm{Heun}},\cN(0,\Id_d))
       {\asymp\sqrt d\,h^2}\,.
    \label{eq:et-gaussian-lower}
\end{equation}
\end{proposition}

\begin{proof}
For the Heun method, one step is
\[
    X^+=\Bigl(1-h+\frac{h^2}{2}\Bigr)\,X
       +\sqrt{2h}\,\Bigl(1-\frac h2\Bigr)\,G_d\,,
\]
where \(G_d\sim\cN(0,\Id_d)\) is independent of \(X\).
Its invariant distribution is $\cN(0,s_h^{\mathrm{Heun}}\Id_d)$, where
\begin{equation*}
    s_h^{\mathrm{Heun}}
    \deq\frac{2h-2h^2+h^3/2}
       {2h-2h^2+h^3-h^4/4}
    =1-\frac{h^2}{4}+O(h^3)\,.
\end{equation*}
The exact formula
$\W_2(\cN(0,s\Id_d),\cN(0,\Id_d))=\sqrt d\,\lvert\sqrt s-1\rvert$
gives \eqref{eq:et-gaussian-lower}.
\end{proof}

\subsubsection{Stationary bias}
\label{subsec:common-overdamped-proof}

\begin{proof}[Proof of \Cref{prop:heun-oscillatory-lower}]
Let \(\pi^{(1)},K,\mu\) be the scalar target, Heun kernel, and invariant
distribution in \Cref{lem:et-one-step-defect}, and put
\(\varphi(x)\deq\cos(x/\sqrt h)\). The scalar calculation gives
\[
 c h\leqslant\lvert\pi^{(1)}(K\varphi-\varphi)\rvert\,,
 \qquad \operatorname{Lip}(\sqrt h\varphi)\leqslant1\,,
 \qquad \operatorname{Lip}(K(\sqrt h\varphi))\leqslant1\,.
\]
The Lipschitz expectation bound and \(\mu K=\mu\) imply
\[
 \W_2(\pi^{(1)}K,\pi^{(1)})\geqslant ch^{3/2}\,,\qquad
 2\W_2(\mu,\pi^{(1)})\geqslant ch^{3/2}\,.
\]
For any bounded \(g\), Pinsker's inequality gives
\( \lvert(\nu-\pi^{(1)})(g)\rvert
 \leqslant\norm g_\infty\sqrt{2\operatorname{KL}(\nu\mmid\pi^{(1)})}\).
Apply this with \((\nu,g)=(\pi^{(1)}K,\varphi)\) and
\((\mu,K\varphi-\varphi)\), using
\(\norm\varphi_\infty\leqslant1\),
\(\norm{K\varphi-\varphi}_\infty\leqslant2\), and invariance. Then
\[
 \min\{\operatorname{KL}(\pi^{(1)}K\mmid\pi^{(1)}),\,
        \operatorname{KL}(\mu\mmid\pi^{(1)})\}\geqslant ch^2\,.
\]
The target and both numerical laws are products. Additivity of squared
\(\W_2\) and KL over products, summed over the \(d-1\geqslant d/2\)
oscillatory coordinates, proves \eqref{eq:overdamped-common-ET} and the KL bound.
\end{proof}

\subsection{The condition number dependence for Heun}
\label{subsec:overdamped-condition-number-proof}

We use blocks with \(m\) coordinates \(x\in\R^m\) and one coordinate
\(y\in\R\), where \(mh\leqslant1\). We take \(\ell^2\deq h/\tau\), with a fixed numerical \(\tau\) chosen
large enough, and use \(\theta\deq x/\ell\), \(\eta\deq y/\sqrt K\).

\subsubsection{A modulated block and its stationary mean}
\label{subsec:heun-modulated-block-proof}

We first record the stationary expansion used for the moment estimates.

\begin{lemma}[Stationary affine recursion with a bounded perturbation]
\label{lem:stationary-affine-perturbation}
Let \((U_n,G_n,\xi_n)_{n\in\mathbb Z}\) be jointly stationary, with
\(U_n,\xi_n\in\R^p\), \(\E\norm{U_0}^2<\infty\), and independent
\(G_n\sim\cN(0,\Id_p)\). Suppose
\[
 U_{n+1}=\Lambda U_n+\Sigma G_n+\xi_n\,,\qquad
 \Lambda\deq\operatorname{diag}(\lambda_1,\ldots,\lambda_p)\,,\quad
 \Sigma\deq\operatorname{diag}(\sigma_1,\ldots,\sigma_p)\,,
\]
where \(0\leqslant\lambda_i<1\) and \(\lvert\xi_{n,i}\rvert\leqslant b_i\)
almost surely for \(1\leqslant i\leqslant p\).
Then,
for every \(1\leqslant i\leqslant p\),
\[
 \norm{U_{0,i}}_{L^2}\leqslant
 \frac{\lvert\sigma_i\rvert}{\sqrt{1-\lambda_i^2}}+\frac{b_i}{1-\lambda_i}\,.
\]
\end{lemma}
\begin{proof}
For every integer \(N\geqslant1\), iteration gives
\[
 U_0=\Lambda^NU_{-N}
       +\sum_{j=0}^{N-1}\Lambda^j\Sigma G_{-1-j}
       +\sum_{j=0}^{N-1}\Lambda^j\xi_{-1-j}\,.
\]
For \(1\leqslant i\leqslant p\), stationarity and independence give $\norm{\lambda_i^NU_{-N,i}}_{L^2}
   =\lambda_i^N\,\norm{U_{0,i}}_{L^2}\to 0$.
Also, \(\sum_{j=0}^{\infty}\lambda_i^j\,\lvert\xi_{-1-j,i}\rvert
\leqslant b_i/(1-\lambda_i)\) almost surely.
The asserted bounds follow from the triangle inequality, by letting $N\to \infty$.
\end{proof}

We now estimate the mean error of one block.

\begin{lemma}[Mean error in a modulated block]
\label{lem:heun-modulated-block}
There are numerical \(a,c,C,h_0>0\), \(\tau>1\), and \(K_0\geqslant2\)
such that,
for every integer \(m\geqslant1\), \(K\geqslant K_0\), and
\(0<h\leqslant h_0\) with \(mh\leqslant1\), the potential with
\(\ell^2=h/\tau\),
\begin{equation}
 V(x,y)\deq\frac{\norm x^2}{2}+\frac{y^2}{2K}
       +\ell^2\mathcal A(y/\sqrt K)\sum_{i=1}^m\cos(x_i/\ell)\,,
 \qquad \mathcal A(t)\deq a\,(2+\sin t)\,,
 \label{eq:heun-modulated-potential}
\end{equation}
satisfies
\((1-5a)D\preceq\nabla^2V\preceq(1+5a)D\), where
\(D\deq\operatorname{diag}(\Id_m,K^{-1})\).
Its Heun invariant distribution \(\mu\in\mathcal P_2\) and target
\(\pi\propto\e^{-V}\) satisfy
\begin{equation}
 \frac{\lvert\E_\mu Y-\E_\pi Y\rvert}{\sqrt K}\geqslant cmh^2\,,
 \qquad
 \E_\mu\bigl[\norm X^2+Y^2/K\bigr]\leqslant Cm\,.
 \label{eq:heun-modulated-bias}
\end{equation}
\end{lemma}
\begin{proof}
Fix \(0<a\leqslant1/100\); we choose \(\tau\) large and then
\(h_0\leqslant1/\tau\) small. Constants \(C\) below are independent of
\(\tau,m,K,h\). Since \(m\ell^2\leqslant1/\tau\leqslant1\),
\[
 \norm{D^{-1/2}\nabla^2V D^{-1/2}-\Id_{m+1}}_{\mathrm{op}}
 \leqslant3a+a\sqrt{m\ell^2}+am\ell^2\leqslant5a\,.
\]
Thus, \Cref{prop:et-contractivity} gives \(\mu\in\mathcal P_2(\R^{m+1})\).
Put \(\theta_j\deq X_j/\ell\) for \(j\in[m]\), \(\eta\deq Y/\sqrt K\),
\(r\deq h/K\).
We use the Heun coefficients
\(A_t^{\mathrm{Heun}}\deq1-t+t^2/2\) and
\(\sigma_t^{\mathrm{Heun}}\deq\sqrt{2t}\,(1-t/2)\).
For fresh independent standard
Gaussians \(G_0,\ldots,G_m\), the Heun predictors, for every \(j\in[m]\), are

\[
 \begin{aligned}
 \theta_j^{\mathrm{pred}}
 &=\underbrace{(1-h)\,\theta_j+\sqrt{2\tau}\,G_j}_{\bar\theta_j^{\mathrm{pred}}}
   +h\mathcal A(\eta)\sin\theta_j\,,\\
 \eta^{\mathrm{pred}}
 &=\underbrace{(1-r)\,\eta+\sqrt{2r}\,G_0}_{\bar\eta^{\mathrm{pred}}}
   -r\ell^2\mathcal A'(\eta)\sum_{k\in[m]}\cos\theta_k\,.
 \end{aligned}
\]
The corrected outputs are
\begin{equation}
 \begin{aligned}
 \theta_j^+
 &=\underbrace{A_h^{\mathrm{Heun}}\theta_j+\frac{\sigma_h^{\mathrm{Heun}}}{\ell}\,G_j}_{\bar\theta_j}
   +\underbrace{\frac h2\,\Bigl\{(1-h)\,\mathcal A(\eta)\sin\theta_j
             +\mathcal A(\eta^{\mathrm{pred}})\sin\theta_j^{\mathrm{pred}}\Bigr\}}_{\zeta_j}\,,\\
 \eta^+
 &=\underbrace{A_r^{\mathrm{Heun}}\eta+\sigma_r^{\mathrm{Heun}}\,G_0}_{\bar\eta}
   -{\underbrace{\frac{r\ell^2}{2}\,\Bigl\{(1-r)\,\mathcal A'(\eta)
          \sum_{k\in[m]}\cos\theta_k
          +\mathcal A'(\eta^{\mathrm{pred}})
          \sum_{k\in[m]}\cos\theta_k^{\mathrm{pred}}\Bigr\}}_{\zeta_0}}\,.
 \end{aligned}
 \label{eq:heun-modulated-update}
\end{equation}
Since \(\lvert\mathcal A\rvert\leqslant3a\) and
\(\lvert\mathcal A'\rvert\leqslant a\), these perturbations satisfy
\(\lvert\zeta_j\rvert\leqslant3ah\) for \(j\in[m]\) and
\(\lvert\zeta_0\rvert\leqslant am\ell^2r\).
To apply \Cref{lem:stationary-affine-perturbation}, take \(p=m+1\) and the
stationary vector \((\theta_1,\ldots,\theta_m,\eta)\).
The perturbation vector in that lemma is
\((\zeta_1,\ldots,\zeta_m,-\zeta_0)\): the weak coordinate is indexed
by \(m+1\) in the lemma and by \(0\) in the block update.
Its first \(m\) affine
coefficients are \(A_h^{\mathrm{Heun}}\), its last is \(A_r^{\mathrm{Heun}}\), and the corresponding
Gaussian coefficients are \(\sigma_h^{\mathrm{Heun}}/\ell\) and
\(\sigma_r^{\mathrm{Heun}}\). The perturbation bounds are \(b_j=3ah\)
for \(j\in[m]\) and \(b_{m+1}=am\ell^2r\). Here,
\(1-A_t^{\mathrm{Heun}}=t\,(1-t/2)\) and
\(s_t^{\mathrm{Heun}}\deq(\sigma_t^{\mathrm{Heun}})^2/(1-(A_t^{\mathrm{Heun}})^2)=(2-t)/(2-t+t^2/2)\leqslant1\).
Thus, for every \(j\in[m]\), the lemma gives
\[
 \norm{\theta_j}_{L^2(\mu)}
 \leqslant\frac{\sqrt{s_h^{\mathrm{Heun}}}}{\ell}+\frac{3a}{1-h/2}
 \leqslant C\sqrt{\tau/h}\,,\qquad
 \norm\eta_{L^2(\mu)}
 \leqslant\sqrt{s_r^{\mathrm{Heun}}}+\frac{am\ell^2}{1-r/2}
 \leqslant1+\frac{2a}{\tau}\,.
\]
Consequently, \(\E_\mu[\norm X^2+Y^2/K]\leqslant Cm\).
In particular, for \(w(t)\deq\mathcal A(t)\mathcal A'(t)=a^2\,(2\cos t+\tfrac12\sin(2t))\),
since
\(2\cos t+\tfrac12\,\sin(2t)\geqslant\tfrac32-t^2\),
\begin{equation}
 \E_\mu w(\eta)
 \geqslant a^2\,\Bigl\{\frac32-\E_\mu\eta^2\Bigr\}
 \geqslant a^2\,\Bigl\{\frac32-(1+2a/\tau)^2\Bigr\}
 \geqslant\frac{a^2}{4}\,.
 \label{eq:heun-block-positive}
\end{equation}

By permutation symmetry, we fix coordinate \(1\) for the scalar
calculations below.

\emph{Gaussian damping.}
For \(\chi(u,v)\deq\E_\mu\e^{\boldsymbol i(u\theta_1+v\eta)}\), stationarity gives
\[
 \lvert\chi(u,v)\rvert\leqslant
 \e^{-\tau(1-h/2)^2u^2-r(1-r/2)^2v^2}\,\lvert\chi(A_h^{\mathrm{Heun}}u,A_r^{\mathrm{Heun}}v)\rvert
       +3ah\,\lvert u\rvert+am\ell^2r\,\lvert v\rvert\,.
\]
If \(1/2\leqslant\lvert u\rvert\leqslant2\), \(\lvert v\rvert\leqslant2\), and
\(J\deq\lfloor1/(8h)\rfloor\), then \(\lvert (A_h^{\mathrm{Heun}})^ju\rvert\geqslant1/4\) for
\(0\leqslant j\leqslant J\). Iteration yields, for small enough \(h\),
\begin{equation}
 \lvert\chi(u,v)\rvert\leqslant\e^{-c\tau/h}
       +Ca h\sum_{j=0}^{J-1}\e^{-\tau j/32}\leqslant Ca h\,.
 \label{eq:heun-block-fourier}
\end{equation}
Define
\(C_\mu\deq\E_\mu[\mathcal A'(\eta)\cos\theta_1]\) and
\(C_{\mathrm{pred}}\deq\E[\mathcal A'(\eta^{\mathrm{pred}})
                                      \cos\theta_1^{\mathrm{pred}}]\).
Then, Gaussian integration and \eqref{eq:heun-block-fourier} give
\begin{equation}
 \begin{split}
 \lvert\E[\mathcal A'(\bar\eta)\cos\bar\theta_1]\rvert
 &\leqslant a\e^{-\tau(1-h/2)^2}
                  \max_{v\in\{-A_r^{\mathrm{Heun}},A_r^{\mathrm{Heun}}\}}\lvert\chi(A_h^{\mathrm{Heun}},v)\rvert
       \leqslant Ca^2h\e^{-\tau/2}\,,\\
 \lvert C_{\mathrm{pred}}\rvert
 &\leqslant a\e^{-\tau-r}\,\Bigl\{
       \max_{v\in\{r-1,1-r\}}\lvert\chi(1-h,v)\rvert
                  +3ah+am\ell^2r\Bigr\}
       \leqslant Ca^2h\e^{-\tau}\,.
 \end{split}
 \label{eq:heun-block-damped-terms}
\end{equation}
For the second line, the two predictor components have independent
Gaussian noises conditional on the input; their phase perturbations
are bounded by \(3ah\) and \(am\ell^2r\).

\emph{The surviving term.}
Equation \eqref{eq:heun-modulated-update} implies
\[
 \begin{split}
 \E\lvert\eta^{\mathrm{pred}}-\eta\rvert+\E\lvert\bar\eta-\eta\rvert&\leqslant C\sqrt r\,,\\
 \E\bigl\lvert\zeta_1-\tfrac h2\,\mathcal A(\eta)\,
             \{(1-h)\sin\theta_1+\sin\bar\theta_1^{\mathrm{pred}}\}\bigr\rvert
 &\leqslant\tfrac{ah}{2}\,\E\lvert\eta^{\mathrm{pred}}-\eta\rvert
             +\tfrac{9a^2h^2}{2}\,.
 \end{split}
\]
Expanding the output cosine and using stationarity gives
\begin{align*}
    C_\mu &=\E[\mathcal A'(\bar\eta)\cos\bar\theta_1]
       -\tfrac h2\,\E[w(\eta)\sin\bar\theta_1\,
                    \{(1-h)\sin\theta_1+\sin\bar\theta_1^{\mathrm{pred}}\}]+R\,, \\
 \lvert R\rvert
 &\leqslant Ca^2\,(m\ell^2r+h\sqrt r+ah^2)\,.
 \end{align*}

Replace \(\bar\theta_1\) by
\(\bar\theta_1^{\mathrm{pred}}\) in the second term in the preceding Taylor expansion of
\(C_\mu\). Indeed,
\[
 \bar\theta_1-\bar\theta_1^{\mathrm{pred}}
 =\frac{h^2}{2}\,\theta_1-\frac h2\,\sqrt{2\tau}\,G_1\,,
 \qquad
 \E\lvert\bar\theta_1-\bar\theta_1^{\mathrm{pred}}\rvert
 \leqslant C\sqrt\tau\,h\,.
\]
Thus this replacement costs at most
\(Ca^2\sqrt\tau\,h^2\leqslant Ca^2h^{3/2}\).
Conditional on the input, \(\bar\theta_1^{\mathrm{pred}}\) is Gaussian
with variance \(2\tau\).
Consequently,
\[
 \begin{split}
 \bigl\lvert\E[w(\eta)\sin\bar\theta_1^{\mathrm{pred}}\sin\theta_1]\bigr\rvert
 &\leqslant3a^2\e^{-\tau}\,,\\
 \E[w(\eta)\sin^2\bar\theta_1^{\mathrm{pred}}]
 &=\frac12\,\E_\mu w(\eta)
   -\frac{\e^{-4\tau}}2\,
       \E[w(\eta)\cos(2(1-h)\theta_1)]\,.
 \end{split}
\]
Combining these two estimates with the Taylor expansion and
\eqref{eq:heun-block-damped-terms}, \(mh\leqslant1\), and
\(\tau h\leqslant1\) proves
\begin{equation}
 \bigl\lvert C_\mu+\tfrac h4\,\E_\mu w(\eta)\bigr\rvert
 \leqslant Ca^2h\,\Bigl\{\e^{-\tau/2}+\frac1{K\tau}
                  +\sqrt h\Bigr\}\,.
 \label{eq:heun-block-response}
\end{equation}
\emph{Comparison of means.}
Let \(C_\pi\deq\E_\pi[\mathcal A'(\eta)\cos\theta_1]\).
The conditional target density relative to \(Z\sim\cN(0,1)\) satisfies
\[
 \Bigl\lvert\frac{\e^{-\ell^2\mathcal A(\eta)\cos(Z/\ell)}}
                  {\E\e^{-\ell^2\mathcal A(\eta)\cos(Z/\ell)}}-1\Bigr\rvert
 \leqslant\e^{6a\ell^2}-1\leqslant Ca\ell^2\,,
 \qquad \lvert C_\pi\rvert\leqslant a\e^{-1/(2\ell^2)}+Ca^2\ell^2
                 \leqslant Ca^2h/\tau\,.
\]
Consequently \eqref{eq:heun-block-damped-terms} and
\eqref{eq:heun-block-response} imply
\[
 \frac{(1-r)\,C_\mu+C_{\mathrm{pred}}}{2-r}-C_\pi
 =-\tfrac h8\,\E_\mu w(\eta)
      +O\bigl(a^2h\,\{\e^{-\tau/2}+\tau^{-1}
                         +\sqrt h\}\bigr)\,.
\]
Choose the fixed numerical \(\tau\) so large, and then \(h_0\) so small,
that the error is at most \(a^2h/64\).
By \eqref{eq:heun-block-positive}, the displayed difference is at most
\(-a^2h/64\). Take expectations in
\(\eta^+=\bar\eta-\zeta_0\) from \eqref{eq:heun-modulated-update}
under stationarity. Since \(\E_\mu\eta^+=\E_\mu\eta\),
\(\E\bar\eta=A_r^{\mathrm{Heun}}\,\E_\mu\eta\), and, by permutation symmetry,
\(\E\zeta_0=\frac{r\ell^2m}{2}\,\{(1-r)\,C_\mu+C_{\mathrm{pred}}\}\),
we obtain
\(\E_\mu\eta=-m\ell^2\,\{(1-r)\,C_\mu+C_{\mathrm{pred}}\}/(2-r)\).
Integrating \(\partial_yV\) against the target gives
\(\E_\pi\eta=-m\ell^2C_\pi\). Together these give
\[
 \E_\mu\eta-\E_\pi\eta
   =-m\ell^2\,\Bigl\{\frac{(1-r)\,C_\mu+C_{\mathrm{pred}}}{2-r}-C_\pi\Bigr\}
   \geqslant\frac{a^2}{64\tau}\,mh^2\,.
\]
This proves the mean error in \eqref{eq:heun-modulated-bias}.
\end{proof}

\subsubsection{Products and the iteration lower bound}
\label{subsec:heun-product-lower-proof}

\begin{proof}[Proof of \Cref{thm:et-linear-kappa-lower}]
\emph{Products of blocks.}

Reserve one coordinate with potential \(z^2/(2K)\). Set
\[
 m\deq\min\{\lfloor(d-1)/2\rfloor,\lfloor h^{-1}\rfloor\}\,,\qquad
 q\deq\lfloor(d-1)/(m+1)\rfloor\,,
\]
Let \(r\in\R^{d-1-q(m+1)}\) denote the quadratic padding
coordinates and \(z\in\R\) the reserved coordinate. With
\(x^{(j)}\in\R^m\), \(y_j\in\R\), and \(\ell^2=h/\tau\), the full
normalized potential is explicitly
\begin{equation}
 \begin{aligned}
 V_h\bigl((x^{(j)},y_j)_{j=1}^q,r,z\bigr)
 &\deq\sum_{j=1}^q\Bigl\{
       \frac{\norm{x^{(j)}}^2}{2}+\frac{y_j^2}{2K}\\
 &\qquad\qquad+a\ell^2\bigl(2+\sin(y_j/\sqrt K)\bigr)
          \sum_{i=1}^m\cos(x_i^{(j)}/\ell)\Bigr\}
       +\frac{\norm r^2}{2}+\frac{z^2}{2K}\,.
 \end{aligned}
 \label{eq:heun-product-potential}
\end{equation}
Thus it consists of \(q\) independent blocks from
\eqref{eq:heun-modulated-potential}, the padding coordinates, and the
reserved coordinate. The full potential satisfies
\((1-5a)\Id_d/K\preceq\nabla^2V_h\preceq(1+5a)\Id_d\).
Tensorization of the invariant distribution follows from independence
of the coordinate blocks and uniqueness. Its vector of \(y\)-coordinate
means and
\eqref{eq:heun-modulated-bias} imply
\[
 K^{-1/2}\,\W_2(\widehat\pi_h^{\mathrm{Heun}},\pi_h)
 \geqslant c\sqrt q\,mh^2
 \asymp\min\{dh^2,\sqrt d\,h^{3/2}\}
 \asymp\frac{dh^2}{\sqrt{1+dh}}\,.
\]
Indeed, \(m\asymp\min\{d,h^{-1}\}\) and \(q\asymp d/m\).

\emph{Initialization and relaxation.}
Initialize all but the reserved coordinate at their numerical invariant
distributions, and set \(Z_0\deq\sqrt K\). For each coupled block,
\eqref{eq:heun-modulated-bias} gives
\[
 \E_\mu[V(X,Y)-\min V]
 \leqslant\tfrac12\,\E_\mu[\norm X^2+Y^2/K]+6am\ell^2\leqslant {Cm}\,.
\]
Thus, with \(d-1-q\,(m+1)\) quadratic padding coordinates,
\[
 \E[V_h(X_0)-V_h(x_\star)]
 \leqslant Cq\,(m+1)+C\,\{d-1-q\,(m+1)\}+\tfrac12\leqslant Cd\,.
\]
The block means remain stationary, while the reserved mean is
\(\sqrt K (A_{h/K}^{\mathrm{Heun}})^N\). Consequently, by projection onto these coordinates,
\[
 K^{-1/2}\,\W_2(\nu_h(K_h^{\mathrm{Heun}})^N,\pi_h)
 \geqslant c\,\frac{dh^2}{\sqrt{1+dh}}\,,\qquad
 K^{-1/2}\,\W_2(\nu_h(K_h^{\mathrm{Heun}})^N,\pi_h)
 \geqslant (A_{h/K}^{\mathrm{Heun}})^N\,.
\]

\emph{Original units and the step restriction.}
For prescribed \(\alpha,\beta\) with \(\beta/\alpha=\kappa\), take
\(K=\kappa\,(1-5a)/(1+5a)\) and replace each constructed potential
\(V\) by \(x\mapsto V(\sqrt{\beta/(1+5a)}\,x)\).
Thus the block parameter \(K\) is comparable to the prescribed condition number \(\kappa\).
Increase \(\kappa_0\) so that \(K\geqslant K_0\).
The Hessian bounds become \(\alpha\Id_d,\beta\Id_d\), the step used
in the construction becomes \(\beta h/(1+5a)\), and the normalized
Wasserstein distance becomes
\(\sqrt{1-5a}\,K^{-1/2}\,\W_2\). The expected energy is unchanged.
Since \(-\log A_t^{\mathrm{Heun}}\asymp t\) for small \(t>0\), sufficiently small
normalized accuracy \(\varepsilon\) therefore requires
\[
 \frac{d\,(\beta h)^2}{\sqrt{1+d\beta h}}\leqslant C\varepsilon\,,
 \qquad
 N\geqslant c\,\frac\kappa{\beta h}\,\log\frac1\varepsilon\,.
\]
Splitting the first inequality at \(d\beta h=1\) gives
\[
 \beta h\leqslant C\,\bigl\{
      (\varepsilon/d)^{1/2}+(\varepsilon^2/d)^{1/3}\bigr\}\,.
\]
The two terms balance at \(\varepsilon=d^{-1}\). Combining this bound
with the second inequality and the assumed small-step bound proves
\eqref{eq:linear-kappa-endpoint-rate}, hence also
\eqref{eq:linear-kappa-third-root}.
\end{proof}

\subsection{Underdamped exponential midpoint}
\label{subsec:exponential-midpoint-Gaussian-proof}

Unlike the oscillatory Heun construction, the underdamped lower bound
already holds for \(V(x)=\norm x^2/2\). Here \(\alpha=\beta=1\),
\(\gamma=2\), and \(\pi=\cN(0,\Id_d)\); the rescaling to general units
is given in \Cref{subsec:exponential-midpoint-lower}. We compute the invariant
covariance of a position--momentum pair and compare its position variance
with that of the target.

We use the following coefficients here and in the randomized underdamped
proofs in \Cref{app:randomized-lower}, obtained by specializing
\eqref{eq:uld-predictor} to \(\gamma=2\):
\begin{equation}
    a_t=\frac{1-\e^{-2t}}2\,,\qquad c_t=\frac{t-a_t}2\,,
    \qquad
    E_t\deq\begin{bmatrix}1&a_t\\[0.25em] 0&\e^{-2t}\end{bmatrix}\,,
    \qquad g_t\deq\begin{bmatrix}c_t\\a_t\end{bmatrix}\,.
    \label{eq:quadratic-flow-coefficients}
\end{equation}
The exponential Euler update on one coordinate pair \(z=(x,p)\), with
the gradient evaluated at the initial position \(x\), is
\(E_tz-g_tx+\eta_t\), where
\begin{equation}
    \eta_t\deq2\int_0^t
    \begin{bmatrix}a_{t-s}\\[0.25em] \e^{-2(t-s)}\end{bmatrix}\,\mathrm d B_s\,.
    \label{eq:quadratic-flow-noise}
\end{equation}

The exact quadratic drift matrix and the normalized hypocoercive metric
from \Cref{sec:underdamped-upper}, restricted to one coordinate pair, are
\[
    A\deq\begin{bmatrix}0&1\\-1&-2\end{bmatrix}\,,
    \qquad H_{\mathrm{hyp}}\deq\begin{bmatrix}2&1\\1&1\end{bmatrix}\,.
\]
For matrices, \(\norm{\cdot}_{H_{\mathrm{hyp}}}\) denotes the operator norm induced by
\(\norm{\cdot}_{H_{\mathrm{hyp}}}\) on \(\R^2\). All constants and remainders below are
independent of \(d\).

For later reference, the unit-curvature update has the
following exact state matrix and noise covariance, obtained from
\eqref{eq:EM-x}--\eqref{eq:EM-v}:
\begin{equation}
    M_h\deq
    \begin{bmatrix}
        1-ha_{h/2}\,(1-c_{h/2})&a_h-ha_{h/2}^2\\[0.5em]
      -h\e^{-h}\,(1-c_{h/2})&\e^{-2h}-h\e^{-h}a_{h/2}
    \end{bmatrix}\,,
    \label{eq:exponential-midpoint-matrix}
\end{equation}
The lower-bound kernel \(\mathcal J_h^{\mathrm{ExpMid}}(t)\) below
is a deterministic Brownian integrand; it is distinct from the Malliavin
variation \(\mathcal J\) in the upper-bound analysis.
The additive noise covariance is
\(Q_h\deq\int_0^h\mathcal J_h^{\mathrm{ExpMid}}(t)
\mathcal J_h^{\mathrm{ExpMid}}(t)^\top\,\mathrm d t\), where
\begin{equation*}
    \mathcal J_h^{\mathrm{ExpMid}}(t)\deq
    2\,\begin{bmatrix}
        a_{h-t}-h a_{h/2}a_{h/2-t}\mathbf 1_{\{t\leqslant h/2\}}\\[0.25em]
       \e^{-2(h-t)}-h\e^{-h}a_{h/2-t}\mathbf 1_{\{t\leqslant h/2\}}
    \end{bmatrix}\,.
\end{equation*}

We prove the isotropic and anisotropic statements together by
expanding the invariant covariance uniformly over the curvature.

\begin{proof}[Proof of \Cref{prop:exponential-midpoint-Gaussian-lower}]
We normalize \(\beta=1\) and \(\gamma=2\), and consider both
unit curvature and the anisotropic target with \(\alpha=\kappa^{-1}\).
We first compute a variance expansion uniform
over the curvature \(\lambda\in(0,1]\). We then show that the variance
error persists at every iteration, while a displaced coordinate of
curvature \(\kappa^{-1}\) forces the required elapsed time.

\emph{Uniform scalar covariance expansion.}
For the potential \(\lambda x^2/2\), the state matrix is
\[
 M_{h,\lambda}\deq
 \begin{bmatrix}
  1-h a_{h/2}\lambda\,(1-c_{h/2}\lambda)
       &a_h-h a_{h/2}^2\lambda\\[0.4em]
  -h\e^{-h}\lambda\,(1-c_{h/2}\lambda)
       &\e^{-2h}-h\e^{-h}a_{h/2}\lambda
 \end{bmatrix}\,.
\]
Its additive Gaussian noise, using the common Brownian path in
\eqref{eq:quadratic-flow-noise}, is
\[
 \eta_h-h\lambda\,
    \begin{bmatrix}a_{h/2}\\[0.25em]\e^{-h}\end{bmatrix}
       \, (\eta_{h/2})_x\,.
\]
Denote its noise covariance by \(Q_{h,\lambda}\), its
invariant position--momentum covariance by \(C_{h,\lambda}^{\mathrm{ExpMid}}\),
and its scalar position variance by
\(s_{h,\lambda}^{\mathrm{ExpMid}}\deq(C_{h,\lambda}^{\mathrm{ExpMid}})_{11}\).
The invariant covariance equation is
\[
 C_{h,\lambda}^{\mathrm{ExpMid}}
 =M_{h,\lambda}C_{h,\lambda}^{\mathrm{ExpMid}}M_{h,\lambda}^\top
     +Q_{h,\lambda}\,.
\]
At \(\lambda=1\), write \(C_h\deq C_{h,1}^{\mathrm{ExpMid}}\) and
\(M_h\deq M_{h,1}\), so this specializes to
\begin{equation}
 C_h=M_hC_hM_h^\top+Q_h\,.
 \label{eq:exponential-midpoint-Lyapunov}
\end{equation}
To control the remainder uniformly as \(\lambda\downarrow0\), write
the three covariance equations in the unknowns
\(\lambda(C_{h,\lambda}^{\mathrm{ExpMid}})_{11}\),
\((C_{h,\lambda}^{\mathrm{ExpMid}})_{12}\), and \((C_{h,\lambda}^{\mathrm{ExpMid}})_{22}\), and divide
each equation by \(h\). The coefficient matrix is analytic in
\((h,\lambda)\), including at \(h=\lambda=0\), and at \(h=0\) it is
\[
 \begin{bmatrix}
  0&-2&0\\1&2&-1\\0&2\lambda&4
 \end{bmatrix}\,.
\]
Its determinant is \(8\), independently of \(\lambda\).
The right-hand side is analytic as well. Hence these three unknowns
are analytic on a uniform neighborhood of
\(\{0\}\times[0,1]\). At \(\lambda=0\), their values are exactly
\[
 1\,,\qquad
 \frac{a_h-h\e^{-h}}{2a_h}\,,\qquad 1\,,
\]
with the middle expression extended continuously to zero at \(h=0\).
Thus \(\lambda(C_{h,\lambda}^{\mathrm{ExpMid}})_{11}-1\) and
\((C_{h,\lambda}^{\mathrm{ExpMid}})_{22}-1\) both have a factor \(\lambda\).
Taylor expansion of the covariance equations now gives
\begin{equation}
 \begin{aligned}
 (C_{h,\lambda}^{\mathrm{ExpMid}})_{11}&=\lambda^{-1}-\frac{h^2}{24}+O(h^3)\,,\\
 (C_{h,\lambda}^{\mathrm{ExpMid}})_{12}&=\frac{h^2}{12}+O(h^3)\,,\\
 (C_{h,\lambda}^{\mathrm{ExpMid}})_{22}&=1+\frac{\lambda h^2}{12}+O(\lambda h^3)\,.
 \end{aligned}
 \label{eq:em-uniform-curvature-covariance}
\end{equation}
All remainders are uniform over \(0<\lambda\leqslant1\).
In particular, \(C_{h,\lambda}^{\mathrm{ExpMid}}\) is positive definite for a universal
range \(0<h\leqslant h_0\). The noise covariance is also positive
definite: the unmodified noise on the second half of the step already
has this property. The covariance equation therefore implies that
every eigenvalue of \(M_{h,\lambda}\) has modulus less than one.
This proves existence and uniqueness of the Gaussian invariant
distribution throughout this universal step size range.
Indeed, stability represents the stationary state as a convergent
sum of independent Gaussian noise variables, with covariance \(C_{h,\lambda}^{\mathrm{ExpMid}}\).

For independent coordinates of curvatures
\(\lambda_1,\ldots,\lambda_d\in(0,1]\), the Gaussian Wasserstein
formula and \eqref{eq:em-uniform-curvature-covariance} give
\[
 \W_2\bigl((\bm\pi_h^{\mathrm{ExpMid}})_x,\pi\bigr)
 =\frac{h^2}{48}\,\Bigl(\sum_{i=1}^d\lambda_i\Bigr)^{1/2}
       \,\{1+O(h)\}\,.
\]
With all curvatures equal to one, this proves
\(s_h\deq s_{h,1}^{\mathrm{ExpMid}}=(C_h)_{11}
=1-h^2/24+O(h^3)\) and the isotropic lower bound
\eqref{eq:exponential-midpoint-lower}.
Taking \(\lambda_1=\cdots=\lambda_{d-1}=1\) and
\(\lambda_d=\kappa^{-1}\) proves the claimed stationary bias in
normalized units.

\emph{The variance error persists at every iteration.}
The initialization in the proposition gives each coordinate pair
covariance \(\operatorname{diag}(0,1)\). Equation
\eqref{eq:em-uniform-curvature-covariance} implies, uniformly in
\(\lambda\),
\[
 \det\bigl(C_{h,\lambda}^{\mathrm{ExpMid}}-\operatorname{diag}(0,1)\bigr)
 =\frac{h^2}{12}+O(h^3)>0\,,\qquad
 C_{h,\lambda}^{\mathrm{ExpMid}}-\operatorname{diag}(0,1)\succeq0\,.
\]
Consequently, the covariance after \(N\) updates is
\[
 C_{h,\lambda}^{\mathrm{ExpMid}}
   -M_{h,\lambda}^N\,
       \{C_{h,\lambda}^{\mathrm{ExpMid}}-\operatorname{diag}(0,1)\}
          (M_{h,\lambda}^N)^\top
 \preceq C_{h,\lambda}^{\mathrm{ExpMid}}\,.
\]
Each position variance is therefore smaller than its target variance,
with a deficit at least as large as at stationarity. Applying the
Gaussian Wasserstein formula to the first \(d-1\geqslant d/2\)
coordinates yields
\[
 \kappa^{-1/2}\,\W_2\bigl(\Law(\widehat X_N),\pi\bigr)
 \geqslant c\sqrt{\frac d\kappa}\,h^2\,.
\]

\emph{Relaxation of the displaced coordinate.}
For \(0<\lambda\leqslant1/2\) and \(h\leqslant h_0\), the two
eigenvalues of \(M_{h,\lambda}\), denoted by
\(\rho_+\) and \(\rho_-\), satisfy uniformly
\[
 0<\rho_-<\rho_+<1\,,\qquad
 1-\rho_+\asymp\lambda h\,,\qquad
 (M_{h,\lambda})_{11}\geqslant\rho_+\,.
\]
Indeed, the discriminant of the characteristic polynomial is
\(4\,(1-\lambda)h^2+O(h^3)\), and
\[
 (1-\rho_+)(1-\rho_-)
 =\det(\Id_2-M_{h,\lambda})
 =h a_h\lambda\,(1-c_{h/2}\lambda)\,.
\]
These formulas give \(1-\rho_-\asymp h\) and
\(1-\rho_+\asymp\lambda h\). The final inequality follows from
\((M_{h,\lambda})_{11}>(M_{h,\lambda})_{22}\) and
\((M_{h,\lambda})_{12}(M_{h,\lambda})_{21}<0\).
Since the initial momentum has expectation zero,
\[
 \begin{split}
 \frac{\E\widehat X_{N,d}}{\E X_{0,d}}
 &=\frac{(M_{h,\lambda})_{11}-\rho_-}{\rho_+-\rho_-}\,
            \rho_+^N
   +\frac{\rho_+-(M_{h,\lambda})_{11}}{\rho_+-\rho_-}\,
            \rho_-^N\\
 &\geqslant\rho_+^N\geqslant\e^{-CN\lambda h}\,,
 \qquad \lambda=\kappa^{-1}\,.
 \end{split}
\]
The initial position in this coordinate is \(\sqrt\kappa\).
The lower bound on Wasserstein distance by the distance of means
therefore supplies the second term in
\eqref{eq:em-anisotropic-finite-time}.

\emph{Original units and iteration count.}
The normalization in \Cref{sec:normalization} replaces the step by
\(\sqrt\beta\,h\) and the position distance by
\(\sqrt\beta\,\W_2\). Thus multiplying normalized position distances
by \(\kappa^{-1/2}\) gives the accuracy
\(\sqrt\alpha\,\W_2\) in the proposition, proving its stationary
formula and \eqref{eq:em-anisotropic-finite-time}.
Accuracy \(\varepsilon\leqslant1/2\) requires both
\[
 \sqrt\beta\,h
 \leqslant C\kappa^{1/4}d^{-1/4}\varepsilon^{1/2}\,,\qquad
 N\sqrt\beta\,h\geqslant c\kappa\log\frac1\varepsilon\,.
\]
Combining these bounds with \(\sqrt\beta\,h\leqslant h_0\)
proves \eqref{eq:em-anisotropic-iteration}.
\end{proof}

\section{Proofs of the randomized lower bounds}
\label{app:randomized-lower}

This appendix proves
\Cref{prop:overdamped-common-construction,prop:randomized-underdamped-Gaussian-lower}
on the standard Gaussian target, then gives the common proof of the
iteration lower bounds in \Cref{cor:midpoint-iteration-lower}.
The stationary position vector has a scale-Gaussian decomposition
\(\sqrt S\,G_d\), where \(S\) is a random variable independent of the standard Gaussian
\(G_d\). Its conditional covariance solves a fixed-point equation
derived from the update. Expanding this equation gives the variance
and centered fourth moment of \(S\), which
\Cref{lem:gaussian-scale-perturbation} converts into a Wasserstein lower
bound. We then obtain iteration lower bounds from initial distributions
with controlled energy and Wasserstein gaps. The roadmap is as follows.
\begin{enumerate}
\item \emph{A common covariance calculation.}
\Cref{lem:scalar-linear-mixture,lem:random-linear-mixture} give a scale-Gaussian decomposition of a stationary random vector
and compute moments of its conditional covariance, assuming it satisfies a fixed point equation at stationarity.

\item \emph{From covariance moments to transport.}
\Cref{lem:gaussian-scale-perturbation} converts the moments in the previous lemmas into the
position lower bound \(c\min\{dh^3,\sqrt{dh^3}\}\), using only the variance
and centered fourth moment, via some Gaussian analysis.

\item \emph{The coefficients for each quadrature rule.}
\Cref{subsec:randomized-overdamped-proof} computes scalar coefficients
for RLMC and PLMC. The matrix calculations for RULMC and PULMC are in
\Cref{subsec:randomized-underdamped-proof,subsec:poisson-underdamped-proof}.
These are then combined with the previous results to give $W_2$ bounds at stationarity.

\item {\emph{From bias to iteration lower bounds.}
The common proof in
\Cref{subsec:lower-iteration-proof} combines the stationarity lower bounds with a particular choice of initialization, converting the error bounds into iteration lower bounds
for all the methods.}
\end{enumerate}
All constants are independent of \(d\) and, for the Poisson methods,
of \(k\).

\subsection{Scale-Gaussian decompositions from random linear recursions}
\label{subsec:random-linear-mixtures}

A scale-Gaussian decomposition represents a random vector as an independent
Gaussian multiplied by a random scalar or matrix. Conditional
on this scale, the vector is Gaussian; its unconditional distribution
need not be Gaussian. We first state the scalar case used for RLMC and PLMC.
\par
Henceforth throughout this appendix, \(S\) denotes a random conditional
state covariance matrix, \(s\) a scalar conditional position variance,
and bars denote their expectations. Each algorithm has its own
\(S\) or \(s\).

\begin{lemma}[Scalar scale-Gaussian decomposition]
\label{lem:scalar-linear-mixture}
Let \(X^+\deq mX+\sqrt q\,G_d\), where \(q\geqslant0\), the pair
\((m,q)\) is fresh, and \(G_d\sim\cN(0,\Id_d)\) is independent of
\((X,m,q)\). If \(\E m^8<1\) and \(\E q^4<\infty\), the unique
invariant distribution in \(\mathcal P_2(\R^d) \) is
\(\Law(\sqrt s\,G_d)\), with \(s\geqslant0\) independent of \(G_d\), and
\[
 s\stackrel{\mathrm d}=m^2s'+q\,,\qquad s'\perp(m,q)\,,\qquad s'\stackrel{\mathrm d}=s\,.
\]
The moments satisfy
\begin{equation}
 \begin{split}
 \overline s\deq\E s&=\frac{\E q}{1-\E m^2}\,,\qquad
 \operatorname{Var}(s)=\frac{\operatorname{Var}(m^2\,\overline s+q)}{1-\E m^4}\,, \qquad
 \E\lvert s-\overline s\rvert^4\leqslant\frac{C\,\E\lvert m^2\,\overline s+q-\overline s\rvert^4}{(1-\sqrt{\E m^8})^2}\,.
 \end{split}
 \label{eq:linear-mixture-scalar-moments}
\end{equation}
\end{lemma}

This is the \(p=1\) case of the following matrix statement,
with \(M=m\) and \(Q=q\).

\begin{lemma}[Matrix scale-Gaussian decomposition]
\label{lem:random-linear-mixture}
Consider \(Z_i^+\deq MZ_i+Q^{1/2}G_i\), \(1\leqslant i\leqslant d\),
with a fresh common pair \((M,Q)\), \(Q\succeq0\), and independent
\(G_i\sim\cN(0,\Id_p)\), independent of the pair and the input.
Suppose \(\E\norm M_{\mathrm{op}}^8<1\) and
\(\E\norm Q_{\mathrm{HS}}^4<\infty\), allowing a fixed linear change
of state coordinates before checking these conditions.
Under the unique invariant distribution in
\(\mathcal P_2((\R^p)^d)\), the random vectors admit the joint decomposition
\((Z_i)_{i=1}^d\stackrel{\mathrm d}=(S^{1/2}G_i)_{i=1}^d\),
where the \(G_i\) are independent standard
Gaussians, independent of the random matrix \(S\succeq0\). Thus,
conditional on \(S\), their joint distribution is
\(\cN(0,S)^{\otimes d}\), and
\begin{equation}
 S\stackrel{\mathrm d}=MS'M^\top+Q\,,\qquad
 S'\perp(M,Q)\,,\qquad
 \overline S\deq\E S=\E[M\overline S M^\top+Q]\,.
 \label{eq:linear-mixture-fixed-point}
\end{equation}
Here \(S'\stackrel{\mathrm d}=S\). For a symmetric \(p\times p\)
matrix \(T\), write \({\operatorname{vech}(T)}\) for the column vector of its upper triangular
entries, listed row by row.
This convention retains each off-diagonal entry once, without
rescaling it; in particular,
\(\operatorname{vech}(T)=(T_{11},T_{12},T_{22})^\top\) when \(p=2\). Let \(L\) be the random matrix that represents the linear transformation
\(T\mapsto MTM^\top\) on these vectors, determined by
\[
 L\operatorname{vech}(T)=\operatorname{vech}(MTM^\top)\qquad\text{for every symmetric matrix }T\,.
\]
Define the vector
\[
 \mathfrak d\deq \operatorname{vech}(M\overline S M^\top+Q-\overline S)\,.
\]
Write
\(\Gamma\deq\operatorname{Cov}(\operatorname{vech}(S))\) for the covariance
of the random covariance coordinates. Then, the covariance and
centered fourth moment satisfy
\begin{equation}
 \Gamma=\E[L\Gamma\,L^\top]+\E[\mathfrak d\mathfrak d^\top]\,,\qquad
 \E{\norm{\operatorname{vech}(S-\overline S)}}^4\leqslant\frac{C\,\E\norm{\mathfrak d}^4}{\bigl(1-\sqrt{\E\norm M_{\mathrm{op}}^8}\bigr)^2}\,.
 \label{eq:linear-mixture-centered-moments}
\end{equation}
\end{lemma}

\begin{proof}[Proof of both lemmas]
Put \(\rho\deq\E\norm M_{\mathrm{op}}^8\).
For independent copies \((M_j,Q_j)_{j\geqslant0}\), put
\(P_0\deq\Id\), \(P_j\deq M_0\cdots M_{j-1}\). The series
\[
 Z_i\deq\sum_{j=0}^\infty P_jQ_j^{1/2}G_{i,j}\,,\qquad
 S\deq\sum_{j=0}^\infty P_jQ_jP_j^\top
\]
converge in \(L^8\) and \(L^4\), since
\(\E\norm{P_j}_{\mathrm{op}}^8\leqslant\rho^j\).
Conditioning on the coefficient pairs gives the mixture and
\eqref{eq:linear-mixture-fixed-point}; synchronous coupling contracts
squared distances by \(\E\norm M_{\mathrm{op}}^2<1\).
Put \(R\deq\operatorname{vech}(S-\overline S)\). Centering gives \(R\stackrel{\mathrm d}=LR'+\mathfrak d\), with
\(R'\perp(L,\mathfrak d)\) and \(\E R'=\E\mathfrak d=0\).
Taking expectations and covariances proves the moment identities in
\eqref{eq:linear-mixture-scalar-moments} and
\eqref{eq:linear-mixture-centered-moments}.

For the fourth moment, use the Hilbert--Schmidt norm
in the contracting state coordinates. Write \(R=\sum_{j=0}^\infty D_j\),
where \(D_j\deq L_0\cdots L_{j-1}\mathfrak d_j\) are martingale differences satisfying
\(\norm{D_j}_{L^4}\leqslant\rho^{j/4}\,\norm{\mathfrak d}_{L^4}\).
For a partial sum \(U\) and its next difference \(D\), expansion gives,
with \(u\deq\norm U_{L^4}\), \(v\deq\norm D_{L^4}\),
\[
 \E\norm{U+D}^4\leqslant u^4+6u^2v^2+4uv^3+v^4
 \leqslant(u^2+4v^2)^2\,.
\]
Thus,
\(\norm R_{L^4}^2\leqslant4\sum_{j=0}^\infty\norm{D_j}_{L^4}^2
\leqslant4\,\norm{\mathfrak d}_{L^4}^2/(1-\sqrt\rho)\).
Equivalence of norms proves
\eqref{eq:linear-mixture-centered-moments}.
\end{proof}

For the small step calculations, suppose that \(\E\norm M_{\mathrm{op}}^8\leqslant1-ch\) and
\[
 L=\Id+hF+O_{L^4}(h^2)\,,\qquad
 \mathfrak d=h^2\xi+O_{L^4}(h^3)\,,\qquad \E\xi=0\,,\qquad\norm\xi_{L^4}\leqslant C\,,
\]
where \(F\) is fixed and stable; \(L^q\) refers to the coefficient
randomness. Equation \eqref{eq:linear-mixture-centered-moments} gives
\begin{equation}
 \Gamma=h^3\int_0^\infty\e^{tF}\,\E[\xi\xi^\top]\,\e^{tF^\top}\,\mathrm d t
       +O(h^4)\,,\qquad \E{\norm{\operatorname{vech}(S-\overline S)}}^4\leqslant Ch^6\,.
 \label{eq:linear-mixture-asymptotic-moments}
\end{equation}
Indeed,
\[
 \frac{T-\E[LTL^\top]}h=-FT-TF^\top+O(h)\,\norm T\,,\qquad
 \E[\mathfrak d\mathfrak d^\top]=h^4\,\E[\xi\xi^\top]+O(h^5)\,,
\]
and \((T\mapsto-FT-TF^\top)^{-1}(Q_0)
=\int_0^\infty\e^{tF}\,Q_0\,\e^{tF^\top}\,\mathrm d t\).

\subsection{Wasserstein distance for Gaussian scale perturbations}
\label{subsec:gaussian-scale-perturbation}

The next estimate compares a Gaussian scale mixture with the
standard Gaussian target, without any assumption on the mean scale.

\begin{lemma}[Wasserstein lower bound for Gaussian scale perturbations]
\label{lem:gaussian-scale-perturbation}
Let \(X\) have a scale-Gaussian decomposition
\(X\stackrel{\mathrm d}=\sqrt s\,G_d\), where \(s\geqslant0\) is independent
of \(G_d\sim\cN(0,\Id_d)\), with
\(v_{\sqrt s}\deq\operatorname{Var}(\sqrt s)<\infty\). Then,
\[
 \W_2(\Law(X),\cN(0,\Id_d))\geqslant c\min\{d v_{\sqrt s},\sqrt{d v_{\sqrt s}}\}\,.
\]
In particular, if
\[
 v_s\deq\operatorname{Var}(s)\in(0,1]\,,\qquad
 \E\lvert s-\E s\rvert^4\leqslant C_4 v_s^2\,,
\]
then, with \(c_{C_4}>0\) depending only on \(C_4\),
\begin{equation}
 \W_2(\Law(X),\cN(0,\Id_d))
 \geqslant c_{C_4}\min\{d v_s,\sqrt{d v_s}\}\,.
 \label{eq:mixture-variance-transfer}
\end{equation}
\end{lemma}

\begin{proof}
Put \(\rho_d\deq\E\norm{G_d}\), \(v_d\deq\operatorname{Var}(\norm{G_d})\),
\(b\deq\E \sqrt s\), and \(\delta\deq\W_2(\Law(\sqrt s\,G_d),\cN(0,\Id_d))\).
Gaussian Poincar\'e gives \(v_d\leqslant1\)
\cite[Theorem~2.2.9]{ChewiBook}; hence, \(\rho_d^2=d-v_d\geqslant d-1\)
for \(d\geqslant2\), while \(\rho_1^2=2/\uppi\).
The norm is \(1\)-Lipschitz, so the differences of radial means and
standard deviations give
\[
 \rho_d\,\lvert b-1\rvert\leqslant\delta\,,\qquad
 \sqrt{d v_{\sqrt s}+b^2 v_d}\leqslant\sqrt{v_d}+\delta\,.
\]
Since \(b^2\geqslant2b-1\geqslant1-2\delta/\rho_d\),
\[
 d v_{\sqrt s}\leqslant\delta^2+2\,(\sqrt{v_d}+v_d/\rho_d)\,\delta
 \leqslant\delta^2+C\delta\,,
\]
which proves the first assertion.

For the second, let \(\overline s\deq\E s\) and let \(s'\) be an independent copy.
H\"older and Cauchy--Schwarz give
\[
 \begin{split}
 (\E\lvert s-s'\rvert)^2
 &\geqslant\frac{(\E\lvert s-s'\rvert^2)^3}{\E\lvert s-s'\rvert^4}
 =\frac{8v_s^3}{2\,\E\lvert s-\overline s\rvert^4+6v_s^2}\geqslant\frac{4v_s}{C_4+3}\,,\\
 (\E\lvert s-s'\rvert)^2
 &\leqslant\E(\sqrt s-\sqrt{s'})^2\,
                \E(\sqrt s+\sqrt{s'})^2
 \leqslant8\,\overline s\,\operatorname{Var}(\sqrt s)\,.
 \end{split}
\]
If \(\overline s\leqslant2\), then \(\operatorname{Var}(\sqrt s)\geqslant v_s/(4\,(C_4+3))\),
and the first assertion applies. If \(\overline s>2\), reverse Minkowski's inequality gives
\[
 \W_2(\Law(\sqrt s G_d),\cN(0,\Id_d))
 \geqslant\sqrt d\,(\sqrt{\overline s}-1)
 \geqslant(\sqrt2-1)\sqrt d
 \geqslant(\sqrt2-1)\min\{d v_s,\sqrt{d v_s}\}\,,
\]
since \(v_s\leqslant1\).
\end{proof}

\subsection{Overdamped randomized midpoint methods}
\label{subsec:randomized-overdamped-proof}

For \(V=\norm \cdot^2/2\), the conditional update has scalar
coefficients. We calculate their moments and apply
\Cref{lem:scalar-linear-mixture,lem:gaussian-scale-perturbation}.

\subsubsection{RLMC}
\label{subsec:rlmc-gaussian-proof}

For one uniform evaluation time, the leading centered covariance defect
is \(h^2\,(1-2\Theta)\).

\begin{proposition}[RLMC radial lower bound on a Gaussian target]
\label{prop:gaussian-radial-lower}
For every \(d\geqslant1\) and \(0<h<1\), RLMC has a unique invariant
distribution with finite second moment. For \(0<h\leqslant h_0\),
\begin{equation}
    \W_2(\widehat\pi_h^{\mathrm{RLMC}},\cN(0,\Id_d))
    \geqslant c\min\{dh^3,\sqrt{dh^3}\}\,.
    \label{eq:gaussian-radial-lower}
\end{equation}
In particular, if $dh^3\leqslant1$, then the lower bound is $cdh^3$.

More precisely, there is a random variable $s_h^{\mathrm{RLMC}} \ge 0$, independent of
$G_d\sim\cN(0,\Id_d)$, such that
\begin{equation}
    \widehat\pi_h^{\mathrm{RLMC}}=\Law(\sqrt{s_h^{\mathrm{RLMC}}}\,G_d)\,,
    \label{eq:gaussian-scale-mixture}
\end{equation}
and
\begin{align}\label{eq:gaussian-scale-mean}
    \overline s_h^{\mathrm{RLMC}}\deq\E s_h^{\mathrm{RLMC}}
    &=\frac{2-2h+h^2}{2-2h+h^2-h^3/3}
      =1+\frac{h^3}{6}+O(h^4)\,, \qquad
    \operatorname{Var}(s_h^{\mathrm{RLMC}})
    =\frac{h^3}{12}+O(h^4)\,.
\end{align}
\end{proposition}

\begin{proof}
In this proof, write \(s_h=s_h^{\mathrm{RLMC}}\),
\(m=m_h^{\mathrm{RLMC}}\), and \(q=q_h^{\mathrm{RLMC}}\), omitting
algorithm labels from the local coefficient calculations.
Conditionally on \(\Theta\sim\mathsf{Unif}[0,1]\), splitting the
Brownian increment at \(\Theta h\) in \eqref{eq:rm-step} gives
\[
 X^+=m(\Theta)\,X+\sqrt{q(\Theta)}\,G_d\,,\qquad
 m(u)\deq1-h+uh^2\,,\qquad q(u)\deq2h\,\{(1-h)^2u+1-u\}\,,
\]
with \(G_d\sim\cN(0,\Id_d)\) independent of \((X,\Theta)\).
Since \(0\leqslant m(u)\leqslant1-h+h^2<1\),
\Cref{lem:scalar-linear-mixture} gives existence, uniqueness, and
\eqref{eq:gaussian-scale-mixture}. Its scalar mean formula gives
\[
 \overline s_h=\frac{\E q(\Theta)}{1-\E m(\Theta)^2}
     =\frac{2-2h+h^2}{2-2h+h^2-h^3/3}
     =1+\frac{h^3}{6}+O(h^4)\,.
\]
For the centered covariance forcing, the identity $m(u)^2+q(u)-1=h^2\,(1-2u)+h^4u^2$
implies
\[
 \mathfrak d\deq m(\Theta)^2\,\overline s_h+q(\Theta)-\overline s_h=h^2\,(1-2\Theta)+O(h^3)\,,\qquad
 1-\E m(\Theta)^{2j}=2jh+O(h^2)\quad(j=2,4)\,.
\]
Equation \eqref{eq:linear-mixture-scalar-moments} gives
\[
 \operatorname{Var}(s_h)
   =\frac{h^4/3+O(h^5)}{4h+O(h^2)}=\frac{h^3}{12}+O(h^4)\,,\qquad
 \E\lvert s_h-\overline s_h\rvert^4\leqslant\frac{Ch^8}{h^2}=Ch^6\,.
\]
Equation \eqref{eq:mixture-variance-transfer} proves
\eqref{eq:gaussian-radial-lower}.
For \(V=\lambda\,\norm \cdot^2/2\), scaling replaces
\(h\) by \(\lambda h\) and \(\W_2\) by \(\sqrt\lambda\,\W_2\).
\end{proof}

\subsubsection{PLMC}
\label{subsec:plmc-gaussian-proof}

For PLMC, the calculation reduces to the variance of a weighted
Bernoulli sum; all required moments are uniform in \(k\geqslant2\).

\begin{proposition}[PLMC lower bound on a Gaussian target]
\label{prop:plmc-gaussian-lower}
For every \(d\geqslant1\), integer \(k\geqslant1\), and
\(0<h\leqslant h_0\), PLMC on \(V=\norm \cdot^2/2\) has a unique invariant
distribution with finite second moment which satisfies
\begin{equation}
    \W_2(\widehat\pi_h^{\mathrm{PLMC},k},\cN(0,\Id_d))
       \geqslant c\min\{dh^3,\sqrt{dh^3}\}
       \qquad(k\geqslant2)\,.
    \label{eq:poisson-gaussian-radial-lower}
\end{equation}
For \(k=1\), Poisson midpoint is ULA and the stronger bound
\begin{equation}
    \W_2(\widehat\pi_h^{\mathrm{PLMC},1},\cN(0,\Id_d))
       \geqslant c\sqrt d\,h
    \label{eq:poisson-k-one-lower}
\end{equation}
holds. The constants are uniform in \(k\).
\end{proposition}

\begin{proof}
For this proof, write
\(m_{\mathsf B}=m_{h,k}^{\mathrm{PLMC}}(\mathsf B)\),
\(q_{\mathsf B}=q_{h,k}^{\mathrm{PLMC}}(\mathsf B)\), and
\(s_{h,k}=s_{h,k}^{\mathrm{PLMC}}\); the algorithm and fixed parameters
are omitted from the local coefficients.
For the Bernoulli array \(\mathsf B\) in \eqref{eq:poisson-step}, put
\[
 N_{\mathsf B}\deq\sum_{i=1}^{k-1}\mathsf B_i\,,\qquad
 U_{\mathsf B}\deq\sum_{i=1}^{k-1}\frac ik\,\mathsf B_i\,,\qquad
 R\deq\sum_{i,j=1}^{k-1}\min\{i/k,j/k\}\,\mathsf B_i\mathsf B_j\,.
\]
Conditional on \(\mathsf B\), the quadratic update is
\begin{equation*}
 X^+=m_{\mathsf B}X+\sqrt{q_{\mathsf B}}G_d\,,\qquad
 m_{\mathsf B}\deq1-h+h^2U_{\mathsf B}\,,\qquad
 q_{\mathsf B}\deq2h-4h^2U_{\mathsf B}+2h^3R\,.
\end{equation*}
Here \(G_d\sim\cN(0,\Id_d)\) is independent of \((X,\mathsf B)\),
and \(q_{\mathsf B}\geqslant0\) is the coordinate variance of
\(\sqrt2\,(B_h-h\sum_{i=1}^{k-1}\mathsf B_iB_{hi/k})\).
For \(t\geqslant0\),
\begin{equation}
 \E\e^{tN_{\mathsf B}}
 =\Bigl(1+\frac{\e^t-1}{k}\Bigr)^{k-1}
 \leqslant\exp(\e^t-1)\,,\qquad U_{\mathsf B}\leqslant N_{\mathsf B}\,,\quad
 R\leqslant N_{\mathsf B}^2\,.
 \label{eq:bernoulli-count-moments}
\end{equation}
Consequently all fixed moments below are uniform in \(k\), and
\[
 \E m_{\mathsf B}^{2j}=1-2jh+O(h^2)\quad(1\leqslant j\leqslant4)\,.
\]
\Cref{lem:scalar-linear-mixture} gives the invariant distribution
\(\Law(\sqrt{s_{h,k}}G_d)\). Since
\[
 m_{\mathsf B}^2+q_{\mathsf B}-1
 =h^2\,(1-2U_{\mathsf B})+2h^3\,(R-U_{\mathsf B})+h^4U_{\mathsf B}^2\,,\qquad
 \E U_{\mathsf B}=\frac{k-1}{2k}\,,
\]
\eqref{eq:linear-mixture-scalar-moments} gives
\[
 \overline s_{h,k}\deq\E s_{h,k}\,,\qquad \overline s_{h,k}-1
 =\frac{\E[m_{\mathsf B}^2+q_{\mathsf B}-1]}{1-\E m_{\mathsf B}^2}
 =O(h/k+h^2)\,.
\]
For \(k\geqslant2\), the centered forcing satisfies
\[
 \mathfrak d\deq m_{\mathsf B}^2\,\overline s_{h,k}+q_{\mathsf B}-\overline s_{h,k}
       =-2h^2\,(U_{\mathsf B}-\E U_{\mathsf B})+O_{L^4}(h^3)\,.
\]
Indeed, centering the defect removes its mean, and
\((\overline s_{h,k}-1)(m_{\mathsf B}^2-\E m_{\mathsf B}^2)=O_{L^4}(h^3)\).
Independence of the indicators gives
\begin{equation}
 \operatorname{Var}(U_{\mathsf B})
 =\frac1k\,\Bigl(1-\frac1k\Bigr)\sum_{i=1}^{k-1}\frac{i^2}{k^2}
 =\frac{(k-1)^2\,(2k-1)}{6k^3}\in[1/16,1/3)\,.
 \label{eq:poisson-quadrature-variance}
\end{equation}
Equation \eqref{eq:linear-mixture-scalar-moments} now gives
\[
 \operatorname{Var}(s_{h,k})
 =h^3\operatorname{Var}(U_{\mathsf B})+O(h^4)\,,\qquad
 \E\lvert s_{h,k}-\overline s_{h,k}\rvert^4\leqslant Ch^6\,.
\]
\Cref{lem:gaussian-scale-perturbation} proves
\eqref{eq:poisson-gaussian-radial-lower} uniformly in \(k\geqslant2\).

For \(k=1\), \(m_{\mathsf B}=1-h\), \(q_{\mathsf B}=2h\), and
\eqref{eq:linear-mixture-scalar-moments} gives the deterministic variance
\(s_{h,1}=(1-h/2)^{-1}\). Thus
\(\W_2=\sqrt d\,((1-h/2)^{-1/2}-1)\geqslant c\sqrt d\,h\),
proving \eqref{eq:poisson-k-one-lower}.
\end{proof}

\subsubsection{Stationary bias}
\label{subsec:randomized-overdamped-completion}

The scale estimates give the randomized stationary bounds immediately.
Their common conversion to iteration lower bounds is in
\Cref{subsec:lower-iteration-proof}.

\begin{proof}[Proof of \Cref{prop:overdamped-common-construction}]
Apply \Cref{prop:gaussian-radial-lower,prop:plmc-gaussian-lower} in
dimension \(d\). They give existence, uniqueness, finite second moments,
and \eqref{eq:overdamped-common-RLMC}--\eqref{eq:overdamped-common-ULA},
with constants uniform in \(k\). A fixed sufficiently small stationary
accuracy forces \(h\lesssim d^{-1/3}\) for RLMC and PLMC with
\(k\geqslant2\), and \(h\lesssim d^{-1/2}\) for PLMC with \(k=1\).
\end{proof}

\subsection{Underdamped randomized midpoint}
\label{subsec:randomized-underdamped-proof}

We apply \Cref{lem:random-linear-mixture} to a position--momentum
covariance matrix. Here and in \Cref{subsec:poisson-underdamped-proof},
\(\alpha=\beta=1\), \(\gamma=2\), and \(\pi=\cN(0,\Id_d)\).
We use the coefficients, drift matrix \(A\), and metric \(H_{\mathrm{hyp}}\) defined in
\Cref{subsec:exponential-midpoint-Gaussian-proof}, including
\eqref{eq:quadratic-flow-coefficients}--\eqref{eq:quadratic-flow-noise}.
The rescaling to general units is stated in \Cref{subsec:exponential-midpoint-lower}.
We distinguish the covariance \(\overline S_h^{\mathrm{RULMC}}\) of an invariant
coordinate pair from its random conditional covariance \(S_h^{\mathrm{RULMC}}\)
given the past quadrature variables. They satisfy
\(\overline S_h^{\mathrm{RULMC}}=\E S_h^{\mathrm{RULMC}}\).

\subsubsection{Quadratic coefficients}

Condition on the common quadrature variable \(\Theta=\theta\in[0,1]\).
The predictor is
\(X_{\theta h}^{\mathrm{pred}}=(1-c_{\theta h})x+a_{\theta h}p+(\eta_{\theta h})_1\).
Substitution into the RULMC update in \Cref{sec:underdamped-algorithms} gives
\begin{equation}
    M_h^{\mathrm{RULMC}}(\theta)\deq
    \begin{bmatrix}
        1-ha_{(1-\theta)h}\,(1-c_{\theta h})&a_h-ha_{(1-\theta)h}a_{\theta h}\\[0.5em]
      -h\e^{-2(1-\theta)h}\,(1-c_{\theta h})&\e^{-2h}-h\e^{-2(1-\theta)h}a_{\theta h}
    \end{bmatrix}\,.
    \label{eq:M-RULMC}
\end{equation}
The additive noise is
\(\eta_h-h\,\begin{bmatrix}a_{(1-\theta)h}&\e^{-2(1-\theta)h}\end{bmatrix}^{\top}\,(\eta_{\theta h})_1\).
By \eqref{eq:quadratic-flow-noise}, its Brownian kernel at time
\(t\in[0,h]\) is
\begin{equation}
    \mathcal J_{h,\theta}^{\mathrm{RULMC}}(t)\deq
    \begin{bmatrix}
        2a_{h-t}-2h a_{(1-\theta)h} a_{\theta h-t}\mathbf 1_{\{t\leqslant \theta h\}}\\[0.25em]
       2\e^{-2(h-t)}-2h\e^{-2(1-\theta)h}a_{\theta h-t}\mathbf 1_{\{t\leqslant \theta h\}}
    \end{bmatrix}\,,
    \label{eq:noise-kernel-lower}
\end{equation}
Here \(\mathcal J_{h,\theta}^{\mathrm{RULMC}}\) denotes the
deterministic Brownian noise integrand, distinct from the
Malliavin variation \(\mathcal J\) in the upper-bound analysis. We write
\(Q_h(\theta)=Q_h^{\mathrm{RULMC}}(\theta)\) within this calculation, so
\[
    Q_h(\theta)\deq\int_0^h\mathcal J_{h,\theta}^{\mathrm{RULMC}}(t)\mathcal J_{h,\theta}^{\mathrm{RULMC}}(t)^\top\,\mathrm d t\,.
\]
The state matrix is a uniform \(O(h^2)\) perturbation of
\(\e^{hA}\). Since \(A^\top H_{\mathrm{hyp}}+H_{\mathrm{hyp}}A\preceq-H_{\mathrm{hyp}}\), the exact flow satisfies
\(\norm{\e^{hA}}_{H_{\mathrm{hyp}}}\leqslant\e^{-h/2}\). Hence
\[
    \sup_{\theta\in[0,1]}\,\norm{M_h^{\mathrm{RULMC}}(\theta)}_{H_{\mathrm{hyp}}}
    \leqslant1-ch
\]
for \(0<h\leqslant h_0\). Since \(\sup_{\theta\in[0,1]}\,\norm{Q_h(\theta)}\leqslant Ch\),
\Cref{lem:random-linear-mixture} gives a unique invariant distribution
\(\bm\pi_h^{\mathrm{RULMC}}\), with conditional covariance \(S_h^{\mathrm{RULMC}}\).

\subsubsection{Conditional covariance fluctuations}

We need only second-order expansions of the update coefficients.

\begin{lemma}[Conditional covariance fluctuations]
\label{lem:conditional-covariance-fluctuations}
For every \(d\geqslant1\) and \(0<h\leqslant h_0\), there is a non-negative random variable
\(s_h^{\mathrm{RULMC}}\), independent of \(G_d\sim\cN(0,\Id_d)\), such that
\begin{equation}
    (\bm\pi_h^{\mathrm{RULMC}})_x
    =\Law\bigl(\sqrt{s_h^{\mathrm{RULMC}}}\,G_d\bigr)\,.
    \label{eq:Gaussian-scale-mixture}
\end{equation}
It satisfies
\[
    \operatorname{Var}(s_h^{\mathrm{RULMC}})
       =\frac{7h^3}{96}+O(h^4)\,, \qquad
    \E\lvert s_h^{\mathrm{RULMC}}-\E s_h^{\mathrm{RULMC}}\rvert^4\leqslant Ch^6\,.
\]
\end{lemma}

\begin{proof}
By \Cref{lem:random-linear-mixture}, \(s_h^{\mathrm{RULMC}}\deq(S_h^{\mathrm{RULMC}})_{11}\)
gives \eqref{eq:Gaussian-scale-mixture}. Put
\[
 \Psi_{h,\theta}^{\mathrm{RULMC}}(S)\deq M_h^{\mathrm{RULMC}}(\theta)S M_h^{\mathrm{RULMC}}(\theta)^\top+Q_h(\theta)\,.
\]
Uniformly in \(\theta\in[0,1]\), expansion of
\eqref{eq:M-RULMC} and \eqref{eq:noise-kernel-lower} gives
\[
 \begin{aligned}
 M_h^{\mathrm{RULMC}}(\theta)&=\Id_2+hA+
 h^2\,\begin{bmatrix}\theta-1&-1\\2-2\theta&2-\theta\end{bmatrix}+O(h^3)\,,\\
 Q_h(\theta)&=\begin{bmatrix}O(h^3)&2h^2+O(h^3)\\
                2h^2+O(h^3)&4h-8h^2+O(h^3)\end{bmatrix}\,.
 \end{aligned}
\]
Consequently,
\[
 \Psi_{h,\theta}^{\mathrm{RULMC}}(\Id_2)-\Id_2
 =h^2\,(2\theta-1)\,\begin{bmatrix}1&-1\\-1&-1\end{bmatrix}+O(h^3)\,.
\]
Since \(\E(2\Theta-1)=0\), the mean equation
\eqref{eq:linear-mixture-fixed-point} and the contraction bound give
\[
 \norm{\overline S_h^{\mathrm{RULMC}}-\Id_2}
 \leqslant\frac Ch\,\norm{\E\Psi_{h,\Theta}^{\mathrm{RULMC}}(\Id_2)-\Id_2}
 \leqslant Ch^2\,.
\]
Use \(\operatorname{vech}(S)=(S_{11},S_{12},S_{22})^\top\). Define
\(L_h^{\mathrm{RULMC}}(\theta)\) as the \(3\times3\) matrix whose action is
specified, for every symmetric \(2\times2\) matrix \(T\), by
\[
 L_h^{\mathrm{RULMC}}(\theta)\operatorname{vech}(T)
 =\operatorname{vech}\bigl(M_h^{\mathrm{RULMC}}(\theta)T M_h^{\mathrm{RULMC}}(\theta)^\top\bigr)\,.
\]
The identity \(\operatorname{vech}(AS+SA^\top)=F\operatorname{vech}(S)\) gives the expansion
\begin{equation}
 L_h^{\mathrm{RULMC}}(\theta)=\Id_3+hF+O(h^2)\,,\qquad
 F\deq\begin{bmatrix}0&2&0\\-1&-2&1\\0&-2&-4\end{bmatrix}\,.
 \label{eq:covariance-generator}
\end{equation}
Since \(\overline S_h^{\mathrm{RULMC}}=\Id_2+O(h^2)\), the centered forcing
\(\mathfrak d_h^{\mathrm{RULMC}}(\theta)\deq \operatorname{vech}(\Psi_{h,\theta}^{\mathrm{RULMC}}(\overline S_h^{\mathrm{RULMC}})
-\overline S_h^{\mathrm{RULMC}})\) satisfies
\[
 \mathfrak d_h^{\mathrm{RULMC}}(\theta)=h^2\,(2\theta-1)\,v+O(h^3)\,,\qquad
 v\deq(1,-1,-1)^\top\,.
\]
The eigenvalues of \(F\) are all \(-2\), and
\begin{equation}
 e_1^\top\e^{tF}
 =\e^{-2t}\,\begin{bmatrix}(1+t)^2&2t\,(1+t)&t^2\end{bmatrix}\,.
 \label{eq:covariance-position-row}
\end{equation}
Thus, \eqref{eq:linear-mixture-asymptotic-moments} and
\(\E(2\Theta-1)^2=1/3\) give
 \begin{align*}
 \operatorname{Var}(s_h^{\mathrm{RULMC}})
 &=\frac{h^3}{3}\,\int_0^\infty(1-2t^2)^2\e^{-4t}\,\mathrm d t+O(h^4)
   =\frac{7h^3}{96}+O(h^4)\,,\\
 \E\lvert s_h^{\mathrm{RULMC}}-\E s_h^{\mathrm{RULMC}}\rvert^4&\leqslant Ch^6\,. \qedhere
 \end{align*}
\end{proof}

\begin{proof}[Proof of \Cref{prop:randomized-underdamped-Gaussian-lower} (RULMC)]
Apply \eqref{eq:mixture-variance-transfer} to
\Cref{lem:conditional-covariance-fluctuations}.
\end{proof}

\subsection{Underdamped Poisson midpoint}
\label{subsec:poisson-underdamped-proof}

Let \(\Delta\deq h/k\), \(t_i\deq i\Delta\), and let
\(\mathsf B_i\sim\mathsf{Bernoulli}(1/k)\), \(1\leqslant i\leqslant k-1\),
be independent of one another and of the Brownian motion, as in
\eqref{eq:uld-PULMC-x}--\eqref{eq:uld-PULMC-v}. Write
\(\mathsf B\deq(\mathsf B_1,\ldots,\mathsf B_{k-1})\). The PULMC definitions
\eqref{eq:uld-PULMC-x}--\eqref{eq:uld-PULMC-v} can be written
\begin{equation}
    z^+=E_hz-g_hx+\eta_h
       +k\sum_{i=1}^{k-1}\mathsf B_i E_{h-t_{i+1}}g_\Delta
                (x-X_{t_i}^{\mathrm{pred}})\,,
    \qquad
    X_{t_i}^{\mathrm{pred}}=(1-c_{t_i})x+a_{t_i}p+(\eta_{t_i})_1\,.
    \label{eq:Poisson-quadratic-macrostep}
\end{equation}
Composing the \(k\) substeps with the gradient evaluated at the batch's
initial point and the predictor corrections selected by \(\mathsf B_i\)
gives this formula, using the same Brownian path as the predictors.
The \(i=0\) correction is zero. Reusing the gradient at the initial point,
a batch therefore uses
\(2-1/k\) gradient queries in expectation.

The PULMC covariance calculation uses the Bernoulli variance computed
for PLMC.

\begin{lemma}[PULMC conditional covariance fluctuations]
\label{lem:pulmc-covariance-fluctuations}
For every \(d\geqslant1\), integer \(k\geqslant1\), and
\(0<h\leqslant h_0\), PULMC has a unique invariant distribution
\(\bm\pi_h^{\mathrm{PULMC},k}\in\mathcal P_2\), whose position marginal is
\(\Law(\sqrt{s_{h,k}^{\mathrm{PULMC}}}G_d)\), with
\(s_{h,k}^{\mathrm{PULMC}}\geqslant0\) independent of
\(G_d\sim\cN(0,\Id_d)\). For \(k\geqslant2\),
\[
 \begin{split}
 \operatorname{Var}(s_{h,k}^{\mathrm{PULMC}})
 &=\frac{(k-1)^2\,(2k-1)}{64k^3}\,h^3+O(h^4)\,, \qquad
 \E\lvert s_{h,k}^{\mathrm{PULMC}}-\E s_{h,k}^{\mathrm{PULMC}}\rvert^4
\leqslant Ch^6\,.
 \end{split}
\]
For \(k=1\), the scale is deterministic and equals \(1+h/4+O(h^2)\).
All constants and remainders are uniform in \(k\).
\end{lemma}

\begin{proof}
In this proof, write
\(M_{h,k}(\mathsf B)=M_{h,k}^{\mathrm{PULMC}}(\mathsf B)\) and
\(Q_{h,k}(\mathsf B)=Q_{h,k}^{\mathrm{PULMC}}(\mathsf B)\), suppressing the
algorithm label in the local coefficients and covariance map.
Put \(\ell_i\deq kE_{h-t_{i+1}}g_\Delta\) and
\(N_{\mathsf B}\deq\sum_{i=1}^{k-1}\mathsf B_i\). Substitution of the
predictor in \eqref{eq:Poisson-quadratic-macrostep} gives the conditional
state matrix and additive noise
\begin{equation*}
 \begin{split}
 M_{h,k}(\mathsf B)
 &\deq E_h-g_he_1^\top+\sum_{i=1}^{k-1}\mathsf B_i\ell_i\,
                  \begin{bmatrix}c_{t_i}&-a_{t_i}\end{bmatrix}\,, \qquad
 \text{noise}=\eta_h-\sum_{i=1}^{k-1}\mathsf B_i\ell_i(\eta_{t_i})_1\,.
 \end{split}
\end{equation*}
Let \(Q_{h,k}(\mathsf B)\) be this noise's conditional covariance, and
\(Q_h^0\deq\operatorname{Cov}(\eta_h)\).
Writing \(E_{22}\deq e_2e_2^\top\), the coefficients in
\eqref{eq:quadratic-flow-coefficients} satisfy, uniformly in \(i,k\),
\[
 \norm{\ell_i}\leqslant Ch\,,\qquad
 (\ell_i)_2=h+O(h^2)\,,\qquad
 \ell_i\,\begin{bmatrix}c_{t_i}&-a_{t_i}\end{bmatrix}
       =-h^2\,\frac ik\, E_{22}+O(h^3)\,.
\]
Also, \eqref{eq:quadratic-flow-noise} gives
\(\norm{\eta_h}_{L^2}\leqslant C\sqrt h\) and
\(\norm{(\eta_{t_i})_1}_{L^2}\leqslant Ch^{3/2}\). Hence
\[
 \begin{split}
 \norm{Q_{h,k}(\mathsf B)-Q_h^0}
   &\leqslant C\,(N_{\mathsf B}h^3+N_{\mathsf B}^2h^5)\,,\\
 \norm{M_{h,k}(\mathsf B)}_{H_{\mathrm{hyp}}}
   &\leqslant1-ch+Ch^2\,(1+N_{\mathsf B})\,,\qquad
 \operatorname{tr}Q_{h,k}(\mathsf B)
   \leqslant Ch\,(1+N_{\mathsf B})^2\,.
 \end{split}
\]
The exponential moment bound \eqref{eq:bernoulli-count-moments} implies
\(\E\norm{M_{h,k}(\mathsf B)}_{H_{\mathrm{hyp}}}^8\leqslant1-ch\), uniformly in
\(k\), after reducing \(h_0\). Thus \Cref{lem:random-linear-mixture}
gives existence, uniqueness, and the position marginal
\(\Law(\sqrt{s_{h,k}^{\mathrm{PULMC}}}G_d)\), where
\(s_{h,k}^{\mathrm{PULMC}}\deq(S_{h,k}^{\mathrm{PULMC}})_{11}\).

For \(k=1\), the correction sum vanishes. Expansion of
\eqref{eq:quadratic-flow-coefficients}--\eqref{eq:quadratic-flow-noise}
gives
\[
 (E_h-g_he_1^\top)(E_h-g_he_1^\top)^\top+Q_h^0-\Id_2
 =h^2E_{22}+O(h^3)\,.
\]
The mean equation \eqref{eq:linear-mixture-fixed-point} and
\(A\Id_2/4+\Id_2 A^\top/4=-E_{22}\) imply
\(S_{h,1}^{\mathrm{PULMC}}=\Id_2+h\Id_2/4+O(h^2)\).
Indeed, the covariance operator divided by \(h\) converges to
\(T\mapsto-AT-TA^\top\), whose inverse is
\(Q\mapsto\int_0^\infty\e^{tA}Q\e^{tA^\top}\,\mathrm d t\).

For \(k\geqslant2\), write
\(\Psi_{h,k,\mathsf B}(S)\deq M_{h,k}(\mathsf B)S M_{h,k}(\mathsf B)^\top
+Q_{h,k}(\mathsf B)\).
The preceding coefficient bounds give, in every fixed \(L^q\) norm
over \(\mathsf B\),
\[
 \begin{split}
 M_{h,k}-\E M_{h,k}
 &=-h^2\sum_{i=1}^{k-1}\Bigl(\mathsf B_i-\frac1k\Bigr)
                      \frac ik\, E_{22}+O_{L^q}(h^3)\,,\\
 Q_{h,k}-\E Q_{h,k}&=O_{L^q}(h^3)\,,\qquad
 M_{h,k}=\Id_2+hA+O_{L^q}(h^2)\,.
 \end{split}
\]
In particular,
\[
 \begin{split}
 \Psi_{h,k,\mathsf B}(\Id_2)-\E\Psi_{h,k,\mathsf B}(\Id_2)
 &=h^2\xi_k(\mathsf B)E_{22}+O_{L^4}(h^3)\,, \qquad
 \xi_k(\mathsf B)\deq-2\sum_{i=1}^{k-1}
             \Bigl(\mathsf B_i-\frac1k\Bigr)\,\frac ik\,.
 \end{split}
\]
The mean defect is \(O(h^2)\). Since
\(\E\norm{M_{h,k}}_{H_{\mathrm{hyp}}}^2\leqslant1-ch\), the mean equation in
\eqref{eq:linear-mixture-fixed-point} gives
\(\overline S\deq\E S_{h,k}^{\mathrm{PULMC}}=\Id_2+O(h)\).
Centering at \(\overline S\) instead of \(\Id_2\) adds only
\[
 M_{h,k}\,(\overline S-\Id_2)M_{h,k}^\top
 -\E[M_{h,k}\,(\overline S-\Id_2)M_{h,k}^\top]=O_{L^4}(h^3)\,.
\]
In the coordinates \(\operatorname{vech}(S)=(S_{11},S_{12},S_{22})^\top\), define
\(L_{h,k}\) as the random \(3\times3\) matrix whose action on every
symmetric \(2\times2\) matrix \(T\) is specified by $L_{h,k}\operatorname{vech}(T)
 =\operatorname{vech}(M_{h,k}(\mathsf B)T M_{h,k}(\mathsf B)^\top)$.
Define the centered covariance forcing $\mathfrak d_{h,k}\deq\operatorname{vech}(\Psi_{h,k,\mathsf B}(\overline S)-\overline S)$.
Both depend on the Bernoulli array \(\mathsf B\). The coefficient bounds give
\[
 L_{h,k}=\Id_3+hF+O_{L^4}(h^2)\,,\qquad
 \mathfrak d_{h,k}=h^2\xi_k e_3+O_{L^4}(h^3)\,,
\]
with \(F\) from \eqref{eq:covariance-generator}. By
\eqref{eq:poisson-quadrature-variance},
\[
 \sigma_k^2\deq\E\xi_k^2
 =\frac{2\,(k-1)^2\,(2k-1)}{3k^3}\,.
\]
Equations \eqref{eq:linear-mixture-asymptotic-moments} and
\eqref{eq:covariance-position-row} now give
\begin{align*}
 \operatorname{Var}(s_{h,k}^{\mathrm{PULMC}})
 &=h^3\sigma_k^2\int_0^\infty t^4\e^{-4t}\,\mathrm d t+O(h^4)
   =\vartheta_kh^3+O(h^4)\,,\\
 \vartheta_k&\deq\frac{(k-1)^2(2k-1)}{64k^3}\in[3/512,1/32)\,,\qquad
 \E\lvert s_{h,k}^{\mathrm{PULMC}}-\E s_{h,k}^{\mathrm{PULMC}}\rvert^4
       \leqslant Ch^6\,. \qedhere
\end{align*}
\end{proof}

\begin{proof}[Proof of \Cref{prop:randomized-underdamped-Gaussian-lower} (PULMC)]
For \(k\geqslant2\), apply \Cref{lem:gaussian-scale-perturbation} to
\Cref{lem:pulmc-covariance-fluctuations}, using
\((k-1)^2\,(2k-1)/(64k^3)\geqslant3/512\).
For \(k=1\), the Gaussian formula gives
\(\W_2=\sqrt d\,\lvert\sqrt{s_{h,1}^{\mathrm{PULMC}}}-1\rvert
\geqslant c\sqrt d\,h\).
\end{proof}

\subsection{Covariance fluctuations before stationarity}
\label{subsec:randomized-underdamped-completion}

To keep the initial momentum exactly Gaussian and independent of the
position in~\Cref{cor:midpoint-iteration-lower}, we start the covariance recursion at \(\Id_2\) and quantify
its fluctuations before stationarity.

\begin{lemma}[Covariance fluctuations before stationarity]
\label{lem:finite-time-underdamped-covariance}
Start a Gaussian coordinate pair from \(\cN(0,\Id_2)\), independently
of the future updates, and let \(S_n\) be its conditional covariance
after \(n\) steps, given the quadrature variables. Put \(s_n\deq(S_n)_{11}\).
These are the covariances of the method under consideration;
we suppress its algorithm label and the fixed parameters \(h\) and \(k\)
in this lemma.
For RULMC and PULMC with \(k\geqslant2\),
\[
 \begin{split}
 \E S_n&=\Id_2+O(h)\,,\qquad
 \operatorname{Var}(s_n)
 =h^3\sigma^2\int_0^{nh}
       (e_1^\top\e^{tF}v)^2\,\mathrm d t+O(h^4)\,,\qquad
 \E\lvert s_n-\E s_n\rvert^4\leqslant Ch^6\,,
 \end{split}
\]
where \(F\) is defined in \eqref{eq:covariance-generator}, and we have:
\[
 \begin{array}{c|c|c}
 \text{Method}&v&\sigma^2\\ \hline
 \mathrm{RULMC}&(1,-1,-1)^\top&1/3\\
 \mathrm{PULMC}\,,\ k\geqslant2&e_3&2\,(k-1)^2(2k-1)/(3k^3)
 \end{array}
\]
For PULMC with \(k=1\), \(S_n\) is deterministic and
\[
 s_n-1=h\int_0^{nh}t^2\e^{-2t}\,\mathrm d t+O(h^2)\,.
\]
The constants and remainders are uniform over \(n\geqslant0\),
\(k\), and \(0<h\leqslant h_0\).
\end{lemma}

\begin{proof}
Write \(\Psi_j(S)\deq M_jSM_j^\top+Q_j\) for the fresh covariance
map at step \(j\), so \(S_0=\Id_2\) and \(S_{j+1}=\Psi_j(S_j)\).
The coefficient calculations in
\Cref{lem:conditional-covariance-fluctuations,lem:pulmc-covariance-fluctuations}
give
\[
 \E\norm{M_j}_{H_{\mathrm{hyp}}}^8\leqslant1-ch\,,\qquad
 \E\Psi_j(\Id_2)-\Id_2=O(h^2)\,.
\]
Thus the operator \(\mathcal T(T)\deq\E[M_jTM_j^\top]\) has
\(\norm{\mathcal T^j}\leqslant C\e^{-cjh}\) in fixed matrix coordinates.
With \(\overline S_j\deq\E S_j\), iteration gives
\[
 \overline S_n-\Id_2
 =\sum_{j=0}^{n-1}\mathcal T^j
       \bigl(\E\Psi_0(\Id_2)-\Id_2\bigr)=O(h)
\]
uniformly in \(n\).

Use the coordinates \(\operatorname{vech}(T)=(T_{11},T_{12},T_{22})^\top\), and set
\[
 R_j\deq \operatorname{vech}(S_j-\overline S_j)\,,\qquad
 L_j\operatorname{vech}(T)= \operatorname{vech}(M_jTM_j^\top)\,,\qquad
 \mathfrak d_j\deq \operatorname{vech}\bigl(\Psi_j(\overline S_j)-\overline S_{j+1}\bigr)\,.
\]
Then, \(R_{j+1}=L_jR_j+\mathfrak d_j\), with \(\E R_j=\E\mathfrak d_j=0\),
and \(R_j\) is independent of \((L_j,\mathfrak d_j)\).
The expansions in
\eqref{eq:M-RULMC} and
\eqref{eq:Poisson-quadratic-macrostep} give
\(M_j-\E M_j=O_{L^8}(h^2)\). Together with
\(\overline S_j-\Id_2=O(h)\), this shows that centering the
covariance map at \(\overline S_j\) instead of \(\Id_2\) changes its
centered forcing by only \(O_{L^4}(h^3)\). The centered defects
computed in the two cited lemmas therefore give, uniformly in \(j\),
\[
 L_j=\Id_3+hF+O_{L^4}(h^2)\,,\qquad
 \mathfrak d_j=h^2\xi_jv+O_{L^4}(h^3)\,,\qquad
 \E\xi_j=0\,,\quad \E\xi_j^2=\sigma^2\,,\quad
 \norm{\xi_j}_{L^4}\leqslant C\,.
\]
Here \(\xi_j\deq2\Theta_j-1\) for RULMC and
\(\xi_j\deq-2\sum_{i=1}^{k-1}(\mathsf B_{j,i}-1/k)\,i/k\) for PULMC.

Let \(\Gamma_j\deq\E[R_jR_j^\top]\) and
\(\mathcal C(T)\deq\E[L_jTL_j^\top]\).
Independence and centering give
\[
 \Gamma_{j+1}=\mathcal C(\Gamma_j)
                   +h^4\sigma^2vv^\top+O(h^5)\,,\qquad \Gamma_0=0\,.
\]
The contraction estimate implies
\(\norm{\mathcal C^j}\leqslant C\e^{-cjh}\), and
\(\mathcal C(T)=T+h\,(FT+TF^\top)+O(h^2)\,\norm T\).
The matrices
\[
 \widetilde\Gamma_j\deq h^3\sigma^2\int_0^{jh}
          \e^{tF}vv^\top\e^{tF^\top}\,\mathrm d t
\]
satisfy the same recursion with a residual \(O(h^5)\): stability of
\(F\) gives \(\widetilde\Gamma_j=O(h^3)\) uniformly, so the
\(O(h^2)\) difference between \(\mathcal C\) and the exact covariance
flow contributes only \(O(h^5)\).
Summing these residuals using contraction gives
\(\Gamma_j=\widetilde\Gamma_j+O(h^4)\) uniformly in \(j\), proving the
variance formula.

For the fourth moment, the finite expansion
\[
 R_n=\sum_{j=0}^{n-1}L_{n-1}\cdots L_{j+1}\mathfrak d_j
\]
has centered martingale differences when the coefficient pairs are
revealed in the order \(n-1,n-2,\ldots,0\). The fourth moment
argument in \Cref{lem:random-linear-mixture}, applied to this finite
sum, gives
\[
 \norm{R_n}_{L^4}^2
 \leqslant C\sum_{j=0}^{n-1}
       \e^{-c(n-1-j)h}\,\norm{\mathfrak d_j}_{L^4}^2
 \leqslant Ch^3\,.
\]

For PULMC with \(k=1\), the deterministic covariance defect computed
in \Cref{lem:pulmc-covariance-fluctuations} is
\(\Psi_j(\Id_2)-\Id_2=h^2e_2e_2^\top+O(h^3)\), and
\(M_j=\Id_2+hA+O(h^2)\).
The same comparison of recursions, now for \(S_n-\Id_2\), gives
\[
 S_n-\Id_2
 =h\int_0^{nh}\e^{tA}e_2e_2^\top\e^{tA^\top}\,\mathrm d t
       +O(h^2)
\]
uniformly in \(n\). Since \(e_1^\top\e^{tA}e_2=t\e^{-t}\), its
position entry is the asserted formula.
\end{proof}

\subsection{From bias to iteration lower bounds}
\label{subsec:lower-iteration-proof}

The idea of the following proof is that the displaced final
coordinate forces the elapsed time \(Nh\) to exceed a universal constant
\(\tau\) before the accuracy can be reached. The remaining coordinates
start at, or close to, numerical stationarity; their bias, quantified in
the preceding results, then forces \(h\) to be small. The underdamped
case requires more work, since \(P_0\sim\cN(0,\Id_d)\) is independent
of \(X_0\) and the initial distribution is therefore not numerically
stationary, but the principle is the same.

\begin{proof}[Proof of \Cref{cor:midpoint-iteration-lower}]
We first prove the bounds at bounded condition number.
For each overdamped method, let \(\mu_h\) be the numerical invariant
distribution of its first \(d-1\) coordinates. Use the position starts
\[
 \Law(X_0)=
 \begin{cases}
  \mu_h\otimes\delta_1\,,&\text{Heun, RLMC, and PLMC}\,,\\
  \delta_{e_d}\,,&\text{exponential midpoint}\,,\\
  \cN(e_d,\Id_d)\,,&\text{RULMC and PULMC}\,.
 \end{cases}
\]
For every underdamped method take \(P_0\sim\cN(0,\Id_d)\)
independently of \(X_0\), and make every initialization independent of
future innovations. All the targets are centered, and \(\alpha\asymp1\), so
normalized and unnormalized accuracy differ only by universal constants.

\emph{A common lower bound on elapsed time.}
For any algorithm, let \(\overline M_h\) denote the expected
one-step coefficient of the final coordinate: in the overdamped case,
the mean of the scalar coefficient multiplying the current \(d\)-th
position; in the underdamped case, the mean of the \(2\times2\) state
matrix acting on the \(d\)th position--momentum pair. For Heun and RLMC it is \(1-h+h^2/2\);
for PLMC it is \(1-h+(k-1)h^2/(2k)\). The underdamped formulas
\eqref{eq:exponential-midpoint-matrix}, \eqref{eq:M-RULMC},
and \eqref{eq:Poisson-quadratic-macrostep}, together
with \eqref{eq:bernoulli-count-moments}, give
\(\norm{\overline M_h-\Id}\leqslant Ch\), uniformly in \(k\).
The same bound holds for the overdamped case. Independence of fresh
coefficients from the current state therefore gives, for the difference of the mean in the $d$-th position coordinate from the initial mean,
\[
 \lvert\E X_{N,d}-1\rvert
 \leqslant\norm{\overline M_h^N-\Id}
 \leqslant(1+Ch)^N-1\leqslant\e^{CNh}-1\,.
\]
Fix a universal \(\tau>0\) so that this is at most \(1/2\) when
\(Nh\leqslant\tau\). Since $W_2$ is lower bounded by the difference
of means, accuracy \(\varepsilon\leqslant\varepsilon_0\) for a small
enough universal \(\varepsilon_0\) requires \(Nh>\tau\).

\emph{Error bounds after that elapsed time.}
For the overdamped methods, the first \(d-1\) block of coordinates retain the stationary
marginal distribution \(\mu_h\) at every iteration. Indeed,
the updates of this block do not involve the final coordinate; after
averaging over the quadrature variables, they form the same method in
dimension \(d-1\), for which \(\mu_h\) is invariant. Conditional on the
quadrature variables the law of the block does change; only its average
over these variables is stationary.
Projection onto this block, the scalar Heun bound from the proof of
\Cref{prop:heun-oscillatory-lower} summed over its \(d-1\geqslant d/2\)
oscillatory coordinates, and
\Cref{prop:gaussian-radial-lower,prop:plmc-gaussian-lower}
in dimension \(d-1\geqslant d/2\)
give errors at least \(c\sqrt d\,h^{3/2}\) for Heun,
\(c\min\{dh^3,\sqrt{dh^3}\}\) for RLMC and PLMC with \(k\geqslant2\),
and \(c\sqrt d\,h\) for PLMC with \(k=1\).

For exponential midpoint, let \(C_{N,h}^{\mathrm{ExpMid}}\) be the covariance of each
position--momentum pair after \(N\) steps, so that
\(C_{0,h}^{\mathrm{ExpMid}}=\operatorname{diag}(0,1)\). At the fixed step size \(h\), denote
its invariant covariance by
\(C_h^{\mathrm{ExpMid}}\deq C_{h,1}^{\mathrm{ExpMid}}\). The invariant covariance expansion in
the proof of \Cref{prop:exponential-midpoint-Gaussian-lower} gives
\[
 (C_h^{\mathrm{ExpMid}}-C_{0,h}^{\mathrm{ExpMid}})_{11}=1+O(h^2)\,,\qquad
 \det(C_h^{\mathrm{ExpMid}}-C_{0,h}^{\mathrm{ExpMid}})=h^2/12+O(h^3)\,.
\]
Thus \(C_h^{\mathrm{ExpMid}}-C_{0,h}^{\mathrm{ExpMid}}\succeq0\) for sufficiently small \(h\), and
\eqref{eq:exponential-midpoint-Lyapunov} yields
\[
 C_{N,h}^{\mathrm{ExpMid}}=C_h^{\mathrm{ExpMid}}-M_h^N\,(C_h^{\mathrm{ExpMid}}-C_{0,h}^{\mathrm{ExpMid}})(M_h^N)^\top\preceq C_h^{\mathrm{ExpMid}}\,.
\]
The position covariance at time $N$ is \(s_{N,h}^{\mathrm{ExpMid}}\Id_d\), with
\(s_{N,h}^{\mathrm{ExpMid}}\deq(C_{N,h}^{\mathrm{ExpMid}})_{11}\leqslant(C_h^{\mathrm{ExpMid}})_{11}=s_{h,1}^{\mathrm{ExpMid}}=1-h^2/24+O(h^3)<1\). The Gaussian Wasserstein
formula consequently gives an error at least
\(\sqrt d\,(1-\sqrt{s_{h,1}^{\mathrm{ExpMid}}})\geqslant c\sqrt d\,h^2\) at every iteration.

For the randomized underdamped methods, the first \(d-1\) positions
are conditionally independent centered Gaussians with common variance
\(s_N\) from \Cref{lem:finite-time-underdamped-covariance}. Every
vector \(v\) in that lemma satisfies $\int_0^\tau(e_1^\top\e^{tF}v)^2\,\mathrm d t>0$
by \eqref{eq:covariance-position-row}, and \(\sigma^2\) is uniformly
bounded below. The universal \(\tau\) was fixed small by the
mean calculation for the final coordinate. Since the integrand is non-negative, for \(Nh>\tau\) the
integral in \Cref{lem:finite-time-underdamped-covariance} is at least
its value over \([0,\tau]\), which is positive and depends only on
\(\tau\). It is also at most the integral over \((0,\infty)\), which
is finite because \(\e^{tF}\) decays exponentially. After shrinking the
universal \(h_0\) so that the \(O(h^4)\) remainder is at most half the
leading term, the lemma gives
\(\operatorname{Var}(s_N)\asymp h^3\) and
\(\E\lvert s_N-\E s_N\rvert^4\lesssim h^6\) whenever \(Nh>\tau\).
Projection and \Cref{lem:gaussian-scale-perturbation}
in dimension \(d-1\geqslant d/2\) give an error at
least \(c\min\{dh^3,\sqrt{dh^3}\}\). For PULMC with \(k=1\), the
lemma instead gives \(s_N-1\geqslant ch\) for \(Nh>\tau\), by
the same comparison with the integral over \([0,\tau]\), and the Gaussian
Wasserstein formula gives \(c\sqrt d\,h\).

The preceding errors impose the step restrictions
\[
 h\lesssim
 \begin{cases}
  \varepsilon^{2/3}d^{-1/3}\,,&\text{Heun}\,,\\
  \varepsilon^{1/2}d^{-1/4}\,,&\text{exponential midpoint}\,,\\
  \varepsilon^{1/3}d^{-1/3}\,,&\text{randomized methods, excluding }k=1\,,\\
  \varepsilon d^{-1/2}\,,&\text{PLMC/PULMC with }k=1\,.
 \end{cases}
\]
Combining each restriction with \(N>\tau/h\) proves all the asserted
iteration bounds, uniformly in \(k\).

\emph{Initialization energy.}
For Heun, each scalar update differs from its quadratic update (\(a_\star=0\)) by at
most \(Ch^{3/2}\), because
\(V'(x)-x=-a_\star\sqrt h\sin(x/\sqrt h)\) on an oscillatory
coordinate, and we take a step of size $h$. The quadratic coefficient has square at most \(1-ch\)
and noise variance at most \(Ch\). Young's inequality absorbs the oscillatory
perturbation into half the contraction, giving
\[
 \E[(X_i^+)^2]\leqslant(1-ch)\,\E X_i^2+Ch\,.
\]
Thus each invariant coordinate has second moment at most \(C\);
the bounded oscillatory potential contributes at most \(Cdh\).
For RLMC and PLMC, the invariant second moments are uniformly bounded
by \eqref{eq:gaussian-scale-mean} and the mean calculation in
\Cref{prop:plmc-gaussian-lower}. Hence all overdamped starts have
\(\E[V(X_0)-V(x_\star)]\leqslant Cd\). The two underdamped starts
have energy gaps \(1/2\) and \((d+1)/2\), respectively, and the
randomized position start has \(\W_2(\Law(X_0),\pi)=1\).
Finally, strong convexity gives
\(\alpha\,\E\norm{X_0-x_\star}^2\leqslant2\,\E[V(X_0)-V(x_\star)]\),
while integration by parts gives
\(\alpha\,\E_\pi\norm{X-x_\star}^2\leqslant d\). The triangle inequality through \(\delta_{x_\star}\) therefore yields \(\alpha\W_2^2(\nu_0,\pi)\leqslant Cd\).

The final Heun assertion follows from \Cref{thm:et-linear-kappa-lower},
using \eqref{eq:linear-kappa-endpoint-rate} and deducing
the initial Wasserstein bound from its energy bound as above.
\end{proof}

\section{Analytic background and notation}
\label{app:toolbox}

This appendix records the conventions needed for the upper bounds.
For Wasserstein geometry, functional inequalities, and kinetic
contraction, we follow \cite[Sections~1.3, 2.2, and~5.3]{ChewiBook}.
For Gaussian differentiation and heat interpolation, the closest
reference is \cite[Appendix~C]{HigherOrderRHMC}.  We recall the formulas
used here, including motivation for the Witten semigroup representation;
for the other standard proofs, we refer to these references and the
monographs cited below.

The tools enter the upper bounds in three groups.
\begin{enumerate}
\item \emph{Derivative norms and Gaussian smoothing.}
\Cref{subsec:tb-schatten,subsec:tb-malliavin,subsec:tb-stein} recall
Schatten norms, Gaussian differentiation, and Stein kernels. Together
they turn bounds on derivatives of centered fluctuations into transport
bounds in \Cref{app:overdamped-proof,app:underdamped-proof}.

\item \emph{Regularity of the overdamped corrector.}
\Cref{subsec:tb-witten,subsec:tb-interpolation} explain the Witten
semigroup and the interpolation and functional calculus used in
\Cref{subsec:et-S4-multiplier}. The key identity passes a derivative
through the scalar semigroup to the semigroup on vector fields.

\item \emph{The transport cost of the underdamped conditional mean.}
\Cref{subsec:tb-sobolev} relates weighted divergence and negative Sobolev
norms to Wasserstein transport. This explains why
\Cref{subsec:uld-cancellation} estimates the divergence of the conditional
mean rather than only its size.
\end{enumerate}

\subsection{Schatten norms}
\label{subsec:tb-schatten}

For Schatten norms and their ideal and H\"older inequalities, see \cite{Simon2005} for a general reference.

For a compact operator \(A\) with singular values
\((\sigma_j)_{j\geqslant1}\), set
\[
    \norm A_{\mathfrak S_p}
    \deq\Bigl(\sum_{j\geqslant1}\sigma_j^p\Bigr)^{1/p}\,,
    \qquad
    \norm A_{\mathfrak S_\infty}\deq\norm A_{\mathrm{op}}\,.
\]
Thus, \(\mathfrak S_2\) is the Hilbert--Schmidt norm.  We use the ideal
property and the Schatten H\"older inequality:
\[
    \norm{AB}_{\mathfrak S_p}
    \leqslant\norm A_{\mathrm{op}}\,\norm B_{\mathfrak S_p}\,,
    \qquad
    \norm{AB^*}_{\mathfrak S_2}
    \leqslant\norm A_{\mathfrak S_4}\,\norm B_{\mathfrak S_4}\,.
\]
The ideal inequality also holds with the two norms interchanged.
In particular,
\[
    \norm A_{\mathfrak S_4}
    \leqslant\norm A_{\mathrm{op}}^{1/2}\,\norm A_{\mathfrak S_2}^{1/2}\,.
\]
For an operator \(\mathcal H\to\R^d\), an operator norm bound of order
one and a Hilbert--Schmidt bound of order \(\sqrt d\) therefore give
an \(\mathfrak S_4\) bound of order \(d^{1/4}\).  Squaring the latter
in the Stein kernel estimate gives the required \(\sqrt d\) dependence.

\subsection{Gaussian differentiation}
\label{subsec:tb-malliavin}

For Gaussian differentiation and Wiener chaos, see \cite[Chapter~1]{Nualart2006} for a general reference.

Let \(G\) be an isonormal Gaussian process over a separable Hilbert
space \(\mathcal H\): the random variables \(G(f)\), \(f\in\mathcal H\),
are jointly centered Gaussian, with covariance
\(\E[G(f)G(f')]=\ip f{f'}_{\mathcal H}\).  For Brownian noise on
\([0,h]\), we use
\[
    \mathcal H\deq L^2([0,h];\R^d)\,,
    \qquad G(f)\deq\int_0^h\ip{f_s}{\mathrm d B_s}\,.
\]
A direction \(f\in\mathcal H\) acts by the path shift
\(B\mapsto B+\varepsilon\int_0^\cdot f_s\,\mathrm d s\).
The Malliavin derivative \(D_GF\) is the derivative along these shifts,
initially defined for smooth cylindrical functionals and then closed
in the Gaussian Sobolev spaces \(\mathbb D^{1,p}\);
see \cite{Nualart2006,NourdinPeccati2012}.  For a vector-valued
functional \(R\), it is an operator \(D_GR:\mathcal H\to\R^d\).
Under our Hessian assumptions, differentiating the SDE gives the
usual first variation equation.  In the overdamped case,
\[
    D_GX_t f=\sqrt2\int_0^tJ_{t,s}f_s\,\mathrm d s\,,
    \qquad
    \partial_tJ_{t,s}=-\nabla^2V(X_t)J_{t,s}\,,
    \quad J_{s,s}=\Id_d\,.
\]
Convexity gives \(\norm{J_{t,s}}_{\mathrm{op}}\leqslant1\).
Only first variations, and hence only the Hessian of \(V\), enter
our Malliavin estimates.

Write \(\mathcal L_{\mathrm{OU}}\) for the Gaussian
Ornstein--Uhlenbeck generator and \( (\mathcal Q_t)_{t\ge 0} \) for its semigroup.
The Mehler formula and derivative commutation are
\[
    \mathcal Q_tF(G)
    =\E'F\bigl(\e^{-t}G+\sqrt{1-\e^{-2t}}\,G'\bigr)\,,
    \qquad
    D_G\mathcal Q_tF=\e^{-t}\mathcal Q_tD_GF\,.
\]
Here, \(G'\) is an independent copy of $G$.  Conditional Jensen makes
\(\mathcal Q_t\) a contraction also for vector and operator values,
including in \(L^4(\Omega;\mathfrak S_4)\).
On the \(n\)-th Wiener chaos, \(-\mathcal L_{\mathrm{OU}}\) acts as
multiplication by \(n\).  Its inverse is taken on centered functionals.
The formula used in the Malliavin--Stein argument is
\begin{equation}
    -D_G\mathcal L_{\mathrm{OU}}^{-1}F
    =\int_0^\infty\e^{-t}\mathcal Q_tD_GF\,\mathrm d t\,.
    \label{eq:tb-ou-gradient-representation}
\end{equation}
It follows by integrating the semigroup and commuting the derivative;
the same identity holds for cylindrical approximations on Wiener space
and passes to the limit by closability.

\paragraph{Normalized Hermite polynomials.}
For \(P\sim\cN(0,\Id_d)\) and a multi-index
\(\ell\in\mathbb N_0^d\), write
\(\lvert\ell\rvert\deq\sum_{i\in[d]}\ell_i\) and
\(\ell!\deq\prod_{i\in[d]}\ell_i!\).  The orthonormal Hermite basis is
\[
    H_\ell(p)\deq\prod_{i\in[d]}
       \frac{\mathrm{He}_{\ell_i}(p_i)}{\sqrt{\ell_i!}}\,,
    \qquad
    \E[H_\ell(P)H_k(P)]=\mathbf1_{\{\ell=k\}}\,,
\]
where the probabilists' Hermite polynomials are defined by
\(\exp(tq-t^2/2)=\sum_{n\in\mathbb N_0}\mathrm{He}_n(q)\,t^n/n!\).
In particular,
\(\mathrm{He}_0=1\), \(\mathrm{He}_1(q)=q\),
\(\mathrm{He}_2(q)=q^2-1\), and
\(\mathrm{He}_3(q)=q^3-3q\).
The \(n\)-th momentum Wiener chaos means the span of the polynomials
\(H_\ell\) with \(\lvert\ell\rvert=n\); different degrees are orthogonal.

These polynomials diagonalize the weighted Laplacian on phase space.
Let \(\pi(\mathrm d x)\propto\e^{-V(x)}\,\mathrm d x\), let
\(\bm\pi\deq\pi\otimes\cN(0,\Id_d)\) on phase space
\(\R^d\times\R^d\) with coordinates \((x,p)\), and set
\[
    \mathcal A_x\deq-\Delta_x+\nabla V\cdot\nabla_x\,,
    \qquad
    \mathcal A_p\deq-\Delta_p+p\cdot\nabla_p\,,
    \qquad
    \mathcal A\deq\mathcal A_x+\mathcal A_p\,.
\]
Integration by parts gives
\(\int\varphi\,\mathcal A\psi\,\mathrm d\bm\pi
=\int (\ip{\nabla_x\varphi}{\nabla_x\psi}
+\ip{\nabla_p\varphi}{\nabla_p\psi})\,\mathrm d\bm\pi\),
so \(\mathcal A\) is non-negative and self-adjoint in
\(L^2(\bm\pi)\); these are the operators of
\eqref{eq:uld-weighted-laplacian}.  The operator \(\mathcal A_p\) is the
negative of the Ornstein--Uhlenbeck generator in the momentum variable
\(p\), and \(\mathcal A_x\) is the weighted Laplacian of \(\pi\).
The Hermite differential equation gives
\[
    \mathcal A_pH_\ell=\lvert\ell\rvert H_\ell\,.
\]
Consequently, if \(M_\ell\colon\R^d\to\R\) is a function of position
alone, then
\[
    \mathcal A\bigl(M_\ell(x)H_\ell(p)\bigr)
    =\bigl((\mathcal A_x+\lvert\ell\rvert)M_\ell\bigr)(x)\,H_\ell(p)\,.
\]
Thus, \(\mathcal A\) acts on the coefficient of \(H_\ell\) through the
shifted spatial operator \(\mathcal A_x+\lvert\ell\rvert\).
The momentum degree projections therefore contract the full
gradient seminorm, as used in
\Cref{lem:uld-midpoint-second-generator}; see also \cite{Nualart2006}.

\subsection{Stein kernels and Gaussian convolution}
\label{subsec:tb-stein}

For Malliavin--Stein theory and Gaussian convolution bounds, see \cite{NourdinPeccati2012,PaulinWhalley2026} for general references.

A Stein kernel of a centered vector \(R\in\R^d\) is a matrix-valued
function \(\tau_R\) satisfying
\[
    \E\ip R{\varphi(R)}
    =\E\operatorname{tr}\bigl\{\tau_R(R)\nabla\varphi(R)^\top\bigr\}
\]
for smooth test fields, with \(\nabla\varphi=[\nabla\varphi_1,\ldots,\nabla\varphi_d]\)
as in \eqref{eq:et-stein-identity}. This extends Gaussian integration by parts:
if \(R\sim\cN(0,\Sigma)\), one can take \(\tau_R=\Sigma\).
For a centered Malliavin-differentiable functional, the standard
construction is
\[
    \tau_R(R)=\E\bigl[(D_GR)
       (-D_G\mathcal L_{\mathrm{OU}}^{-1}R)^*\mid R\bigr]\,;
\]
see \cite{NourdinPeccati2012}.  The Schatten inequalities and
\eqref{eq:tb-ou-gradient-representation} give the quantitative bound
used in \Cref{lem:et-malliavin-stein}.

The relevance to discretization is a gain under Gaussian convolution.
For independent \(R\sim\cN(0,\tau)\) and
\(G\sim\cN(0,\Id_d)\), with \(\sigma>0\),
\[
    \W_2\bigl(\Law(R+\sigma G),\Law(\sigma G)\bigr)
    =\norm{(\sigma^2\Id_d+\tau)^{1/2}-\sigma\Id_d}_{\mathfrak S_2}
    \leqslant\frac{\norm\tau_{\mathfrak S_2}}{2\sigma}\,.
\]
The cost depends on the covariance, so it is quadratic in the size of
\(R\).  \Cref{lem:et-gaussian-convolution}, from
\cite{PaulinWhalley2026}, replaces the covariance by a Stein kernel.
The direct interpolation proof in
\Cref{lem:et-gaussian-convolution} retains the sharp coefficient
\(1/(2\sigma)\) from the Gaussian example.
For a perturbation that depends on the Gaussian variable,
\cite[Lemma~4.11]{HigherOrderRHMC} bounds the transport cost through products of
its size and derivatives.  \Cref{lem:uld-midpoint-random-displacement}
extends this principle to a non-Gaussian base measure by removing a
divergence-free field from the interpolation current.

\subsection{The Witten Laplacian}
\label{subsec:tb-witten}

For the Witten Laplacian and its relation to diffusion operators, see \cite{HelfferNier2005} for a general reference.

We explain the representation used in \Cref{subsec:et-S4-multiplier}.
Throughout this subsection, assume
\(\alpha\Id\preceq\nabla^2V\preceq\beta\Id\), as in the overdamped
setting; the normalized bound in \Cref{app:overdamped-proof} takes \(\alpha=\kappa^{-1}\).

Recall that \(P_tf(x)=\E[f(X_t^x)]\) and
\(Lf=\Delta f-\ip{\nabla V}{\nabla f}\), where \(X_t^x\) is the
Langevin diffusion started at \(x\).  The Jacobian
\(J_t^x\deq D_xX_t^x\), with entries
\((J_t^x)_{ij}=\partial_{x_j}(X_t^x)_i\) for \(i,j\in[d]\), satisfies
\[
    \partial_tJ_t^x=-\nabla^2V(X_t^x)J_t^x\,,\qquad J_0^x=\Id\,.
\]
Differentiating under the expectation gives the chain rule identity
\[
    \nabla(P_tf)(x)
    =\E\bigl[(J_t^x)^\top\nabla f(X_t^x)\bigr]\,.
\]
This motivates averaging vector fields with the transposed Jacobian.
Differentiating the drift in \(L\) identifies the corresponding generator:
\[
    \nabla Lf=(L-\nabla^2V)\nabla f\,,
\]
where \(L\) acts componentwise and \(\nabla^2V\) by matrix multiplication.
We therefore define
\[
    L^{(1)}\deq L-\nabla^2V\,,\qquad P_t^{(1)}\deq\e^{tL^{(1)}}\,.
\]
Equivalently, \(u(t,x)\deq(P_t^{(1)}M)(x)\) solves
\[
    \partial_tu=Lu-\nabla^2V\,u\,,\qquad u(0,x)=M(x)\,.
\]
For \(M=\nabla f\), this is exactly the equation obtained by
spatially differentiating \(\partial_tP_tf=LP_tf\).  Thus
\(\nabla P_tf=P_t^{(1)}\nabla f\).  A $1$-form in Euclidean coordinates
is a column vector, and \(-L^{(1)}\) is the Witten
Laplacian on $1$-forms.  For matrix fields \(M\), these operators
act on each column separately.

The chain rule motivates the representation for gradient fields.
For arbitrary bounded $1$-forms, the Feynman--Kac representation is given in
\cite[Proposition~3.2, equation~(3.2)]{Huguet2021}; applying it to each
column gives, for bounded matrix fields \(M\),
\begin{equation}
    (P_t^{(1)}M)(x)
    =\E\bigl[(J_t^x)^\top M(X_t^x)\bigr]\,.
    \label{eq:tb-witten-feynman-kac}
\end{equation}
In Euclidean space, the deformed parallel transport in that reference
is \(J_t^x\), since its defining equation reduces to the Jacobian
equation above.

The Jacobian equation also gives the contraction used in
\Cref{lem:et-witten-endpoints}.
For \(v\in\R^d\), the lower Hessian bound gives
\[
    \frac{\mathrm d}{\mathrm d t}\,\norm{J_t^xv}^2
    =-2\,\ip{J_t^xv}{\nabla^2V(X_t^x)J_t^xv}
    \leqslant-2\alpha\,\norm{J_t^xv}^2\,,
\]
so \(\norm{J_t^x}_{\mathrm{op}}\leqslant\e^{-\alpha t}\).
Thus, \eqref{eq:tb-witten-feynman-kac} combines
averaging along the Langevin diffusion with a contracting matrix factor.
\Cref{lem:et-witten-endpoints} makes this observation quantitative in the
normalized setting of \Cref{app:overdamped-proof} and extends the
representation to \(M\in L^p(\pi;\mathfrak S_p)\), for
\(\pi\)-almost every starting point.

\subsection{Interpolation and the sectorial calculus}
\label{subsec:tb-interpolation}

For complex interpolation and sectorial functional calculus, see \cite{BerghLofstrom1976,Haase2006} for general references.

The spectral theorem bounds multipliers on
\(L^2(\pi;\mathfrak S_2)\).  To obtain the
\(L^4(\pi;\mathfrak S_4)\) bound needed for the Stein kernel, we use
complex interpolation:
\[
    [L^\infty(\pi;\mathfrak S_\infty),
       L^2(\pi;\mathfrak S_2)]_{1/2}
    =L^4(\pi;\mathfrak S_4)
\]
with equal norms.  The complex interpolation method
\cite{BerghLofstrom1976} and Stein's theorem for analytic operator
families \cite{Stein1956} transfer the endpoint semigroup estimates to
complex times on this intermediate space.

\paragraph{Why a functional calculus is needed.}
The proof of \Cref{lem:et-S4-corrector} writes the corrector gradient as
a function of the Witten Laplacian,
\(\nabla\chi=h\,r(-hL^{(1)})\,\nabla^2V\), and requires a bound on
\(r(-hL^{(1)})\) on \(L^4(\pi;\mathfrak S_4)\) that is uniform in
\(h\).  On \(L^2(\pi;\mathfrak S_2)\), the operator \(-L^{(1)}\) is
non-negative and self-adjoint (\Cref{lem:et-witten-endpoints}), and the
spectral theorem gives directly
\[
    \norm{r(-hL^{(1)})}_{L^2(\pi;\mathfrak S_2)\to L^2(\pi;\mathfrak S_2)}
    \leqslant\sup_{\lambda\geqslant0}\,\lvert r(\lambda)\rvert\,.
\]
The space \(L^4(\pi;\mathfrak S_4)\) is not a Hilbert space, so there is
no spectral theorem there.  The sectorial calculus replaces the spectral
theorem by the Cauchy integral formula; its only input is a bound on the
resolvent.

\paragraph{The Cauchy integral formula.}
For \(0<\eta<\uppi\), let
\(\Sigma_\eta\deq\{z\ne0:\lvert\arg z\rvert<\eta\}\) be the open sector
of half-angle \(\eta\) around the positive real axis.  Orient its
boundary \(\partial\Sigma_\eta\) from \(\infty\e^{\boldsymbol i\eta}\) to
\(0\) and then to \(\infty\e^{-\boldsymbol i\eta}\), so that
\(\Sigma_\eta\) lies to the left of the boundary.  Let \(g\) be holomorphic on a larger
sector \(\Sigma_\vartheta\), \(\vartheta>\eta\), with
\(\lvert g(z)\rvert\leqslant C\min\{\lvert z\rvert^s,\lvert z\rvert^{-s}\}\)
for some \(s>0\).  Cauchy's formula gives, for every
\(\lambda\in\Sigma_\eta\),
\[
    g(\lambda)
    =\frac1{2\uppi\boldsymbol i}\int_{\partial\Sigma_\eta}
       \frac{g(z)}{z-\lambda}\,\mathrm d z\,.
\]
Substituting a closed operator \(\mathcal T\) on a Banach space
\(\mathcal X\) for \(\lambda\) turns \(1/(z-\lambda)\) into the resolvent
\((z-\mathcal T)^{-1}\) and suggests the definition
\begin{equation}
    g(\mathcal T)
    \deq\frac1{2\uppi\boldsymbol i}\int_{\partial\Sigma_\eta}
       g(z)\,(z-\mathcal T)^{-1}\,\mathrm d z\,.
    \label{eq:tb-contour-calculus}
\end{equation}
This makes sense provided that the resolvent exists on
\(\partial\Sigma_\eta\) and the integral converges; by Cauchy's theorem,
it does not depend on the choice of \(\eta\).  For a
non-negative self-adjoint operator on a Hilbert space,
\eqref{eq:tb-contour-calculus} agrees with the spectral theorem, by
applying the scalar formula inside the spectral integral.

\paragraph{Sectorial operators.}
Convergence of \eqref{eq:tb-contour-calculus} is guaranteed by the
following resolvent condition.  A closed operator \(\mathcal T\) on a
Banach space \(\mathcal X\) is \emph{sectorial of angle}
\(\omega\in[0,\uppi)\) if its spectrum lies in the closed sector
\(\overline{\Sigma_\omega}\), and, for every
\(\vartheta\in(\omega,\uppi)\),
\[
    M_\vartheta(\mathcal T)
    \deq\sup\bigl\{\lvert z\rvert\,
       \norm{(z-\mathcal T)^{-1}}_{\mathcal X\to\mathcal X}:
       z\ne0\,,\ \lvert\arg z\rvert\geqslant\vartheta\bigr\}
    <\infty\,.
\]
The size \(1/\lvert z\rvert\) is that of the scalar resolvent
\(1/(z-\lambda)\) when \(z\) and \(\lambda\geqslant0\) are separated by a
fixed angle.  Since
\((z-h\mathcal T)^{-1}=h^{-1}\,(z/h-\mathcal T)^{-1}\), the constants are
invariant under scaling: \(M_\vartheta(h\mathcal T)=M_\vartheta(\mathcal T)\)
for every \(h>0\).  For example, on \(L^2(\pi;\mathfrak S_2)\), the
resolvent norm of the self-adjoint operator \(-L^{(1)}\) is the
reciprocal of the distance from \(z\) to its spectrum in \([0,\infty)\),
which is at least \(\lvert z\rvert\sin\vartheta\) when
\(\lvert\arg z\rvert\geqslant\vartheta\); so \(-L^{(1)}\) is sectorial of
angle \(0\) there.  The angle \(\uppi/4\) obtained on
\(L^4(\pi;\mathfrak S_4)\) in \Cref{lem:et-witten-analyticity} is worse,
but any angle below \(\uppi/2\) suffices here: the multiplier \(r\) has
poles on the imaginary axis, so the contour must stay in the right
half-plane.

\paragraph{From analytic semigroups to sectoriality.}
We describe a standard equivalence used in the proof of
\Cref{lem:et-witten-analyticity}; see \cite{Haase2006} for the general
statement.  The resolvent is the Laplace transform of the semigroup.  If
\(\operatorname{Re}w>0\), contractivity of \(\e^{tL^{(1)}}\) on
\(L^4(\pi;\mathfrak S_4)\) gives
\[
    (w-L^{(1)})^{-1}=\int_0^\infty\e^{-wt}\,\e^{tL^{(1)}}\,\mathrm d t\,,
    \qquad
    \norm{(w-L^{(1)})^{-1}}\leqslant\frac1{\operatorname{Re}w}\,,
\]
with operator norms on \(L^4(\pi;\mathfrak S_4)\).  With \(z=-w\), this
covers only \(\operatorname{Re}z<0\), and the bound \(1/\operatorname{Re}w\)
degenerates near the imaginary axis.  Analyticity lets us rotate the
integration ray.  By \Cref{lem:et-witten-analyticity},
\(\zeta\mapsto\e^{\zeta L^{(1)}}\) is holomorphic and contractive in
\(\lvert\arg\zeta\rvert<\uppi/4\). By Cauchy's theorem, for
\(\lvert\phi\rvert<\uppi/4\) and
\(\operatorname{Re}(w\e^{\boldsymbol i\phi})>0\),
\begin{equation}
\begin{aligned}
    (w-L^{(1)})^{-1}
    &=\e^{\boldsymbol i\phi}\int_0^\infty
       \exp(-w\e^{\boldsymbol i\phi}s)\,
       \exp(s\e^{\boldsymbol i\phi}L^{(1)})\,\mathrm d s\,,\\
    \norm{(w-L^{(1)})^{-1}}
    &\leqslant\frac1{\lvert w\rvert\cos(\arg w+\phi)}\,.
\end{aligned}
    \label{eq:tb-rotated-laplace}
\end{equation}
Now fix \(\vartheta\in(\uppi/4,\uppi/2)\) and let
\(\lvert\arg z\rvert\geqslant\vartheta\), so that
\(\lvert\arg w\rvert\leqslant\uppi-\vartheta\).  Choose \(\phi\) with
sign opposite to \(\arg w\) and
\(\lvert\phi\rvert=\min\{\lvert\arg w\rvert,(3\uppi/4-\vartheta)/2\}\).
Then \(\lvert\arg w+\phi\rvert\leqslant\uppi/2-(\vartheta-\uppi/4)/2\),
and \eqref{eq:tb-rotated-laplace} gives
\[
    \norm{(z+L^{(1)})^{-1}}
    \leqslant\frac1{\sin((\vartheta-\uppi/4)/2)\,\lvert z\rvert}\,.
\]
In particular, every \(z\) outside \(\overline{\Sigma_{\uppi/4}}\) lies in
the resolvent set, so the spectrum of \(-L^{(1)}\) lies in
\(\overline{\Sigma_{\uppi/4}}\), and the resolvent bound
\eqref{eq:et-witten-resolvent} holds with a constant depending only on
\(\vartheta\).  In general, the negative of the generator of a bounded analytic semigroup
on the sector \(\lvert\arg\zeta\rvert<\delta\) is sectorial of angle
\(\uppi/2-\delta\), and the converse is also true.

\paragraph{The multiplier estimate.}
The following estimate is standard.  It is the elementary part of
the holomorphic functional calculus for sectorial operators constructed
in \cite[Section~2.3]{Haase2006}, applied to functions that decay
polynomially at zero and at infinity after subtracting a constant and a
multiple of \((1+z)^{-1}\) (the extended Dunford--Riesz class).  For
completeness, we include the short proof, which also makes the
uniformity under scaling explicit.

\begin{lemma}[Uniform multiplier bound for sectorial operators]
\label{lem:tb-sectorial-multiplier}
Let \(\mathcal T\) be a closed operator on a Banach space
\(\mathcal X\) that is sectorial of angle \(\omega\in[0,\uppi)\), and let
\(\vartheta\in(\omega,\uppi)\).  Let \(f\colon\Sigma_\vartheta\to\mathbb C\)
be holomorphic, and suppose that there are
\(f(0),f(\infty)\in\mathbb C\) and constants \(C_f,s>0\) such that, for
\(z\in\Sigma_\vartheta\),
\[
    \lvert f(z)-f(0)\rvert\leqslant C_f\,\lvert z\rvert^s
    \quad\text{if }\lvert z\rvert\leqslant1\,,
    \qquad
    \lvert f(z)-f(\infty)\rvert\leqslant C_f\,\lvert z\rvert^{-s}
    \quad\text{if }\lvert z\rvert\geqslant1\,.
\]
For \(h>0\), define
\[
    f(h\mathcal T)
    \deq f(\infty)\,\Id+(f(0)-f(\infty))\,(\Id+h\mathcal T)^{-1}
       +g(h\mathcal T)\,,
    \quad
    g(z)\deq f(z)-f(\infty)-\frac{f(0)-f(\infty)}{1+z}\,,
\]
where \(g(h\mathcal T)\) is given by \eqref{eq:tb-contour-calculus} with
\(\eta\deq(\omega+\vartheta)/2\).  Then,
\[
    \sup_{h>0}\,\norm{f(h\mathcal T)}_{\mathcal X\to\mathcal X}
    \leqslant\lvert f(\infty)\rvert
       +\lvert f(0)-f(\infty)\rvert\,M_\eta(\mathcal T)
       +\frac{2C_gM_\eta(\mathcal T)}{\uppi\min\{s,1\}}\,,
    \quad
    C_g\deq C_f+\frac{\lvert f(0)-f(\infty)\rvert}{\cos(\vartheta/2)}\,.
\]
\end{lemma}

\begin{proof}
First, \(g\) decays at zero and at infinity.  For
\(z\in\Sigma_\vartheta\),
\[
    \lvert1+z\rvert^2
    =(1+\lvert z\rvert)^2-4\,\lvert z\rvert\sin^2\frac{\arg z}2
    \geqslant\cos^2\frac\vartheta2\,(1+\lvert z\rvert)^2\,,
\]
since \(4\,\lvert z\rvert\leqslant(1+\lvert z\rvert)^2\).  If
\(\lvert z\rvert\leqslant1\), write
\(g(z)=f(z)-f(0)+(f(0)-f(\infty))\,z/(1+z)\); if
\(\lvert z\rvert\geqslant1\), write
\(g(z)=f(z)-f(\infty)-(f(0)-f(\infty))/(1+z)\).  With the lower bound on
\(\lvert1+z\rvert\), both cases give
\[
    \lvert g(z)\rvert
    \leqslant C_g\min\{\lvert z\rvert^{s'},\lvert z\rvert^{-s'}\}\,,
    \qquad s'\deq\min\{s,1\}\,,\qquad z\in\Sigma_\vartheta\,.
\]
Next, on the rays \(z=t\e^{\pm\boldsymbol i\eta}\), \(t>0\), of
\(\partial\Sigma_\eta\), scale invariance gives
\(\norm{(z-h\mathcal T)^{-1}}\leqslant M_\eta(h\mathcal T)/t
=M_\eta(\mathcal T)/t\).  The integral in
\eqref{eq:tb-contour-calculus} therefore converges absolutely, and
\[
    \norm{g(h\mathcal T)}
    \leqslant\frac{2}{2\uppi}\int_0^\infty
       C_g\min\{t^{s'},t^{-s'}\}\,\frac{M_\eta(\mathcal T)}t\,\mathrm d t
    =\frac{2C_gM_\eta(\mathcal T)}{\uppi s'}\,.
\]
The decay of \(g\) at zero compensates for the growth of the resolvent
near the origin, and its decay at infinity makes the integral over the
unbounded contour converge.  Finally,
\((\Id+h\mathcal T)^{-1}=-(-1-h\mathcal T)^{-1}\) is a resolvent at
\(z=-1\), where \(\lvert\arg z\rvert=\uppi\geqslant\eta\), so its norm is
at most \(M_\eta(h\mathcal T)=M_\eta(\mathcal T)\).  The triangle
inequality completes the proof.
\end{proof}

For
\(\mathcal T=-L^{(1)}\), the resolvents on \(L^4(\pi;\mathfrak S_4)\) and
\(L^2(\pi;\mathfrak S_2)\) are Laplace transforms of the same semigroup,
so they agree on the intersection of the two spaces; hence so do the
operators \(f(-hL^{(1)})\) defined on either space.
\Cref{lem:et-S4-multiplier} applies \Cref{lem:tb-sectorial-multiplier}
with \(\mathcal T=-L^{(1)}\) on \(L^4(\pi;\mathfrak S_4)\) and \(f=r\).

\subsection{Weighted divergence, negative Sobolev norms, and transport}
\label{subsec:tb-sobolev}

For weighted differential operators and Wasserstein tangent norms, see \cite{HelfferNier2005} and \cite[Section~1.3]{ChewiBook} for general references.

For \(\pi(\mathrm d x)\propto\e^{-V(x)}\,\mathrm d x\), integration by
parts gives \(\partial_i^*=-\partial_i+\partial_iV\) in \(L^2(\pi)\).
If \(\Psi=(\Psi_1,\ldots,\Psi_d)\), with each \(\Psi_i\) valued in a finite-dimensional Euclidean space \(E\), the weighted divergence is
\[
    \delta_\pi\Psi\deq\sum_{j\in[d]}\partial_j^*\Psi_j
       =\sum_{j\in[d]}\bigl(-\partial_j\Psi_j+(\partial_jV)\Psi_j\bigr)\,.
\]
The Bochner--Witten identity used in
\Cref{lem:uld-midpoint-second-generator,lem:uld-midpoint-random-displacement}
is the following weighted integration by parts formula for smooth
compactly supported fields:
\begin{equation}
\begin{split}
    \norm{\delta_\pi\Psi}_2^2
    ={}&\int\sum_{i,j\in[d]}
       \ip{\partial_i\Psi_j}{\partial_j\Psi_i}_E\,\mathrm d\pi+\int\sum_{i,j\in[d]}(\partial_i\partial_jV)\,
          \ip{\Psi_i}{\Psi_j}_E\,\mathrm d\pi\,.
\end{split}
    \label{eq:tb-weighted-divergence}
\end{equation}
Indeed, differentiating the weight term gives
\(\partial_i\partial_j^*=\partial_j^*\partial_i+\partial_i\partial_jV\).
Applying the adjoint relation twice therefore yields
\[
\begin{split}
    \norm{\delta_\pi\Psi}_2^2
    &=\sum_{i\in[d]}\ip{\Psi_i}{\partial_i\delta_\pi\Psi}_{L^2(\pi;E)}\\
    &=\sum_{i,j\in[d]}\ip{\Psi_i}{\partial_j^*\partial_i\Psi_j}_{L^2(\pi;E)}
       +\int\sum_{i,j\in[d]}(\partial_i\partial_jV)\,
          \ip{\Psi_i}{\Psi_j}_E\,\mathrm d\pi\\
    &=\sum_{i,j\in[d]}\ip{\partial_j\Psi_i}{\partial_i\Psi_j}_{L^2(\pi;E)}
       +\int\sum_{i,j\in[d]}(\partial_i\partial_jV)\,
          \ip{\Psi_i}{\Psi_j}_E\,\mathrm d\pi\,.
\end{split}
\]
The second integral is abbreviated as
\(\int\ip\Psi{\nabla^2V\,\Psi}_E\,\mathrm d\pi\): the Hessian acts on the
spatial indices.  If \(\nabla^2V\preceq\beta\Id\), Cauchy--Schwarz
bounds the first sum by \(\norm{\nabla_x\Psi}_2^2\), giving
\(\norm{\delta_\pi\Psi}_2^2\leqslant\norm{\nabla_x\Psi}_2^2+\beta\,\norm\Psi_2^2\).
Thus, the divergence is controlled using the bounded Hessian, even
though its definition contains the possibly unbounded gradient of \(V\).

Let \(\mu\) have a positive density, let \(D\) be the gradient, and let
\(\delta_\mu\deq D^*\) be its adjoint in \(L^2(\mu)\).
The homogeneous space \(\dot H^1(\mu)\) is the completion of smooth
functions modulo constants under \(\norm{D\varphi}_2\).  For a
continuous functional \(r\) which is zero on constants, its dual norm has
the minimum flux characterization
\[
    \norm r_{\dot H^{-1}(\mu)}
    =\sup_{\norm{D\varphi}_2\leqslant1}\ip r\varphi
    =\inf_{\delta_\mu A=r}\,\norm A_2\,.
\]
The minimizing field belongs to the closure of gradients.  This is the
same orthogonal projection that identifies Wasserstein tangent vectors
in \cite[Theorem~1.3.20]{ChewiBook}: adding a divergence-free field
leaves the continuity equation unchanged, whereas removing that field
reduces the transport cost without changing the endpoints.

\clearpage
\printbibliography

@book{HighamKloeden2021,
  author = {Higham, D. J. and Kloeden, P. E.},
  title = {An introduction to the numerical simulation of stochastic differential equations},
  date = {2021},
  publisher = {Society for Industrial and Applied Mathematics},
}

@inproceedings{Dalalyan2017,
  title = {Further and stronger analogy between sampling and optimization: {L}angevin {M}onte {C}arlo and gradient descent},
  author = {Dalalyan, Arnak S.},
  booktitle = {Proceedings of the 2017 Conference on Learning Theory},
  pages = {678--689},
  year = {2017},
  editor = {Kale, Satyen and Shamir, Ohad},
  volume = {65},
  series = {Proceedings of Machine Learning Research},
  month = {7},
  publisher = {PMLR},
}

@article{DurmusMoulines2019,
  author = {Durmus, Alain and Moulines, \'{E}ric},
  title = {High-dimensional {B}ayesian inference via the unadjusted
              {L}angevin algorithm},
  journal = {Bernoulli},
  fjournal = {Bernoulli. Official Journal of the Bernoulli Society for
              Mathematical Statistics and Probability},
  volume = {25},
  year = {2019},
  number = {4A},
  pages = {2854--2882},
}

@article{DurmusMajewskiMiasojedow2019,
  author = {Durmus, Alain and Majewski, Szymon and Miasojedow, B\l{}a\.{z}ej},
  title = {Analysis of {L}angevin {M}onte {C}arlo via convex optimization},
  journal = {J. Mach. Learn. Res.},
  fjournal = {Journal of Machine Learning Research (JMLR)},
  volume = {20},
  year = {2019},
  pages = {Paper No. 73, 46},
}

@inproceedings{ChengEtAl2018,
  title = {Underdamped {L}angevin {MCMC}: a non-asymptotic analysis},
  author = {Cheng, Xiang and Chatterji, Niladri S. and Bartlett, Peter L. and Jordan, Michael I.},
  booktitle = {Proceedings of the 31st  Conference on Learning Theory},
  pages = {300--323},
  year = {2018},
  editor = {Bubeck, Sébastien and Perchet, Vianney and Rigollet, Philippe},
  volume = {75},
  series = {Proceedings of Machine Learning Research},
  month = {7},
  publisher = {PMLR},
}

@article{DalalyanRiouDurand2020,
  author = {Dalalyan, Arnak S. and Riou-Durand, Lionel},
  title = {On sampling from a log-concave density using kinetic
              {L}angevin diffusions},
  journal = {Bernoulli},
  fjournal = {Bernoulli. Official Journal of the Bernoulli Society for
              Mathematical Statistics and Probability},
  volume = {26},
  year = {2020},
  number = {3},
  pages = {1956--1988},
}

@incollection{VempalaWibisono2019,
  title = {Rapid convergence of the unadjusted {L}angevin algorithm: isoperimetry suffices},
  author = {Vempala, Santosh S. and Wibisono, Andre},
  booktitle = {Advances in Neural Information Processing Systems 32},
  editor = {H. Wallach and H. Larochelle and A. Beygelzimer and F. d\textquotesingle Alch\'{e}-Buc and E. Fox and R. Garnett},
  pages = {8094--8106},
  year = {2019},
  publisher = {Curran Associates, Inc.},
}

@inproceedings{ZhangEtAl2023,
  title = {Improved discretization analysis for underdamped {L}angevin {M}onte {C}arlo},
  author = {Zhang, Matthew S. and Chewi, Sinho and Li, Mufan (Bill) and Balasubramanian, Krishnakumar and Erdogdu, Murat A.},
  booktitle = {Proceedings of Thirty Sixth Conference on Learning Theory},
  pages = {36--71},
  year = {2023},
  editor = {Neu, Gergely and Rosasco, Lorenzo},
  volume = {195},
  series = {Proceedings of Machine Learning Research},
  month = {7},
  publisher = {PMLR},
}

@article{Zhang2025Midpoint,
  title = {Analysis of {L}angevin midpoint methods using an anticipative {G}irsanov theorem},
  author = {Matthew S. Zhang},
  year = {2025},
  journal = {arXiv preprint 2507.12791},
}

@inproceedings{ShenLee2019,
  author = {Shen, Ruoqi and Lee, Yin Tat},
  booktitle = {Advances in Neural Information Processing Systems},
  editor = {H. Wallach and H. Larochelle and A. Beygelzimer and F. d\textquotesingle Alch\'{e}-Buc and E. Fox and R. Garnett},
  publisher = {Curran Associates, Inc.},
  title = {The randomized midpoint method for log-concave sampling},
  volume = {32},
  year = {2019},
}

@inproceedings{HeBalasubramanianErdogdu2020,
  author = {He, Ye and Balasubramanian, Krishnakumar and Erdogdu, Murat A.},
  booktitle = {Advances in Neural Information Processing Systems},
  editor = {H. Larochelle and M. Ranzato and R. Hadsell and M. F. Balcan and H. Lin},
  pages = {7366--7376},
  publisher = {Curran Associates, Inc.},
  title = {On the ergodicity, bias and asymptotic normality of randomized midpoint sampling method},
  volume = {33},
  year = {2020},
}

@inproceedings{YuKaragulyanDalalyan2024,
  title = {Langevin {M}onte {C}arlo for strongly log-concave distributions: randomized midpoint revisited},
  author = {Lu Yu and Avetik Karagulyan and Arnak S. Dalalyan},
  booktitle = {The Twelfth International Conference on Learning Representations},
  year = {2024},
}

@article{AltschulerChewi2026,
  author = {Jason M. Altschuler and Sinho Chewi},
  title = {Shifted composition {III}: local error framework for {KL} divergence},
  journal = {Found. Comput. Math.},
  fjournal = {Foundations of Computational Mathematics. The Journal of the
              Society for the Foundations of Computational Mathematics},
  year = {2026},
  doi = {10.1007/s10208-026-09753-x},
  url = {https://link.springer.com/article/10.1007/s10208-026-09753-x},
}

@inproceedings{AltschulerChewiZhang2026,
  author = {Altschuler, Jason M. and Chewi, Sinho and Zhang, Matthew S.},
  title = {Shifted composition {IV}: toward ballistic acceleration for log-concave sampling},
  year = {2026},
  publisher = {Association for Computing Machinery},
  address = {New York, NY, USA},
  booktitle = {Proceedings of the 58th Annual ACM Symposium on Theory of Computing},
  pages = {1739--1750},
  numpages = {12},
  venue = {Salt Lake City, UT, USA},
  series = {STOC '26},
}

@article{CaoLuWang2021,
  author = {Cao, Yu and Lu, Jianfeng and Wang, Lihan},
  title = {Complexity of randomized algorithms for underdamped {L}angevin
              dynamics},
  journal = {Commun. Math. Sci.},
  fjournal = {Communications in Mathematical Sciences},
  volume = {19},
  year = {2021},
  number = {7},
  pages = {1827--1853},
}

@inproceedings{KandasamyNagaraj2024,
  author = {Kandasamy, Saravanan and Nagaraj, Dheeraj},
  booktitle = {Advances in Neural Information Processing Systems},
  editor = {A. Globerson and L. Mackey and D. Belgrave and A. Fan and U. Paquet and J. Tomczak and C. Zhang},
  pages = {65972--66024},
  publisher = {Curran Associates, Inc.},
  title = {The {P}oisson midpoint method for {L}angevin dynamics: provably efficient discretization for diffusion models},
  volume = {37},
  year = {2024},
}

@book{ChewiBook,
  author = {Chewi, Sinho},
  title = {Log-concave sampling},
  note = {Final draft, August 20, 2026},
  year = {2026},
  url = {https://chewisinho.github.io/main.pdf},
}

@inproceedings{SrinivasanNagaraj2026,
  author = {Srinivasan, Rishikesh and Nagaraj, Dheeraj},
  booktitle = {International Conference on Learning Representations},
  editor = {C. Vondrick and B. Hariharan and C. Raffel and L. Pinto and D. Yang and A. Faust},
  pages = {78341--78367},
  title = {Poisson midpoint method for log concave sampling: beyond the strong error lower bounds},
  volume = {2026},
  year = {2026},
  url = {https://proceedings.iclr.cc/paper_files/paper/2026/file/7e872599932b6172bec451b5d7bf6fbf-Paper-Conference.pdf},
}

@article{PedrottiWhalley2026,
  author = {Pedrotti, Francesco and Whalley, Peter A.},
  title = {{Wasserstein} mixing time of the unadjusted {Langevin} algorithm},
  date = {2026},
  journal = {arXiv preprint 2608.02430},
}

@book{HelfferNier2005,
  author = {Helffer, B. and Nier, F.},
  title = {Hypoelliptic estimates and spectral theory for {Fokker--Planck} operators and {Witten} laplacians},
  date = {2005},
  series = {Lecture Notes in Mathematics},
  volume = {1862},
  publisher = {Springer},
}

@article {Huguet2021,
    AUTHOR = {Huguet, Baptiste},
     TITLE = {Intertwining relations for diffusions in manifolds and
              applications to functional inequalities},
   JOURNAL = {Stochastic Process. Appl.},
  FJOURNAL = {Stochastic Processes and their Applications},
    VOLUME = {145},
      YEAR = {2022},
     PAGES = {1--28},
}

@book{Simon2005,
  author = {Simon, B.},
  title = {Trace ideals and their applications},
  date = {2005},
  edition = {2},
  series = {Mathematical Surveys and Monographs},
  volume = {120},
  publisher = {American Mathematical Society},
}

@book{Nualart2006,
  author = {Nualart, David},
  title = {The {M}alliavin calculus and related topics},
  series = {Probability and its Applications (New York)},
  edition = {Second},
  publisher = {Springer-Verlag, Berlin},
  year = {2006},
  pages = {xiv+382},
}

@book{BerghLofstrom1976,
  author = {Bergh, J. and L\"ofstr\"om, J.},
  title = {Interpolation spaces: an introduction},
  date = {1976},
  publisher = {Springer},
}

@article{Stein1956,
  author = {Stein, E. M.},
  title = {Interpolation of linear operators},
  date = {1956},
  journaltitle = {Transactions of the American Mathematical Society},
  volume = {83},
  number = {2},
  pages = {482--492},
}

@book{Haase2006,
  author = {Haase, M.},
  title = {The functional calculus for sectorial operators},
  date = {2006},
  publisher = {Birkh\"auser},
  location = {Basel},
}

@book{NourdinPeccati2012,
  author = {Nourdin, I. and Peccati, G.},
  title = {Normal approximations with {Malliavin} calculus: from {Stein}'s method to universality},
  date = {2012},
  publisher = {Cambridge University Press},
}

@article{NourdinPeccatiSwan2014,
  author = {Nourdin, I. and Peccati, G. and Swan, Y.},
  title = {Entropy and the fourth moment phenomenon},
  date = {2014},
  journaltitle = {Journal of Functional Analysis},
  volume = {266},
  number = {5},
  pages = {3170--3207},
}

@article{PaulinWhalley2026,
  title = {Theoretical guarantees for stochastic gradient sampling methods via {G}aussian convolution inequalities},
  author = {Daniel Paulin and Peter A. Whalley},
  year = {2026},
  journal = {arXiv preprint 2604.24632},
}

@article{HigherOrderRHMC,
  title = {Smoothed {P}icard {H}amiltonian {M}onte {C}arlo},
  author = {Fan Chen and Sinho Chewi and Jianfeng Lu and Matthew S. Zhang},
  year = {2026},
  journal = {arXiv preprint 2609.06906},
}

@article{TalayTubaro1990,
  author = {Talay, D. and Tubaro, L.},
  title = {Expansion of the global error for numerical schemes solving stochastic differential equations},
  date = {1990},
  journaltitle = {Stochastic Analysis and Applications},
  volume = {8},
  number = {4},
  pages = {483--509},
}

@article{MattinglyStuartTretyakov2010,
  author = {Mattingly, J. C. and Stuart, A. M. and Tretyakov, M. V.},
  title = {Convergence of numerical time-averaging and stationary measures via {P}oisson equations},
  date = {2010},
  journaltitle = {SIAM Journal on Numerical Analysis},
  volume = {48},
  number = {2},
  pages = {552--577},
}

@article{JohndrowMattingly2017,
  author = {Johndrow, James E. and Mattingly, Jonathan C.},
  title = {Error bounds for approximations of {Markov} chains used in {Bayesian} sampling},
  date = {2017},
  journal = {arXiv preprint 1711.05382},
}

@article{LuYeZhou2025,
  author = {Lu, Jianfeng and Ye, Xuda and Zhou, Zhennan},
  title = {Mean square error analysis of stochastic gradient and variance-reduced sampling algorithms},
  date = {2025},
  journal = {arXiv preprint 2511.04413},
}

@inproceedings{LiEtAl2019,
  author = {Li, Xuechen and Wu, Yi and Mackey, Lester and Erdogdu, Murat A.},
  booktitle = {Advances in Neural Information Processing Systems},
  editor = {H. Wallach and H. Larochelle and A. Beygelzimer and F. d\textquotesingle Alch\'{e}-Buc and E. Fox and R. Garnett},
  publisher = {Curran Associates, Inc.},
  title = {Stochastic {R}unge--{K}utta accelerates {L}angevin {M}onte {C}arlo and beyond},
  volume = {32},
  year = {2019},
}

@article{SanzSernaZygalakis2021,
  author = {Sanz-Serna, Jes\'us Mar\'ia and Zygalakis, Konstantinos C.},
  title = {Wasserstein distance estimates for the distributions of numerical approximations to ergodic stochastic differential equations},
  journal = {Journal of Machine Learning Research},
  year = {2021},
  volume = {22},
  number = {242},
  pages = {1--37},
}

@article{PaulinWhalley2024,
  author = {Paulin, Daniel and Whalley, Peter A.},
  title = {Correction to `{Wasserstein} distance estimates for the distributions of numerical approximations to ergodic stochastic differential equations'},
  date = {2024},
  journaltitle = {Journal of Machine Learning Research},
  volume = {25},
  number = {376},
  pages = {1--9},
}

@article{FosterLyonsOberhauser2021,
  author = {Foster, James and Lyons, Terry and Oberhauser, Harald},
  title = {The shifted {ODE} method for underdamped {L}angevin {MCMC}},
  journal = {arXiv preprint 2101.03446},
  year = {2021},
}

@article{ZhangEtAl2025DDPM,
  title = {Sublinear iterations can suffice even for {DDPMs}},
  author = {Matthew S. Zhang and Stephen Huan and Jerry Huang and Nicholas M. Boffi and Sitan Chen and Sinho Chewi},
  year = {2025},
  journal = {arXiv preprint 2511.04844},
}

@article{DalalyanTsybakov2012,
  author = {Dalalyan, Arnak S. and Tsybakov, Alexandre B.},
  title = {Sparse regression learning by aggregation and {L}angevin {M}onte-{C}arlo},
  journal = {J. Comput. System Sci.},
  fjournal = {Journal of Computer and System Sciences},
  volume = {78},
  year = {2012},
  number = {5},
  pages = {1423--1443},
}

@article{DalalyanKaragulyan2019,
  author = {Dalalyan, Arnak S. and Karagulyan, Avetik},
  title = {User-friendly guarantees for the {L}angevin {M}onte {C}arlo with inaccurate gradient},
  journal = {Stochastic Process. Appl.},
  fjournal = {Stochastic Processes and their Applications},
  volume = {129},
  year = {2019},
  number = {12},
  pages = {5278--5311},
}

@inproceedings{ChengBartlett2018,
  title = {Convergence of {L}angevin {MCMC} in {KL}-divergence},
  author = {Cheng, Xiang and Bartlett, Peter},
  booktitle = {Proceedings of Algorithmic Learning Theory},
  pages = {186--211},
  year = {2018},
  editor = {Janoos, Firdaus and Mohri, Mehryar and Sridharan, Karthik},
  volume = {83},
  series = {Proceedings of Machine Learning Research},
  month = {4},
  publisher = {PMLR},
}

@inproceedings{GaneshTalwar2020,
  author = {Ganesh, Arun and Talwar, Kunal},
  booktitle = {Advances in Neural Information Processing Systems},
  editor = {H. Larochelle and M. Ranzato and R. Hadsell and M. F. Balcan and H. Lin},
  pages = {7222--7233},
  publisher = {Curran Associates, Inc.},
  title = {Faster differentially private samplers via {R}\'{e}nyi divergence analysis of discretized {L}angevin {MCMC}},
  volume = {33},
  year = {2020},
}

@inproceedings{ErdogduHosseinzadehZhang2022,
  title = {Convergence of {L}angevin {M}onte {C}arlo in chi-squared and {R}\'enyi divergence},
  author = {Erdogdu, Murat A. and Hosseinzadeh, Rasa and Zhang, Shunshi},
  booktitle = {Proceedings of the 25th International Conference on Artificial Intelligence and Statistics},
  pages = {8151--8175},
  year = {2022},
  editor = {Camps-Valls, Gustau and Ruiz, Francisco J. R. and Valera, Isabel},
  volume = {151},
  series = {Proceedings of Machine Learning Research},
  month = {3},
  publisher = {PMLR},
}

@article{ChewiEtAl2025LMC,
  author = {Chewi, Sinho and Erdogdu, Murat A. and Li, Mufan and Shen,
              Ruoqi and Zhang, Matthew S.},
  title = {Analysis of {L}angevin {M}onte {C}arlo from {P}oincar\'e{} to log-{S}obolev},
  journal = {Found. Comput. Math.},
  fjournal = {Foundations of Computational Mathematics. The Journal of the
              Society for the Foundations of Computational Mathematics},
  volume = {25},
  year = {2025},
  number = {4},
  pages = {1345--1395},
}

@article{MaEtAl2021,
  author = {Yi-An Ma and Niladri S. Chatterji and Xiang Cheng and Nicolas Flammarion and Peter L. Bartlett and Michael I. Jordan},
  title = {Is there an analog of {N}esterov acceleration for gradient-based {MCMC}?},
  volume = {27},
  journal = {Bernoulli},
  number = {3},
  publisher = {Bernoulli Society for Mathematical Statistics and Probability},
  pages = {1942--1992},
  year = {2021},
}

@article{AltschulerChewi2024,
  author = {Altschuler, Jason M. and Chewi, Sinho},
  title = {Faster high-accuracy log-concave sampling via algorithmic warm starts},
  journal = {J. ACM},
  fjournal = {Journal of the ACM},
  volume = {71},
  year = {2024},
  number = {3},
  pages = {Art. 24, 55},
}

@inproceedings{ChewiEtAl2021MALA,
  title = {Optimal dimension dependence of the {M}etropolis-adjusted {L}angevin algorithm},
  author = {Chewi, Sinho and Lu, Chen and Ahn, Kwangjun and Cheng, Xiang and Le Gouic, Thibaut and Rigollet, Philippe},
  booktitle = {Proceedings of Thirty Fourth Conference on Learning Theory},
  pages = {1260--1300},
  year = {2021},
  editor = {Belkin, Mikhail and Kpotufe, Samory},
  volume = {134},
  series = {Proceedings of Machine Learning Research},
  month = {8},
  publisher = {PMLR},
}

@inproceedings{LeeShenTian2021,
  title = {Structured logconcave sampling with a restricted {G}aussian oracle},
  author = {Lee, Yin Tat and Shen, Ruoqi and Tian, Kevin},
  booktitle = {Proceedings of Thirty Fourth Conference on Learning Theory},
  pages = {2993--3050},
  year = {2021},
  editor = {Belkin, Mikhail and Kpotufe, Samory},
  volume = {134},
  series = {Proceedings of Machine Learning Research},
  month = {8},
  publisher = {PMLR},
}

@inproceedings{ChenChewiSalimWibisono2022,
  title = {Improved analysis for a proximal algorithm for sampling},
  author = {Chen, Yongxin and Chewi, Sinho and Salim, Adil and Wibisono, Andre},
  booktitle = {Proceedings of Thirty Fifth Conference on Learning Theory},
  pages = {2984--3014},
  year = {2022},
  editor = {Loh, Po-Ling and Raginsky, Maxim},
  volume = {178},
  series = {Proceedings of Machine Learning Research},
  month = {7},
  publisher = {PMLR},
}

@inproceedings{ChenEtAl2026FORS,
  title = {High-accuracy sampling for diffusion models and log-concave distributions},
  author = {Chen, Fan and Chewi, Sinho and Daskalakis, Constantinos Costis and Rakhlin, Alexander},
  year = {2026},
  booktitle = {Proceedings of the 43rd International Conference on Machine Learning},
  volume = {306},
  series = {Proceedings of Machine Learning Research},
  pages = {13835--13869},
  publisher = {PMLR},
  url = {https://proceedings.mlr.press/v306/chen26p.html},
}

@article{ChenEtAl2026Exact,
  title = {Exact simulation of diffusions and improved algorithms for log-concave sampling},
  author = {Fan Chen and Sinho Chewi and Alexander Rakhlin and Matthew S. Zhang},
  year = {2026},
  journal = {arXiv preprint 2608.05022},
}

@article{ZhangAltschulerChewi2026,
  title = {Algorithmic warm starts for {H}amiltonian {M}onte {C}arlo},
  author = {Matthew S. Zhang and Jason M. Altschuler and Sinho Chewi},
  year = {2026},
  journal = {arXiv preprint 2603.22741},
}

@article{BouRabeeMarsden2025,
  author = {Bou-Rabee, Nawaf and Marsden, Milo},
  title = {Unadjusted {H}amiltonian {MCMC} with stratified {M}onte {C}arlo time integration},
  journal = {Ann. Appl. Probab.},
  fjournal = {The Annals of Applied Probability},
  volume = {35},
  year = {2025},
  number = {1},
  pages = {360--392},
}

@book{Buckdahn1994,
  author = {Buckdahn, Rainer},
  title = {Anticipative {Girsanov} transformations and {Skorohod} stochastic differential equations},
  date = {1994},
  series = {Memoirs of the American Mathematical Society},
  volume = {111},
  number = {533},
  publisher = {American Mathematical Society},
}

@article{ScottEtAl2026,
  author = {Maximilian Scott and D{\'{a}}ire O'Kane and Andra{\v{z}} Jelin{\v{c}}i{\v{c}} and James Foster},
  title = {Underdamped {L}angevin {MCMC} with third order convergence},
  journal = {Journal of Machine Learning Research},
  year = {2026},
  volume = {27},
  number = {126},
  pages = {1--63},
}

@inproceedings{Wibisono2018,
  title = {Sampling as optimization in the space of measures: the {L}angevin dynamics as a composite optimization problem},
  author = {Wibisono, Andre},
  booktitle = {Proceedings of the 31st Conference on Learning Theory},
  pages = {2093--3027},
  year = {2018},
  editor = {Sébastien Bubeck and Vianney Perchet and Philippe Rigollet},
  volume = {75},
  series = {Proceedings of Machine Learning Research},
  month = {7},
  publisher = {PMLR},
}

@book{AmbrosioGigliSavare2008,
  author = {Ambrosio, L. and Gigli, N. and Savar\'{e}, G.},
  title = {Gradient flows in metric spaces and in the space of probability measures},
  date = {2008},
  edition = {2},
  series = {Lectures in Mathematics ETH Z\"{u}rich},
  publisher = {Birkh\"{a}user},
}

@article{Peyre2018,
  author = {Peyre, R\'{e}mi},
  title = {Comparison between {$W_2$} distance and {$\dot H^{-1}$} norm, and localization of {W}asserstein distance},
  journal = {ESAIM Control Optim. Calc. Var.},
  volume = {24},
  number = {4},
  pages = {1489--1501},
  year = {2018},
  doi = {10.1051/cocv/2017050},
}
\end{document}